\documentclass[twoside, 11pt]{article}
\usepackage{geometry}
\usepackage{times}
\usepackage[utf8]{inputenc}
\usepackage[vietnamese,english]{babel}
\usepackage[displaymath, tightpage]{preview}
\usepackage[all, cmtip]{xy}

\usepackage{amsmath, bm}
\usepackage{amsxtra}
\usepackage{amscd}
\usepackage{amsthm}
\usepackage{amsfonts}
\usepackage{amssymb}
\usepackage{bbm}
\usepackage{mathtools}
\usepackage{XCharter}
\usepackage{sansmath}
\usepackage[T1]{fontenc}
\usepackage[title]{appendix}
\usepackage{enumitem}
\usepackage{authblk}

\usepackage{fancyhdr}
\usepackage{sectsty}
\sectionfont{\fontsize{13}{15}\selectfont}
\usepackage{tikz-cd}
\usetikzlibrary{matrix,calc,arrows}
\usepackage[new]{old-arrows}

\usepackage{hyperref}
\usepackage[scaled=0.8]{helvet} % ss

\newcommand{\BB}{Bia\l ynicki-Birula }

\newcommand{\divisor}{\mathrm{div}}
\newcommand{\del}{\partial}

\newcommand{\End}{\mathrm{End}}
\newcommand{\ext}{\mathrm{Ext}}

\newcommand{\jac}{\mathrm{Jac}}
\newcommand{\ima}{\mathrm{im}}

\newcommand{\Hone}{H^1 (L^2 \Lambda^{-1} ) }
\newcommand{\Hzero}{H^0 (K L^{-2} \Lambda ) }

\newcommand{\spanspace}{\mathrm{Span}}

\newcommand{\prym}{\mathrm{Prym}}
\newcommand{\pr}{\mathrm{pr}}

\newcommand{\bC}{\mathbb{C}}
\newcommand{\bP}{\mathbb{P}}

\newcommand{\SL}{SL_2(\mathbb{C})}

\newcommand{\Res}{\mathrm{Res }}

\newcommand{\tX}{\widetilde{X}}

\newcommand{\SoV}{\mathrm{SoV}}
\newcommand{\tSoV}{\widetilde{\mathrm{SoV}}}

\newcommand{\mM}{\mathcal{M}}
\newcommand{\M}{\mathcal{M}_{\Lambda,d}}

\newcommand{\mN}{\mathcal{N}}
\newcommand{\Nstable}{\mathcal{N}_{\Lambda}}
\newcommand{\N}{\mathcal{N}_{\Lambda,d}}

\newcommand{\pic}{\mathrm{Pic}}

\newcommand{\bp}{{\bm{p} }}

\newcommand{\bq}{{ \bm{q} }}
\newcommand{\bfx}{\mathbf{x}}
\newcommand{\bqcheck}{\bm{\check{q}}}

\newcommand{\by}{\bm{x}}
\newcommand{\bmx}{{ \bm{x} }}
\newcommand{\bmr}{{ \bm{r} }}
\newcommand{\bmk}{{ \bm{k} }}
\newcommand{\bz}{\bm{z}}
\newcommand{\bzcheck}{\bm{\check{z}}}
\newcommand{\bu}{{ \bm{u} }}

\newcommand{\sfx}{{ \mathsf{x} }}
\newcommand{\tbu}{ \bm{\tilde{u}} }
\newcommand{\tu}{\tilde{u}}

\newcommand{\pa}{\partial}

\newcommand{\al}{\alpha}

\newcommand{\be}{\beta}

\newcommand{\ep}{\epsilon}
\newcommand{\bkap}{\pmb{\kappa}}
\newcommand{\la}{\lambda}
\newcommand{\bla}{\pmb{\lambda}}

\newcommand{\wom}{\widetilde{\omega}}

\newcommand{\si}{\sigma}

\newcommand{\vf}{\varphi}

\newcommand{\CH}{{\mathcal H}}

\newcommand{\CM}{{\mathcal M}}
\newcommand{\CN}{{\mathcal N}}
\newcommand{\CO}{{\mathcal O}}

\newcommand{\rf}[1]{(\ref{#1})}

\numberwithin{equation}{section}

\newtheoremstyle{theorem}% name
{3mm}% space above
{1mm}% space below
{\normalfont\itshape}% body font
{}% indent amount
{\normalfont\scshape}% theorem head font
{.}% punctuation after theorem head
{.5em}% space after theorem head
{\thmname{#1}\thmnumber{ #2}\thmnote{ (#3)}}% theorem head spec
\theoremstyle{theorem}
\newtheorem{theorem}{Theorem}
\newtheorem{lemma}[theorem]{Lemma}
\newtheorem{proposition}[theorem]{Proposition}
\newtheorem{corollary}[theorem]{Corollary}

\numberwithin{theorem}{section}

\newtheoremstyle{definition}% name
{3mm}% space above
{1mm}% space below
{\normalfont\normalfont}% body font
{}% indent amount
{\normalfont\scshape}% theorem head font
{.}% punctuation after theorem head
{.5em}% space after theorem head
{\thmname{#1}\thmnumber{ #2}\thmnote{ (#3)}}% theorem head spec
\theoremstyle{definition}
\newtheorem{remark}[theorem]{Remark}
\newtheorem{definition}[theorem]{Definition}
\newtheorem{example}[theorem]{Example}
\newcommand\dqd{\foreignlanguage{vietnamese}{Đinh Quý Dương}}
\begin{document}\thispagestyle{empty}
\title{Classical limit of the geometric Langlands correspondence for $\bm{SL( }\pmb{2, \bC)}$}
\author[a, b]{\dqd
\footnote{Present address: Department of Mathematics, University of Pennsylvania, 
209 South 33rd Street Philadelphia, PA 19104-6395, USA. 
\\ \vspace{5pt}
Email addresses: \textit{duongdinh.mp@gmail.com} (Đinh D.)
and \textit{joerg.teschner@desy.de} (J. Teschner)}}
\author[b ,c]{J\"org Teschner}
\affil[a]{Max Planck Institute for Mathematics, Vivatsgasse 7,
53111 Bonn, Germany}
\affil[b]{Department of Mathematics, University of Hamburg, 
Bundesstraße 55, 20146 Hamburg, Germany}
\affil[c]{Deutsches Elektronen-Synchrotron DESY, Notkestr. 85, 22607 Hamburg, Germany}
\date{}
\maketitle

\maketitle
\begin{quote}
\begin{center}{\bf Abstract}\end{center}
{\small The goal of this paper is to give an explicit description of the integrable structure of the Hitchin moduli spaces. This is done by introducing explicit parameterisations for the different strata of the Hitchin moduli spaces, and by adapting the
Separation of Variables method 
from the theory of integrable models
to the Hitchin moduli spaces. 
The resulting description 
exhibits a clear analogy with 
Drinfeld's
first construction of the 
geometric Langlands correspondence. 
It can be 
seen as a classical limit of 
a version of Drinfeld's construction 
which is adapted to the complex number field.
}
\end{quote}
\tableofcontents

\section{Introduction}

Our goal is to give a more explicit description of the integrable structure of the Hitchin moduli spaces 
which clearly exhibits its relation with the classical limit of the geometric Langlands correspondence
pioneered by Drinfeld \cite{Drinfeld, Fre95}. 

The Hitchin moduli spaces considered in this paper are the moduli spaces $\CM_{H}(\Lambda)$ of 
stable $\SL$-Higgs bundles on Riemann surfaces $X$ of genus $g > 1$ introduced in \cite{Hit87a}.
An $\SL$-Higgs bundle in $\CM_{H}(\Lambda)$ is a pair $(E, \phi)$
where $E$ is a holomorphic rank-2 bundle with fixed determinant $\Lambda$, 
and $\phi \in H^0( \End_0(E) \otimes K)$ 
%is a trace-free holomorphic endomorphism of $E$ twisted by holomorphic one-forms. 
%Such endomorphisms are 
is called a Higgs field. 
Higgs bundles 
$(E,\phi)$ with $E$ being a stable bundle define an open dense subset
of $\CM_{H}(\Lambda)$ isomorphic to $T^\ast\CN_{\Lambda}$, 
where $\CN_{\Lambda}$ is the moduli space of stable rank-$2$ bundles with fixed
determinant $\Lambda$.

%%%%%%%%%%%%%%%%%%%%%%%%%%%%%%%%%%%%%%%%%%%%%%%%%%
\subsection{Integrable structure of Hitchin's moduli spaces}
\label{Hit-int}

Of particular interest for many applications are the integrable structures 
of $\CM_{H}(\Lambda)$ \cite{Hit87a, Hur}. The restriction of the 
symplectic structure of $\CM_{H}(\Lambda)$ to the open dense 
subset $T^\ast\CN_{\Lambda}$ is  identical with the canonical 
cotangent bundle symplectic structure. The integrable structure of $\CM_{H}(\Lambda)$
is characterized by the Hitchin map sending $(E,\phi)$ to the quadratic 
differential $q = \det(\phi)$. 
One may identify the generic fibers of the Hitchin map with
the Prym varieties of the spectral curves $S\equiv S_q$ locally defined by the equation $v^2=q(u)$.
The {symplectomorphism} between $\CM_{H}(\Lambda)$ and the torus fibration defined by
the Hitchin map, equipped with its natural symplectic structure, characterises the integrable 
structure of the Hitchin moduli spaces \cite{Hur, Witten}.
The image of $T^\ast\CN_\Lambda$ forms an open dense subset of  the torus fibration,
and hence the Hitchin moduli space may be regarded as a partial compactification 
of $T^\ast\CN_\Lambda$ \cite{Hit87b}.

One of our goals is to make the integrable structure of $\CM_{H}(\Lambda)$ more 
explicit by introducing an intermediate step between
$\CM_{H}(\Lambda)$ and the torus fibration. To this aim  observe that 
any holomorphic rank-2 bundle $E$ with determinant $\Lambda$ can be realised
by an extension  of the form 
\begin{equation} \label{E-ext}
		0 \rightarrow L \rightarrow E \rightarrow L^{-1}\Lambda \rightarrow 0.
	\end{equation}
This means that $L$ is a sub-line bundle of $E$.
%and with respect to local frames adapted to $L$ transition functions of $E$ are all upper-triangular and define a 1-cocycle in $L^2 \Lambda^{-1}$.
The key observation is that if a Higgs field $\phi$ on $E$ takes the form $
\big(\begin{smallmatrix} \phi_0 & \phantom{-}\phi_- \\ \phi_+ & -\phi_0 \end{smallmatrix}\big)$ 
with respect to local frames adapted to the embedding $L \hookrightarrow E$,
then the local functions $\phi_+$ glue into a section of $KL^{-2}\Lambda$. 
Then the divisor 
\begin{align*}
    &\divisor(\phi_+) = \sum_{i=1}^m u_i,
    &m= 2g-2 + \deg(\Lambda) - 2 \deg(L),
\end{align*}
of zeroes of $\phi_+$, 
supplemented with the values
$v_i=\phi_0(u_i)$, determines\footnote{In this paper we will frequently write $\omega(x)$ and $q(x)$
for the evaluation of a differential $\omega$ or quadratic differential $q$
at a point $x \in X$ w.r.t. some local coordinates of $x$.
Most of the time it will be clear in which coordinate the evaluation is made. 
Here, although $v_i$ depends on the local coordinates chosen to evaluate $\phi_0$,
it transforms according to how the fiber coordinate of $T^\ast X$ transforms w.r.t to change of coordinates, 
and hence unambiguously defines a point $\tu_i \in T^\ast_{u_i} X$.}
an effective divisor
$\tbu = \sum_{i=1}^m \tu_i$ on the spectral curve $S$. 
We will refer to the pair $(S,\tbu)$ as the
 Baker-Akhiezer (BA) data of $(E,\phi)$. 
Note that the BA-data $(S,\tbu)$ are equivalent to the pairs $(\bm{w}, q)$,
where $\bm{w}$ is the unordered collection of $m$ points $\tu_i \in T^\ast_{u_i} X$ determined by $(u_i,v_i)$, 
and $q$ is a quadratic differential satisfying $q(u_i)=v_i^2$.
The collection $\bm{w}$ defines a point in the symmetric product $(T^\ast X)^{[m]}$.
We will see that the BA-data provide a useful 
characterisation of the integrable structure of the Hitchin moduli spaces, being related to the Prym variety by the Abel map \cite{Witten}.
All this is already understood on a somewhat abstract level. Our goal is to  describe these structures more explicitly.

%%%%%%%%%%%%%%%%%%%%%%%%%%%%%%%%%%%%%%%%%%%%%%%%%%%%%%%%%%%%%%%%%%%%%%
%%%%%%%%%%%%%%%%%%%%%%%%%%%%%%%%%%%%%%%%%%%%%%%%%%%%%%%%%%%%%%%%%%%%%%
%%%%%%%%%%%%%%%%%%%%%%%%%%%%%%%%%%%%%%%%%%%%%%%%%%%%%%%%%%%%%%%%%%%%%%
\subsection{Cotangent spaces of moduli of pairs (subbundles, bundles)}

In order to describe the map from $\CM_{H}(\Lambda)$ to BA data $(S, \tbu)$ as explicitly as possible, 
we will need to start from a suitable 
description of the moduli spaces 
$\CM_{H}(\Lambda)$. 
Realizing the rank-2 bundles in terms of extensions (\ref{E-ext}) 
turns out to be useful in this regard. 
An extension class of the form 
(\ref{E-ext}) is a pair $(L, \bfx )$,
with $L$ being a line bundle 
and $\bfx\in H^1(L^2\Lambda^{-1})$. 
To each such extension class $(L, \bfx)$ 
one may associate an isomorphism class of holomorphic rank-2 vector bundle.
A priori, there may be different pairs $(L,\bfx)$, $(L',\bfx')$ realizing the same such isomorphism class.
This holds, in particular, if $L=L'$ and $\bfx'$ is obtained from $\bfx$ by scaling. 

Taking these subtleties into account, it will be
useful to consider the moduli spaces
$\CM_{\Lambda,d}$ of pairs $(L,\bfx)$, with 
$L$ being a line bundle of
degree $d$ and $\bfx \in H^1(L^2\Lambda^{-1})$ as auxiliary objects.
The map that projects to the isomorphism class of the line bundle 
makes $\CM_{\Lambda,d}$ a vector bundle over $\mathrm{Pic}^d$. 
Closely related is the bundle $\CN_{\Lambda,d}$ over $\pic^d$ of projective spaces  
obtained by identifying nonzero elements 
$\bfx,\bfx'\in H^1(L^2\Lambda^{-1})$
related by scaling. 
The map from $\N$ to $\Nstable$
that forgets the subbundle and remembers only the isomorphism classes of rank-2 bundles is a rational map.
For $s_d \coloneqq \deg(\Lambda) - 2d = g-1$, 
this map restricts to an unbranched $2^g:1$ covering over 
the open dense subset of $\Nstable$ defined by very stable bundles, 
which are those that do not admit nonzero nilpotent Higgs fields.
As a result, if $[\sfx] \in \N$ 
projects\footnote{We will frequently use the same notation for
an object and the point it defines in the corresponding moduli space.
For example, we write $E\in \Nstable$, $L \in \pic^d$, and so on.
In addition, we will frequently identify the space of Higgs fields on a stable bundle $E$ with the cotangent fiber of $\Nstable$ at $E$.
}
to a very stable bundle $E \in \Nstable$,
the space of traceless Higgs fields on $E$ 
is canonically isomorphic to the cotangent fiber of $\N$ at $[\sfx]$.
This canonical isomorphism is defined by pulling back the covering map.

It will be important for our goals to understand  the pull-back 
$T^\ast_E \hspace{1pt} \Nstable \rightarrow T^\ast_{[\sfx]} \N$ more explicitly
in the cases $0 < s_d \leq g-1$.
To this end, with $N = g-1 + s_d$,
we are going to define local coordinates $\bla = (\lambda_1, \dots, \lambda_g)$ and $\bmx = (x_1, \dots, x_N)$ on $\M$ 
providing local coordinates on $\pic^d$ and the fibers over it respectively.
These coordinates together with their respective canonical conjugates $\bkap = (\kappa_1, \dots, \kappa_g)$ and 
$\bmk = (k_1, \dots, k_N)$
are local Darboux coordinates on $T^\ast \M$.
For $\sfx = (L, \bfx) \in \M$ which can be represented by  a stable bundle $E$ with subbundle $L$,
we may compare the following spaces
\begin{itemize}
\item the cotangent fiber $T^\ast_\sfx \M$ which can be splits 
$$T^\ast_L \pic^d \oplus T^\ast_\bfx \Hone \simeq H^0(K) \oplus H^0(KL^{-2} \Lambda)$$ 
by the choice of local coordinates; 
%we can write $\xi \in T^\ast_\sfx \M$ as $\kappa(\xi) \oplus k_L(\xi)$, 
%where $\kappa(\xi)$ is a holomorphic differential and $k_L(\xi)$ is a section of $KL^{-2} \Lambda$; 
%%%%%
\item the cotangent fiber $T^\ast_{[\sfx]} \N$ which can be split as  
$$T^\ast_L \pic^d \oplus T^\ast_{[\bfx]} \bP\Hone 
\simeq H^0(K) \oplus \ker(\bfx)$$ 
where $\ker(\bfx)$ is the hyperplane in $\Hzero$
that pairs with $\bfx$ to zero via the Serre duality;
%%%%%%
\item the cotangent fiber $T^\ast_E \Nstable \simeq H^0(\End_0(E) \otimes K)$. 
\end{itemize}
We will assign to a Higgs field $\phi$ on $E$ 
a section of $KL^{-2} \Lambda$ which is contained in $\ker(\bfx)$
and an abelian differential
which is determined by a holomorphic differential.
Together with the splitting of $T^\ast_{[\sfx]} \N$,
this is how  we will explicitly characterize the map $T^\ast_E \hspace{1pt} \Nstable \rightarrow T^\ast_{[\sfx]} \N$ in section 3.
To this aim it will be useful to express Higgs fields in terms of abelian differentials.

Bundles admitting non-zero nilpotent Higgs fields play a special role. Such bundles are often called
wobbly, following reference \cite{DP09} which has emphasised the importance of the wobbly bundles for
the approach to the geometric Langlands correspondence outlined in this reference. 
If $E$ admits more nilpotent Higgs fields 
having $L$ as kernel compared to generic bundles, 
the dimension of the kernel of 
$T^\ast_E \hspace{1pt} \Nstable \rightarrow T^\ast_{[\sfx]} \N$ will be larger than generically.
In particular, for $s_d = g-1$, this map is not an isomorphism 
if $E$ is wobbly. The coordinates introduced in our paper will 
in particular allow us to describe the wobbly loci 
more explicitly. This is expected to be useful for future applications.

%%%%%%%%%%%%%%%%%%%%%%%%%%%%%%%%%%%%%%%%%%%%
\subsection{The Separation of Variables maps}

To each pair $(E,\phi)$
we may assign a pair $(S,\tbu)$ of BA-data by the 
construction outlined in 
Section \ref{Hit-int}. 
It will furthermore be possible to 
reconstruct the classes of 
pairs $(E, \phi)$ from a given set of BA-data. 
The resulting correspondence between 
pairs $(E,\phi)$ and BA-data can be 
described by an explicit set of 
equations.

We are going to demonstrate that the symplectic and integrable structures of $\CM_{H}(\Lambda)$
admit useful descriptions in terms of BA-data. 
Recall that these data are equivalent to pairs $(\bm{w},q)$,
where $\bm{w}$ is the unordered collection of $m = 2g-2 +s_d$ points $\tu_i = (u_i,v_i)$ in $T^\ast X$, 
and $q$ is a quadratic differential 
satisfying $q(u_i)=v_i^2$ for $i=1,\dots,m$.
The smooth part of the symmetric product $(T^\ast X)^{[m]}$ has a canonical symplectic structure induced from that of $T^\ast X$. 
The following is the main result of this paper, which is proved in Section 5.

\begin{theorem}\label{intro-main-thm}
\begin{itemize}
\item[(i)] The construction of BA-divisors defines a rational map 
\begin{equation*}
\begin{tikzcd}
    &\SoV: T^\ast\CN_{\Lambda,d} \arrow[dashed]{r} &(T^\ast X)^{[m]},
    \qquad m = 2g-2 +\deg(\Lambda) - 2d.
\end{tikzcd}
\end{equation*} 
\item[(ii)] $\SoV$ is generically étale,
i.e. its derivative is an isomorphism, with generic $2^{2g}:1$ fibers.
\item[(iii)] The restriction of $\SoV$ to the neighborhood of a generic point in 
$T^\ast\CN_{\Lambda,d}$ is a symplectomorphism.
\end{itemize}
\end{theorem}

We may furthermore note that the eigen-line bundle of a Higgs bundle
can be expressed in terms of its BA-divisors.
Hence the point in the Prym variety defined by this Higgs bundle can also be 
expressed in terms of BA-divisors. 
This indicates how the BA-data can be used to capture
the integrable structure of the Hitchin moduli spaces. 
In this way we will see how both symplectic
and integrable structures of $\CM_{H}(\Lambda)$ can be described rather 
explicitly in terms of the BA-data. 

We have denoted the map in the main theorem \ref{intro-main-thm} 
Separation of Variables (SoV), 
indicating an analogy to the theory of exactly integrable models
\cite{Skl89}. 
Variants of the SoV for Hitchin systems have been previously studied in
\cite{GNR,Kri}
and \cite{ER-unpublished}\footnote{See also \cite{ER96,EFR98,ER02}}. However, the methods 
and scope of these approaches
differ considerably in detail.
In addition to yielding a rather detailed description of the relevant integrable and symplectic geometric structures,
our approach captures  the natural stratification on the Hitchin moduli spaces.
How the degeneration of the BA-data and their analogue for holomorphic connections 
encode a limiting process to lower strata %in the respective moduli spaces
will be discussed in details in forthcoming papers \cite{D24, DFT}.

%We will finally discuss how the relations between 
%the strata $\CM_{H}^l(\Lambda)$ have a fairly simple 
%description in terms of the  BA-data. 

%%%%%%%%%%%%%%%%%%%%%%%%%%%%%%%%%%%%%%%%%%%%
\subsection{Relation to geometric and analytic Langlands correspondence}\label{GA-Langl}

We plan to use the results of this paper in subsequent 
publications to develop more explicit descriptions of 
some instances of the geometric Langlands correspondence, and 
especially of the more recently proposed 
analytic Langlands correspondence \cite{T18, EFK}. 
The results of our paper are in particular related to the approaches to the 
geometric Langlands correspondence introduced in the work of Drinfeld-Laumon \cite{Drinfeld, Lau95} 
and Beilinson-Drinfeld
\cite{BD92}. 

Let $G$ be a simple and simply-connected complex reductive group, 
$\check{G}$ its Langlands dual and $X$ a Riemann surface of genus $g \geq 2$. 
The geometric Langlands correspondence aims to associate 
$\mathcal{D}$-modules $H_\sigma$ 
on the moduli stack $\mathrm{Bun}(X, G)$ of $G$-bundles on $X$ 
to  $\check{G}$-local systems 
$\sigma$ on $X$, such that the $\mathcal{D}$-module $H_\sigma$ 
satisfies an eigenvalue property with respect to the natural 
action of Hecke operators. Referring to 
the $\mathcal{D}$-module $H_\sigma$ as a \textit{Hecke eigensheaf} allows us to represent
the geometric Langlands correspondence schematically as
\begin{equation}\label{GLC}
\left\{\begin{array}{l}
		\check{G}-\text{local systems $\sigma$ on } X
\end{array}\right\}
\overset{\mathrm{GLC}}{\longrightarrow}
\left\{\begin{array}{l}
		\text{Hecke eigensheaves $H_\sigma$ on } \mathrm{Bun}(X, G) 
\end{array} \right\}.
\end{equation}
In the cases where 
$\sigma$ belongs to a distinguished class of $\check{G}$-local systems called \textit{opers},
it was proved by Beilinson and Drinfeld that 
$H_\sigma$ corresponds to the $\mathcal{D}$-module 
defined by the system of eigenvalue equations 
for the \textit{quantum Hitchin Hamiltonians} 
$\mathsf{H}_{i = 1, \dots, \dim \mathrm{Bun}(X,G)}$, 
with the eigenvalues defined by $\sigma$.

The analytic Langlands correspondence \cite{EFK, EFK-2} shifts
the attention to the spectral problem 
of quantum Hitchin Hamiltonians $\mathsf{H}_i$, defined by 
considering simultaneous eigenstates of both $\mathsf{H}_i$
and their anti-holomorphic conjugates $\overline{\mathsf{H}}_i$ 
that are 
\begin{itemize}
\item[i)] single-valued  sections of the 
the density line bundle $\big| K^{}_{\mathrm{Bun(X,G)}}\big|$, 
\item[ii)] square-integrable with respect to the natural scalar product on 
sections of this line bundle, and 
\item[iii)] eigenfunctions of the Hecke operators acting on this Hilbert space.
\end{itemize}
The analytic Langlands correspondence \cite{EFK, EFK-2} generalises and 
strenghtens a proposal made in \cite{T18}\footnote{Reference \cite{T18} did not consider the conditions ii) and iii). The existing results support the conjecture that 
all single-valued eigenstates satisfy ii) and iii). If so, the conjecture of \cite{T18} would essentially imply
the analytic Langlands correspondence.} to  a correspondence between 
{\it real} $\check{G}$-opers, opers having monodromy representation in a split real form of $\check{G}$,
and solutions to conditions i)-iii) above, 
schematically
%In some sense, the passage from the geometric Langlands correspondence 
%to the analytic Langlands correspondence constitutes a shift 
%from Hecke eigensheaves to \textit{Hecke eigenfunctions}.
\begin{equation}\label{analytic-Langlands}
\left\{\begin{array}{l}
\text{real } \check{G}-\text{opers $\sigma$ on } X
\end{array}\right\}
\overset{\mathrm{ALC}}{\longleftrightarrow} 
\left\{\begin{array}{l}
\text{single-valued, square-integrable sections}\\
\text{of $\big| K^{}_{\mathrm{Bun}_G(X)} \big|$  } \text{that are eigenfunctions of }\\
\text{$\mathsf{H}_i$, and eigenstates of Hecke operators}
\end{array} \right\}.
\end{equation}
For $G = SL_2(\mathbb{C})$ one may hope to prove the analytic Langlands correspondence, and 
to describe it in more detail, by suitably modifying Drinfeld's approach to the geometric Langlands
correspondence \cite{Drinfeld}. In order to compare Drinfeld's approach with 
the one used in this paper, let us summarise the latter in the following diagram, 
\begin{equation}\label{zigzag}
\centering
\begin{tikzcd}
& &\CN_{\Lambda,d} \arrow[dashed]{dl}[swap]{i} \arrow{dr}{j}
& &\CN_{\Lambda,d}^\vee  \arrow{dl}{j^\vee} \arrow{r}{\text{SoV}}
&X^{[m]} \\
&\mN_\Lambda & 
&\mathrm{Pic}^d
& &
\end{tikzcd}
\end{equation}
%The cases we are interested in are $0 < s_d \leq g-1$.
Given an extension of the form \eqref{E-ext},
the rational map $i$ forgets the choice of the subbundle $L$,
while the map $j$ remembers only $L$.
The map $j^\vee$ is the fibration dual to $j$,
the fiber of which over $L \in \pic^d$ is the space $\bP \Hzero$. 
The map $\text{SoV}:\N^\vee \rightarrow X^{[m]}$ sending elements of $\Hzero$ to the divisor $\bm{u}$ of 
zeros of a representative $c\in \Hzero$
is a local isomorphism.
This summarises the definition of the local {symplectomorphism} between open dense subsets of 
the cotangent spaces of the top row in  (\ref{zigzag}) discussed in this paper.

A similar diagram appeared in Frenkel's paper \cite{Fre95}, 
which interprets Drinfeld's approach to constructing Hecke eigensheaves 
\cite{Drinfeld} as a variant of the quantum {Separation of Variables} method proposed by Sklyanin. 
The main ideas, adapted to the case $G = SL_2(\mathbb{C})$
and $\check{G} = PSL_2(\mathbb{C})$, 
can be outlined as follows.
Starting from a $PSL_2(\mathbb{C})$-local system on $X$, one can use the map $\text{SoV}$ to define a
local system on $\CN_{\Lambda,d}^\vee$.
As $j$ and $j^\vee$ are dual projective fibrations, we can then 
apply the Radon transform to yield a local system on $\CN_{\Lambda,d}$.
It remains to show that this local system is the pull-back of a local system
on $\mathcal{N}_\Lambda$  
which encodes the sought-after Hecke eigensheaf corresponding 
to the $PSL_2(\mathbb{C})$-local system we started with. 
%These ideas was later consolidated by Laumon \cite{Lau95}.
%Frenkel \cite{Fre95} also pointed out the relations between Drinfeld's 
%approach and the \textit{Separation of Variable} method discovered by Sklyanin 
%in the quantum Gaudin model setting \cite{Skl89}.

%In fact, in the case of $\check{G}$-opers or equivalently 
%Hecke eigensheaves defined by the quantum Hitchin Hamiltonians,
%Drinfeld's and Beilinson-Drinfeld's approaches were 
%shown to be compatible by Gaitsgory \cite{Gai15}.

The results of this paper can be used to develop a variant of Drinfeld's approach 
leading to the definition of an integral transformation which
realizes the analytic Langlands correspondence more explicitly. To this aim one may
note that 
the Darboux variables $(\bm{\lambda},\bm{\kappa},\bm{x},\bm{k})$ used in 
this paper offer a useful starting point for the quantisation of the Hitchin system. 
It is natural to define
a quantization $\bm{\phi}$ of the Higgs field $\phi$ by
replacing the coordinates $\kappa_{i}$ and
$k_r$ by the derivatives $\frac{\hbar}{\mathrm{i}}\frac{\pa}{\pa \lambda_{i}}$ and $\frac{\hbar}{\mathrm{i}}\frac{\pa}{\pa x_{r}}$, realizing the quantized Higgs field as a matrix \[ \bm{\phi}(u)=\bigg(\begin{matrix} \bm{a}(u) & \bm{b}(u) \\ \bm{c}(u) & -\bm{a}(u)\end{matrix}\bigg) \] of first-order
differential operators, and the quadratic
differential $q(z)= \det(\phi(z))$
gets replaced by a second-order differential operator
$\mathsf{T}(z)$, a generating function for the quantized
analogs of Hitchin's Hamiltonians. By performing a Radon transformation diagonalising the differential operator 
$\mathsf{c}(u)$,
and changing variables from the residues of the eigenvalues ${c}(u)$ of $\mathsf{c}(u)$ to the zeros
$u_i$, $i=1,\dots,m$,
of ${c}(u)$, one can define a unitary operator mapping the eigenvalue equations for the quantised 
Hitchin Hamiltonians to the oper equations $(\hbar^2{\partial_{u_i}^2}+t(u_i))\Psi=0$, $i=1,\dots,m$.
We plan to present the details in a forthcoming publication. 
In this way one may, in particular, recognise our results 
as a classical version of Drinfeld's construction. 

%%%%%%%%%%%%%%%%%%%%%%%%%%%%%%%%%%%%%%%%%%%%%%
\subsection{Further generalisations}

We furthermore plan to apply the framework introduced
in this paper to two further generalisations of the
Hitchin integrable system. 

The first is obtained by ``deforming" the Higgs fields $\phi$ into $\lambda$-connections\footnote{The deformation parameter $\lambda$ used here should not be confused with the 
coordinates $\la_{i}$ introduced in \eqref{coordinate-lambda} 
and used elsewhere in this paper.}
$\nabla_{\lambda}=\lambda\pa_z+A(z)$, with $A(z)=\phi(z)+\CO(\lambda)$.
The pairs $(E,\nabla_{\lambda})$ form a moduli space
of considerable interest for various applications. 
These moduli spaces are fibered over the moduli space 
of complex structures on Riemann surfaces $X$.
It is natural to consider the representations 
of $\pi_1(X)$ defined by the holonomy of $\nabla_{\lambda}$, 
leading to the study of isomonodromic deformations, 
deformations of $\nabla_{\lambda}$
induced by deformations of the complex structure of $X$
keeping the holonomy constant. These mathematical problems 
have numerous applications in mathematics and mathematical physics. 
It turns out that our framework can be adapted to the case of
$\lambda$-connections, rendering some aspects of the
theory of isomonodromic deformations more explicit. 
This includes explicit characterisation
of $\lambda$-connections in terms of apparent singularities of projective connections,
which are natural analogues of BA-divisors. 
More details are reported in \cite{D24} and the forthcoming paper \cite{DFT}.

One may furthermore combine quantisation with 
deformation parameter $\hbar$, on the one hand, and $\lambda$-deformation, on the other hand.
The equations defining eigenfunctions of Hitchin's 
Hamiltonians may then be naturally replaced by the 
Knizhnik-Zamolodchikov-Bernard (KZB) equations from conformal 
field theory. Single-valued solutions to the KZB equations
are expected to represent correlation functions of 
a conformal field theory called the $H_3^+$-WZNW model.
There exists an integral transformation generalising the 
quantisation of the SoV-transformation 
intertwining the KZB-equations and the partial 
differential equations
named after Belavin, Polyakov and Zamolodchikov (BPZ). 
Single-valued solutions to the BPZ-equations are
provided by the rigorous construction of 
Liouville theory \cite{KRV}. The $\la$-deformed SoV-transformation
converts the BPZ-solutions into single-valued KZB-solutions, providing a construction of correlation 
functions of the $H_3^+$-WZNW model. This transformation
admits an interpretation as a quantum deformation 
of the analytic Langlands correspondence \cite{GT}.

The results of this 
paper are important geometric background for all these future
applications. 

\vspace{20pt}
\paragraph{Acknowledgements.} 
The authors would like to thank 
Troy Figiel and Vladimir Roubtsov for helpful discussion.
In particular, the authors thank Vladimir Roubtsov for communicating
his unpublished work \cite{ER-unpublished} 
which investigates a variant of the SoV studied here. 
This work was supported by the Deutsche Forschungsgemeinschaft (DFG, German Research Foundation) under Germany’s Excellence Strategy – EXC 2121
“Quantum Universe”
granted to the University of Hamburg and DESY.
\dqd \hspace{0pt}
would also like to thank the Max Planck Institute for Mathematics in Bonn
for hospitality and funding.

%%%%%%%%%%%%%%%%%%%%%%%%%%%%%%%%%%%%%%%%%%%%%%%%%%%%%%%%%%%%%%%%%%%%%%
%%%%%%%%%%%%%%%%%%%%%%%%%%%%%%%%%%%%%%%%%%%%%%%%%%%%%%%%%%%%%%%%%%%%%%
%%%%%%%%%%%%%%%%%%%%%%%%%%%%%%%%%%%%%%%%%%%%%%%%%%%%%%%%%%%%%%%%%%%%%%
\section{Moduli spaces}
We start this section by reviewing the necessary background of 
the moduli space $\Nstable$ of stable rank-2 bundles with fixed determinant $\Lambda$,
before reviewing the moduli spaces $\M$ of extension classes 
and $\N$ of pairs (subbundle, bundle).
Of particular importance is our introduction of local coordinates on $\M$
upon choosing reference divisors and local coordinates on $X$.
We then review the Hitchin moduli spaces of stable Higgs bundles,
and end the section by analysing how Higgs fields can be represented in terms of abelian differentials.

\subsection{Moduli spaces of stable bundles}
A rank-2 bundle $E$ is called stable
if all of its subbundle has degree $< \deg(E)/2$.
The moduli space $\mN_\Lambda$ of all such stable bundles with fixed determinant $\Lambda$
is a complex manifold of dimension $3g-3$ 
\cite{Mumford-62, Thaddeus}.
Tensoring with a line bundle $N$ induces an isomorphism $\mN_\Lambda \simeq \mN_{\Lambda N^2}$
and hence $\mN_\Lambda$ depends only on the degree of $\Lambda$.
For simplicity in this paper we will let $\Lambda$ be either $\mathcal{O}_X$ 
or $\mathcal{O}_X(\check{q}_0)$
for some fixed $\check{q}_0 \in X$. 

%%%%%%%%%%%%%%%%%%%%%%%%%%%%%%
\paragraph{Segre stratification.}
For a bundle $E$
with $\det(E) = \Lambda$,
its Segre invariant is $s(E) = \underset{L \subset E}{\min}\hspace{1pt} (\deg(\Lambda) - 2\deg(L))$.
In other words, the Segre invariant of a rank-2 bundle
is determined by its maximal subbundle, i.e. a subbundle of maximal degree.
Note that $s(E) \equiv \deg(\Lambda)$ mod $2$
and $E$ is stable if and only if $s(E) > 0$.
A theorem by Nagata \cite{Nagata} shows that $s(E) \leq g$.
Hence the lower and upper bounds of the Segre invariants of stable bundles
respectively are
\begin{align*}
&s_{min} = 
\begin{cases}
    2 &\text{ for } \deg(\Lambda) = 0 \\
    1 &\text{ for } \deg(\Lambda) = 1 
\end{cases},
&s_{max} = 
\begin{cases}
    g &\text{ for } \deg(\Lambda) \equiv g \text{ mod } 2 \\
    g-1 &\text{ for } \deg(\Lambda) \equiv g-1 \text{ mod } 2 
\end{cases}.
\end{align*}
We say that bundles with Segre invariant $s_{max}$ are maximally stable.
Note that a generic point $[E] \in \mN_\Lambda$ is the isomorphism class 
of a maximally stable bundle. 

The Segre stratification on $\mN_\Lambda$ is defined by the Segre invariants, namely
$$ 
 N_{\Lambda, s_{min}} \subset 
 N_{\Lambda, s_{min} + 2} \subset \dots 
 \subset \mN_{\Lambda, s_{max} -2} 
 \subset \mN_{\Lambda, s_{max}} = \mN_{\Lambda}
$$
where
\begin{align*}
    \mN_{\Lambda,s} = \{ [E] \in \mN_\Lambda 
    \mid 0 < s(E) \leq s \}
\end{align*} 
is a locally closed subset of $\mN_\Lambda$
of dimension $2g+s-2$ for $s < s_{max}$ \cite{Lange-Narasimhan}.

%%%%%%%%%%%%%%%%%%%%%%%%%%%%%%%%%%%
\paragraph{Secant varieties.}
As Segre strata is defined by the maximal subbundles,
it is natural to investigate them by realizing rank-2 bundles in terms of extension of line bundles.
This is in fact how Lange-Narasimhan studied Segre stratification \cite{Lange-Narasimhan}.
Their strategy relies on the fact that the Segre stratification 
manifests in the moduli space of extension classes in the form of secant varieties, as we now briefly explain.

Let $L$ be a line bundle of degree $d$,
and suppose that $s_{d} \coloneqq \deg(\Lambda) - 2d >0$.
In other words, we let $d < 0$ if $\Lambda = \mathcal{O}_X$ 
and $d \leq 0$ if $\Lambda = \mathcal{O}_X(\check{q}_0)$.
For degree reason, 
the space $\ext(L^{-1}\Lambda, L) \cong H^1(L^2\Lambda^{-1})$ 
of extension classes of the form
\begin{equation}\label{extension-0}
	0 \rightarrow L \rightarrow E \rightarrow L^{-1}\Lambda \rightarrow 0
\end{equation}
is of dimension $N = g-1 + s_{d}$.
As scaling the extensions changes neither the rank-2 bundle $E$ nor its subbundle,
we are interested in the projective space 
$\bP_L \coloneqq \bP H^1(L^2\Lambda^{-1})$.
Note that a hyperplane in $H^0(KL^{-2}\Lambda)$ corresponds to a point in $\bP_L$ via Serre duality.

Via Serre duality, a point in $\bP_L$ 
is equivalent to a hyperplane in $H^0(KL^{-2}\Lambda)$.
Consider the map  
\begin{align*}
&\spanspace_L: X \rightarrow \bP_L, &p \mapsto \spanspace_L(p),
\end{align*}
where $\spanspace_L(p)$ is defined by the hyperplane in $H^0(KL^{-2}\Lambda)$
consisting of sections of $KL^{-2}\Lambda$ vanishing at $p$.
This map is actually an embedding \cite{R-Narasimhan}.
For an effective divisor $\bp = p_1 + ... + p_n$ on $X$,
we denote by $\spanspace_L(\bp)$ the linear subspace of $\bP_L$ 
defined by spanning the lines in $H^1(L^2\Lambda^{-1})$ corresponding to 
$\spanspace_L(p_1)$, ..., $\spanspace_L(p_n)$. 
Varying $\bp$ and taking closure of the union of all subspaces $\spanspace(\bp)$,
one can define the $n$-th secant variety $S_n(X, \bP_L)$.
It is equal to $\bP_L$ if $2n \geq N$
and is otherwise an irreducible variety of dimension $2n-1$  \cite{Lange-notes}.
The role of secant varieties is that they model the Segre strata in $\bP_L$.

\begin{proposition}\label{prop-Lange-Narasimhan} \cite{Lange-Narasimhan}
Let $E$ be a bundle realized by an extension class of the form \eqref{extension-0}.
Then $s(E) \leq 2n-s_d$ if and only if it defines a point in $S_n(X, \bP_L)$.
\end{proposition}

Since a bundle $E$ realized by \eqref{extension-0} already satisfies $s(E) \leq s_d$,
proposition \ref{prop-Lange-Narasimhan} would provide a tighter upper bound
in case $n < s_d$.
In other words, the meaningful flag of secant varieties is
$$ X \cong S_1(X, \bP_L) \subset S_2(X, \bP_L)
\subset \dots \subset S_{s_d - 1}(X, \bP_L) \subset \bP_L.$$
In particular, the secant varieties $S_n(X,\bP_L)$ with $n \leq s_d/2$ 
parametrize extension classes that realize unstable bundles.

%%%%%%%%%%%%%%%%%%%%%%%%%%%%%%%
\paragraph{``Overcounting" bundles.}
Note that if $L$ is a subbundle of $E$, 
$h^0(L^{-1}E)$ then counts the number of linearly independent embeddings of $L$ into $E$.
In particular, if $h^0(L^{-1}E) > 1$ then there exists a family of points in $\bP_L$ which all realise $E$.
In this sense $E$ is ``overcounted" in $\bP_L$ if $h^0(L^{-1}E) > 1$.

%%%%%%%%%%%%%%%%%%%%%%%%%%%%%%%
\paragraph{Explicit construction of $\pmb{\mathbb{P}_L}$.}
Let $L$ be a line bundle of degree $d$ with $s_d > 0$.
Choose and fix a divisor $\bp = p_1 + ... + p_N$ 
such that $\spanspace_L(\bp) = \bP_L$
\footnote{A generic divisor $\bp$ will satisfy this property.}.
Let us now describe extensions of the form \eqref{extension-0}
in terms of transition functions of certain rank-2 bundles they realize.
Our description will provide coordinates on $H^1(L^2\Lambda^{-1})$
and hence projective coordinates on $\bP_L$
by using $\spanspace_L(p_r)$, $r= 1, ..., N$, as building blocks.

Let $\mathbf{q} = \sum_{i=1}^g q_i$ and 
$\mathbf{\check{q}} = \sum_{i=1}^{g - d} \check{q}_i$ be effective divisors on $X$
such that $L \cong \mathcal{O}_X(\mathbf{q} - \mathbf{\check{q}})$
and $\mathbf{p} + \mathbf{q} + \check{\mathbf{q}}$ has no point with multiplicity $> 1$.
Let $w_r$, $z_i$ and $\check{z}_j$ be respectively local coordinates 
in small neighborhoods of $p_r$, $q_i$ and $\check{q}_j$ that do not intersect each other.
Let $X_0$ be the complement of the support of $\mathbf{p} + \mathbf{q} + \mathbf{\check{q}}$.
Given a tuple $\bmx = (x_1, ..., x_N) \in \bC^N$, 
let us define a bundle via the transition functions
\begin{subequations}\label{extension-data-1}
\begin{align}
		&\begin{pmatrix} z_i - z_i(q_i) & 0 \\ 0 & \frac{1}{z_i - z_i(q_i)}	\end{pmatrix},
		& &\begin{pmatrix} \frac{1}{\check{z}_j - \check{z}_j( \check{q}_j )} & 0 \\ 0 & \check{z}_j - \check{z}_j( \check{q}_j ) 	\end{pmatrix},
		& &\begin{pmatrix} 1 & \tfrac{x_r}{w_r - w_r(p_r)} \\ 0 & 1 	\end{pmatrix}
\end{align}
when transiting from $X_0$ to the respective local neighborhoods.
Then a cohomology class $\mathbf{x} \in H^1(L^2)$ is defined 
by the 1-cocyle $\left\{ x_r (w_r - w_r(p_r))^{-1} \right\}_{r=1}^N$.
One can check that,
for $s \in \{1, ..., N\}$, 
the tuple $\Vec{e}_s$ defines an element in $H^1(L^2)$ representing $\spanspace_L(p_s)$.
As $\spanspace_L(\bp) = \bP_L$ by construction,
$(x_1, ..., x_N)$ provide coordinates on $H^1(L^2)$
and projective coordinates on $\bP_L$.
For the case $\Lambda = \mathcal{O}_X(\check{q}_0)$,
by supplementing \eqref{extension-data-1} with transition function
	\begin{align}
		\begin{pmatrix} 1 & 0 \\ 0 & \check{z}_0	\end{pmatrix}
	\end{align}
when transiting from $X_0$ to a local neighborhood of $\check{q}_0$ with local coordinate $\check{z}_0$,
we can carry a similar construction to provide coordinates on $H^1(L^2\Lambda^{-1})$.
\end{subequations}

Let us remind ourselves here that in equipping coordinates on $H^1(L^2\Lambda^{-1})$
via \eqref{extension-data-1}, 
we fix the reference divisor $\bp$
and use the property $\spanspace(\bp) = \bP_L$.
In other words, we are using a different strategy than relying on secant varieties,
which by definition involves varying and taking union of the subspace of $\bP_L$ spanned by the divisors.
We will use the notations 
$E_{\bq, \bmx}$ or $E_{\bla, \bmx}$
to denote the bundle with transition functions \eqref{extension-data-1},
and $L_\bq$ to denote its subbundle isomorphic to $\mathcal{O}_X( \bq - \bqcheck)$.
The construction of $E_{\bq, \bmx}$ will be used frequently throughout this paper; 
for example later in subsection 2.4 we will express Higgs fields on $E_{\bq, \bmx}$  
in terms of abelian differentials.

%%%%%%%%%%%%%%%%%%%%%%%%%%%%%%%%%%%%%%%%%%%%%%%%%%%%%%%%%%%%%%%%%%%%%%
%%%%%%%%%%%%%%%%%%%%%%%%%%%%%%%%%%%%%%%%%%%%%%%%%%%%%%%%%%%%%%%%%%%%%%
%%%%%%%%%%%%%%%%%%%%%%%%%%%%%%%%%%%%%%%%%%%%%%%%%%%%%%%%%%%%%%%%%%%%%%
\subsection{Moduli spaces of pairs of bundles and subbundles}
Let $d$ be an integer such that $s_d = \deg(\Lambda) -2d > 0$.
Consider extensions of the form \eqref{extension-0}, i.e.
pairs $(L, \mathbf{x})$ where $\deg(L) = d$
and $\mathbf{x} \in H^1(L^2 \Lambda^{-1})$.
We say two pairs $(L, \mathbf{x})$ and $(L', \mathbf{x}')$
are equivalent if there is an isomorphism $L \cong L'$
that induces an isomorphism $H^1(L^2 \Lambda^{-1}) \cong H^1({L'}^2 \Lambda^{-1})$
under which $\mathbf{x} \mapsto \mathbf{x}'$.
We denote by $\M$ the space of all such pairs modulo equivalence.
It has a vector bundle structure $\M \overset{J}{\rightarrow} \pic^d$,
where $\pic^d$ be the Picard component of $X$ parameterizing line bundles of degree $d$,
with the fiber over the isomorphism class of $L$ being isomorphic to $H^1(L^2 \Lambda^{-1})$.

Denote by $\N$ the projectivisation of $\M$.
In other words, $\N$ is the moduli space of pairs $(L, [\mathbf{x}])$
where $\deg(L) = d$ and $[\mathbf{x}]$ is the complex line in $H^1(L^2 \Lambda^{-1})$ 
spanned by a nonzero element $\mathbf{x}$;
we will write $[\sfx]$ for the element in $\N$ defined by an element $\sfx \in \M$.
Equivalently, one can regard $\N$ as the moduli space of pairs $(L, E)$,
where $E$ is a rank-2 bundle of determinant $\Lambda$ not isomorphic to $L \oplus L^{-1}\Lambda$ 
and $L$ a subbundle of $E$ of degree $d$.
The equivalence classes of these pairs are defined analogously as in the definition of $\M$.
A formal definition of these moduli spaces can be found in \cite{Gronow}.

For simplicity, most of the time in this paper we will use the same notation
for an object and the point it defines in the corresponding moduli space.
For example, we will frequently write $L \in \pic^d$,
or $(L, \mathbf{x}) \in \M$.

%%%%%%%%%%%%%%%%%%%%%%%%%%%%%%%
\paragraph{Local coordinates.}
Consider an element $(L, \mathbf{x})$ of $\M$.
It is rather straightforward to generalize our construction of 
coordinates on $H^1(L^2\Lambda^{-1})$
to coordinates on a local neighborhood of $(L, \mathbf{x})$.
First, fix a reference divisor $\mathbf{p} = \sum_{r=1}^N p_r$ such that
\begin{subequations}\label{p-q-condition}
\begin{equation}\label{span-condition}
\spanspace(\mathbf{p}) = \bP_L.
\end{equation}
Next, fix a reference divisor $\bqcheck = \sum_{i=1}^{g- d} \check{q}_i$ 
such that there exists a unique effective divisor $\bq = \sum_{i=1}^g q_i$ with $L \cong \mathcal{O}_X(\bq - \bqcheck)$.
This is to say that
\begin{equation}\label{q-non-special}
\bq \text{ is a non-special effective divisor, i.e. }
h^0(X, \mathcal{O}_X(\bq)) = 1.
\end{equation}
Since $\deg(\bqcheck) \geq g$, we can always choose such $\bqcheck$ to make \eqref{q-non-special} hold; in fact a generic one would do.
Finally, since \eqref{span-condition} holds for a generic $\bp$,
we can choose $\bp + \bqcheck$ together such that 
\begin{equation}\label{reduced-condition}
\mathbf{q} + \mathbf{p} + \mathbf{\check{q}} 
\text{ has no points with multiplicity $> 1$}.
\end{equation}
\end{subequations}

Let us now fix $\bp + \bqcheck$, local coordinates $w_r$ around $p_r$
and $\check{z}_j$ around $\check{q}_j$. 
Conditions \eqref{q-non-special} and \eqref{reduced-condition} together allow us to equip coordinates on $J^{-1}(L) = H^1(L^2\Lambda^{-1})$
as in \eqref{extension-data-1}.
In fact, as $L \cong \mathcal{O}_X(\bq - \bqcheck)$ varies,
as long as these two conditions hold
we would be able to equip coordinates on the respective fibers.
Hence we can equip coordinates on the fibers over a neighborhood of $L \in \pic^d$ via \eqref{extension-data-1}.

On the other hand, condition \eqref{q-non-special} allows us to use local coordinates $z_i$ around $q_i$ as local coordinates of this neighborhood of $L$ in $\pic^d$, assuming it is small enough.
The Abel map in this case provides a change of coordinates $\bz = (z_1, ..., z_g) \rightarrow \bla = (\lambda_1, ..., \lambda_g)$ 
for this neighborhood as follows.
Choose a canonical basis of cycles on $X$ and a normalized basis $(\omega_i)_{i=1}^g$ of $H^0(X, K)$.
Fix $x_0 \in X$, 
and let 
\begin{equation}\label{coordinate-lambda}
    \lambda_i(\mathbf{q}) \coloneqq \sum_{j=1}^g \int_{x_0}^{q_j} \omega_i. 
\end{equation}
Then $\bla = (\lambda_1(\bq), ..., \lambda_g(\bq))$
is the evaluation on $\mathbf{q}$ 
of the composition of the Abel map $A: X^{[g]} \rightarrow \jac(X)$
with the isomorphism $\pic^0(X) \cong \jac(X) \overset{\sim}{\rightarrow} \pic^d$
defined by $\mathcal{O}_X \mapsto \mathcal{O}_X(g x_0 - \mathbf{\check{q}})$.

We conclude that given $(L, \mathbf{x}) \in \M$,
we can always choose an effective divisor $\bp + \bqcheck$
satisfying \eqref{p-q-condition} and local coordinates of points in $\bp + \bqcheck + \bq$ 
to equip local coordinates on a neighborhood of $(L, \mathbf{x})$.
Let us call such $\bp + \bqcheck$ a reference divisor for the open set in $\M$ on which we can equip coordinates.
From now on, we will frequently choose such reference divisors 
and denote points on $\M$ by these local coordinates.
In addition, 
we will frequently
use the bundle $E_{\bq, \bmx}$ constructed from the coordinates $(\bz, \bmx)$ of $\sfx \in \M$ (cf. \eqref{extension-data-1})
to represent the isomorphism class of rank-2 bundles determined by $\sfx$.

\begin{proposition}\label{prop-coordinates-dense}
    Let $\bp + \bqcheck = \sum_{i=1}^N p_i + \sum_{i=1}^{g-d} \check{q}_i$
    be a generic divisor on $X$.
    Then the parameters $(\bz, \bmx)$ induced via \eqref{extension-data-1}
    are coordinates on a dense subset of $\M$.
\end{proposition}
\begin{proof}
Let us fix a generic divisor $\bp + \bqcheck$.
We need to show that  
conditions \eqref{p-q-condition} 
hold for generic points $(L, \mathbf{x}) \in \M$.
But condition \eqref{span-condition} holds whenever the linear system 
associated to $KL^{-2}\Lambda$ defines an embedding 
$X \hookrightarrow \mathbb{P}_L$ that is not degenerate,
while conditions \eqref{q-non-special} and \eqref{reduced-condition}
clearly hold for generic $L = \mathcal{O}_X(\bq - \bqcheck)$.
\end{proof}

\begin{remark}
Note that, for each $L \in \pic^d$, 
the coordinates $(x_1, ..., x_N)$ on $H^1(L^2\Lambda^{-1})$ 
are defined depending on the local coordinates $w_1$, ..., $w_r$ of $p_1$, ..., $p_N$.
We will show in Corollary \ref{prop-coordinates-change} that
a change of coordinates $w_r \rightarrow w_r'$ induces a change of coordinates 
$x_r \rightarrow x_r'  = \frac{dw_r}{dw_r'} (p_r) x_r$.
\end{remark}

%%%%%%%%%%%%%%%%%%%%%%%%%%%%%%%%%%%%%%%%
\paragraph{The projections $\M \dashrightarrow \Nstable$ 
and $\N \dashrightarrow \Nstable$.}
It follows from proposition \ref{prop-Lange-Narasimhan}
that the extension classes that realize stable bundles
form open dense subsets $\M^{stable} \subset \M$ and $\N^{stable} \subset \N$.
Forgetting the line bundles and remembering only the realized rank-2 bundles define maps
\begin{align*}
    I: \M^{stable} &\longrightarrow \Nstable,
& i: \N^{stable} &\longrightarrow \Nstable,
\end{align*}
i.e. rational maps $\M \dashrightarrow \Nstable$ 
and $\N \dashrightarrow \Nstable$ respectively.

The fiber of $i$ over a point in $[E] \in \Nstable$ consists of all extension classes that realize $E$ modulo scaling,
or in other words the subline bundles of $E$ of degree $d$.
If $s_d \leq g$ and $[E]$ is generic, 
by Nagata's bound and proposition \ref{prop-Lange-Narasimhan},
$i^{-1}([E])$ consists of maximal subbundles of $E$.
We may then can quote several results from the literature regarding maximal subbundles to properties of generic fibers of $i$.

\begin{proposition} \label{prop-fiber-max-line}
\begin{enumerate}[label=(\roman*)]
\item If $s_d = g$ then all fibers of $i$ are 1-dimensional, and generically 
\footnote{These fibers are smooth over very stable bundles, namely stable bundles 
which do not admit nilpotent Higgs fields. We will briefly recall this concept in the next subsection.}
are smooth \cite{Lange-Narasimhan, Gronow}. 
\item If $s_d = g-1$ then the generic
\footnote{Again, this holds for fibers over very stable bundles.}
fiber is $2^g:1$ \cite{Lange-paper, CCP07}. 
\item If $s_d < g$ then the generic fibers consist of only 1 point \cite{Lange-Narasimhan}.
\end{enumerate}
\end{proposition}

%%%%%%%%%%%%%%%%%%%%%%%%%%%%%%%%%%%%%%%%%%%%%%%%%%%%%%%%%%%%%%%%%%%%%%
%%%%%%%%%%%%%%%%%%%%%%%%%%%%%%%%%%%%%%%%%%%%%%%%%%%%%%%%%%%%%%%%%%%%%%
%%%%%%%%%%%%%%%%%%%%%%%%%%%%%%%%%%%%%%%%%%%%%%%%%%%%%%%%%%%%%%%%%%%%%%
\subsection{Moduli spaces of Higgs bundles}
In the following, we will review the moduli spaces of $\SL$-Higgs bundles.
We will start with recalling the underlying bundles that form stable $\SL$-Higgs bundles, 
before reviewing the integrable structure and stratification induced by the $\bC^\ast$-action on the moduli space of stable $\SL$-Higgs bundles. 

An $\SL$-Higgs bundle is a pair $(E, \phi)$
where $E$ is a holomorphic rank-2 bundle and $\phi$, called a Higgs field, is an element of $H^0( \End_0(E) \otimes K)$,
i.e. it is a trace-less holomorphic endomorphism of $E$ twisted by holomorphic one-forms.
Such a Higgs bundle is called stable
if there is no sub-line bundle $L_E$ of $E$
that destabilizes it and is kept invariant by $\phi$,
i.e. $\deg(E) -2\deg(L_E) \leq 0$ and $\phi(L_E) \subset L_E \otimes K$.
In particular, any Higgs field on a stable bundle 
defines a stable Higgs bundle.

The Hitchin moduli space $\mM_H(\Lambda)$
is the moduli space of $SL_2(\mathbb{C})$-stable Higgs bundles
with the underlying bundles having determinant $\Lambda$. 
This space is a complex manifold of dimension $6g-6$ which 
was first constructed and studied by Hitchin \cite{Hit87a} \cite{Hit87b}.
As $T_{E} \Nstable \cong H^1(\End_0(E))$ for a stable bundle $E$,
by Serre duality the space $H^0( \End_0(E) \otimes K) \cong H^1(\End_0(E)) ^\ast$ 
of Higgs fields $E$ is the cotangent space of $\Nstable$ at $E$.
In fact, the total cotangent space $T^\ast \Nstable$ embeds into an open dense subset of $\mM_H(\Lambda)$.
In addition, there is a natural symplectic structure on  $\mM_H(\Lambda)$ which restricts to the canonical one on $T^\ast \mN$.

%%%%%%%%%%%%%%%%%%%%%%%%%%%%%%%%%%
\subsubsection{Underlying bundles of stable Higgs bundles.}
Besides stable bundles, unstable bundles can also form stable Higgs bundles
as long as their destabilizing subbundles are not kept invariant by the Higgs fields.
We refer to Hitchin's original work \cite{Hit87a} for a complete classification of the underlying bundles that form stable Higgs bundles.
\begin{proposition}\label{Hitchin-classification}
	{\normalfont \cite{Hit87a}} An $\SL$-Higgs bundle $(E, \phi)$ is stable 
 if and only if one of the following conditions holds
	\begin{enumerate}[label=(\roman*)]
		\item $E$ is stable,
		\item $E$ is strictly semi-stable, i.e. $s(E) = 0$, and $g > 2$,
		\item $E \cong L \otimes U$ is strictly semi-stable and $g=2$, where $U$ is either decomposable or an extension of $\mathcal{O}_X$ by itself,
		\item $E$ is destabilized by subbundle $L_E \hookrightarrow E$ with $h^0(KL_E ^{-2} \Lambda ) > 1$,
		\item $E = L_E \oplus L_E^{-1}\Lambda$ with  $h^0(KL_E ^{-2} \Lambda ) = 1$.
	\end{enumerate}
\end{proposition}
We note that, regarding stable Higgs bundles with unstable underlying bundles, 
the key point in the proof of Hitchin's classification is as follows.
Let $E$ be an unstable bundle.
Then a Higgs field $\phi$ on $E$ defines a stable Higgs bundle if and only if,
given a destabilizing subbundle $L_E$ of $E$ and the local form $\phi\mid_{U_\alpha} = \big(\begin{smallmatrix}
 a_\alpha & b_\alpha \\ c_\alpha & -a_\alpha \end{smallmatrix}\big)$ in local frames adapted 
\footnote{This means with respect to these local frames, transition functions of $E$ are of upper-triangular form.}
to $L_E$, we have $c_\alpha \neq 0$.
These lower-left components of the Higgs fields actually glue into a section of $K L_E^{-2} \Lambda$.
In the following we generalize the discussion on this lower-left component, 
which has an important role in this paper.

%%%%%%%%%%%%%%%%%%%%%%%%%%%%%%%%%%%%%%%%%%
\paragraph{The lower-left component of Higgs fields induced by extension classes.}
Consider a rank-2 bundle $E$ of determinant $\Lambda$ (we do not assume stability).
Let $L \hookrightarrow E$ be a subbundle of $E$, 
and denote by ${\bfx} \in H^1(L^2 \Lambda^{-1})$ the extension class defined by this embedding.
Let $\phi$ be a Higgs field on $E$ and suppose that it takes the form $\big(\begin{smallmatrix} a_\alpha & b_\alpha \\ c_\alpha & -a_\alpha \end{smallmatrix}\big)$
in local frames adapted to $L \hookrightarrow E$.
Then the local lower-left components $c_\alpha$ glue in a section of $K L^{-2} \Lambda$
which corresponds to the composition
\begin{equation*}
c_{\bfx}(\phi): L \hookrightarrow E \overset{\phi}{\rightarrow}  E \otimes K \rightarrow L^{-1} \Lambda K.
\end{equation*}
In other words, choosing an embedding of a subbundle $L$ into $E$ defines a map
\begin{equation*}
	c_{\bfx}: H^0(\End_0(E)\otimes K ) \rightarrow H^0(K L^{-2}\Lambda).
\end{equation*}
The condition $c_{\bfx}(\phi) \neq 0$ is equivalent to the condition that
$L$ is not $\phi$-invariant.

Note that $c_{\bfx}$ fits in the l.e.s.
\begin{equation*}
0 \rightarrow H^0(E^\ast L K) \rightarrow H^0(\End_0(E) \otimes K) 
\overset{c_{\bfx}}{\rightarrow} H^0(K L^{-2}\Lambda) \rightarrow H^1(E^\ast L K) \rightarrow...
\end{equation*}
induced by the injection from the bundle $E^\ast L K$ of $L$-invariant Higgs fields.
Now, the extension class ${\bfx} \in H^1(L^2\Lambda^{-1})$ defines via Serre duality 
a hyperplane $\ker({\bfx}) \subset H^0(KL^{-2}\Lambda)$
consisting of sections that evaluate ${\bfx}$ to zero.
It can be shown that
\begin{equation}\label{Serre-duality-constraint}
\ima(c_{\bfx}) \subset \ker({\bfx}).
\end{equation}
This Serre duality constraint follows from Proposition \ref{prop-Higgs-abelian-diff} below.
In case $E$ is stable, it can be shown via Riemann-Roch and Serre duality that $\ima(c_{\bfx})$ is of codimension $h^0(L^{-1}E)$
in $H^0(KL^{-2}\Lambda)$.
Hence if $E$ is stable and is not ``overcounted'' in $\bP_L$, i.e. $h^0(L^{-1}E) = 1$,
then $\ima(c_{\bfx}) = \ker({\bfx})$. 

%%%%%%%%%%%%%%%%%%%%%%%%%%%%%%%%%%%%%%%%%%%
\paragraph{Stratification.}
Let us denote by $W_{\Nstable}$ the copy of $T^\ast\Nstable$ 
in $\mM_H(\Lambda)$. 
For $d \in [\deg(\Lambda), 2g-2 + \deg(\Lambda)]$,
let $W_d$ consist of isomorphism classes of stable Higgs bundles $(E, \phi)$
where $E$ is destabilized by a subbundle $L_E \hookrightarrow E$
and $\deg(K L_E^{-2} \Lambda) = d$.
Then the decomposition
\begin{align*}
\mM_H(\Lambda) = W_{\Nstable} \sqcup \big( \underset{d}{\sqcup} W_{d} \big)
\end{align*}
defines a stratification on $\mM_H(\Lambda)$.
This coincides with the \BB stratification induced by the $\mathbb{C}^\ast$-action on $\mM_H(\Lambda)$ 
defined by scaling the Higgs fields in the rank-2 cases \cite{HH21}.

\begin{example}
For $\Lambda = \mathcal{O}_X$, upon choosing a spin structure $K^{1/2}$, 
consider Higgs bundles of the form
$\left( K^{1/2} \oplus K^{-1/2}, \big(\begin{smallmatrix} 0 & q \\ 1 & 0 \end{smallmatrix} \big)\right)$
where $q \in H^0(K^2)$.
These define a section of the Hitchin fibration and is called the Hitchin section.
There are $2^{2g}$ such Hitchin sections corresponding to $2^{2g}$ choices of $K^{1/2}$,
and together they define $W_0$
\end{example}

Observe that $W_{\Nstable} \cong T^\ast \Nstable$ inherits a stratification from the Segre stratification on $\Nstable$.
The stratification on $\mM_H(\Lambda)$ defined by taking the stratification on $W_{\mN}$ together with the \BB stratification
is hence defined by the maximal subbundles of the underlying bundles.
We can regard this as a natural refinement of the \BB stratification and a generalization of the stratification on $W_{\Nstable}$.
 
%%%%%%%%%%%%%%%%%%%%%%%%%%%%%%%%%
\subsubsection{Spectral correspondence and integrable structure}
The map $h: \mM_H \rightarrow H^0(K^2)$ which assigns to the isomorphism class of a Higgs bundle $(E, \phi)$ the quadratic differential $\det(\phi)$
is called the Hitchin fibration.
This map endows $\mM_H$ with the structure of an algebraic integrable system \cite{Hit87b},
as a generic fiber of $h$ is isomorphic to an abelian variety, namely the Prym variety of the associated spectral curve.
We now briefly recall this construction.

First, note that associated to a quadratic differential $q = \det(\phi)$
is a ``spectral curve'' $S_{q}$ embedded in the total space of $T^\ast X$.
The spectral curve encodes the eigen-values of the Higgs field. 
Concretely, if $u$ is the coordinate of an open set $U \subset X$,
$v$ the fiber coordinate of $T^\ast X\mid_U$ 
and $\phi (u) = \big(\begin{smallmatrix} a(u) & b(u) \\ c(u) & - a(u) \end{smallmatrix}\big)$ locally, 
then $S_q$ is locally defined by
\begin{align}
	v^2 + q(u) = v^2  - a(u)^2 - b(u) c(u) = 0.
	\label{spectral-curve-eqn}
\end{align}
The morphism $S_q \overset{\pi}{\rightarrow} X$ induced by $T^\ast X \rightarrow X$ is a $2:1$ covering that branches at the zeroes of $\det(\phi)$.
The involution $\sigma$ of $S_q$ interchanges the points in $\pi^{-1}(u)$ corresponding to the eigenvalues $\pm (-q(u))^{1/2}$ of $\phi(u)$.
We say a quadratic differential $q$ and its associated spectral curve $S_q$ are non-degenerate if the zeroes of $q$ are all simple.
In this case, $S_q$ is a smooth compact Riemann surface of genus $\tilde{g} = 4g-3$, 
and in particular $\pi^\ast(K)$ has a canonical section defined locally by $(-q(u))^{1/2}$ that vanishes precisely at the ramification divisor $\mathcal{R}_q$ of $S_q$.

%%%%%%%%%%%%%%%%%%%%%%%%%%%%%%%%%%%
\paragraph{Eigen-line bundles.}
A Higgs bundle $(E, \phi)$ with non-degenerate $q = \det(\phi)$ corresponds up to isomorphism 
to a subbundle of $\pi^\ast(E)$ on $S_q \overset{\pi}{\rightarrow} X$, 
as we briefly explain now.
Let $\mathcal{L}$ be the kernel of the morphism $(\pi^\ast(\phi) - v): \pi^\ast(E) \rightarrow \pi^\ast(E \otimes K)$.
In other words at each point $p \equiv (u, v) \in S_q$ it coincides with the eigen-subspace of $\pi^\ast(\phi)(p)$ with the eigen-value $v$.
Since $\pi^\ast(\phi)(p)$ also has $-v$ as its eigen-value, 
one can similarly define a sub-line bundle of $\pi^\ast(E)$ with these eigen-values, which is nothing but $\sigma^\ast(\mathcal{L})$.
The line bundles $\mathcal{L}$ and $\sigma^\ast(\mathcal{L})$ are called the eigen-line bundles of $(E,\phi)$.
They coincide at the ramification divisor of $S_q \overset{\pi}{\rightarrow} X$ and satisfy
\begin{align} \label{L-L-sigma}
	\mathcal{L} \otimes \sigma^\ast(\mathcal{L}) \cong \pi^\ast(\Lambda \otimes K^{-1}).
\end{align}
Conversely, given a line bundle $\mathcal{L}$ satisfying (\ref{L-L-sigma}),
one can show that the direct image $\pi_\ast(\mathcal{L} \otimes \pi^\ast(K))$ is a rank-2 bundle
whose determinant is isomorphic to $\Lambda$ 
and pull-back to $S_q$ contain $\mathcal{L}$ and $\sigma^\ast(\mathcal{L})$ as subbundles. 
A Higgs field can be constructed from the fact that, at each point $p \in S_q$, $\pi^\ast(\phi)$ acts on $\mathcal{L}$ and $\sigma^\ast(\mathcal{L})$ with eigen-values corresponding to $p$ and $\sigma(p)$ respectively.
By construction, the eigen-line bundles of this Higgs bundle are $\mathcal{L}$ and $\sigma^\ast(\mathcal{L})$.

\paragraph{Prym variety and integrable structure.}
Let $q$ be a non-degenerate quadratic differential, i.e. it has only simple zeroes.
We now recall the Prym variety of $S_q$ and how condition \eqref{L-L-sigma}
in fact implies that the Hitchin fiber $h^{-1}(q)$ is isomorphic to this Prym variety \cite{Mumford, Hit87a}.

The Prym variety $\prym(S_q) \subset \jac(S_q)$ is 
defined as the kernel of the norm map $\mathrm{Nm}: \jac(S_q) \rightarrow \jac(X)$ that sends the equivalence class $[D]$ of degree-0 divisors $D$ on $S_q$ to $[\pi(D)]$.
A characterisation of $\prym(S_q)$ more suited for our discussion is 
$$\prym(S_q) = \{ L \in \jac(S_q) \mid L \otimes \sigma^\ast(L) \cong \mathcal{O}_{S_q} \},$$
where we have regarded $\jac(S_q)$ as the set of isomorphism classes of degree-0 line bundles on $S_q$.
Then choosing any line bundle $\mathcal{L}_0$ that satisfies
condition (\ref{L-L-sigma}) allows us to define an isomorphism $h^{-1}(q) \overset{\sim}{\rightarrow} \prym(S_q)$ by $[E, \phi] \mapsto [{\mathcal{L}_0}^{-1} \otimes \mathcal{L}_{(E,\phi)}]$ where $\mathcal{L}_{(E,\phi)}$ is the eigen-line bundle of $(E,\phi)$.
Since a line bundle satisfies (\ref{L-L-sigma}) if and only if it is the eigen-line bundle of a Higgs bundle having $S_q$ as its spectral curve, 
we have defined such an isomorphism simply by identifying a point in $h^{-1}(q)$ with $0 \in \prym(S_q)$.

\begin{remark}
	For $S_q$ non-degenerate, pulling-back line bundles from $X$ to $S_q$ defines an embedding, $\pi^\ast: \jac(X) \hookrightarrow \jac(S_q)$.
	The intersection of $\prym(S_q)$ and the image of $\pi^\ast$ is the discrete set of $2^{2g}$ points $\{ \pi^\ast [L] \mid L^{\otimes 2} \cong \mathcal{O}_X \}$.
\end{remark}

%%%%%%%%%%%%%%%%%%%%%%%%%%%%%%%%%%%%
\paragraph{The nilpotent cone and wobbly bundles.}
The Hitchin fiber over a quadratic differential that has zeroes with non-trivial multiplicity is singular. 
The most singular fiber, called the nilpotent cone, is $h^{-1}(0)$,
consisting of isomorphism classes of stable Higgs bundles with nilpotent Higgs fields.
Clearly the nilpotent cone contains a copy of $\Nstable$ which consists of Higgs bundles of the form $(E,0)$ with $E$ stable. 
We say a stable bundle that admit nonzero nilpotent Higgs fields is wobbly;
otherwise we say it is very stable.
The wobbly bundles define a divisor on $\Nstable$ \cite{Laumon, Pal-Pauly},
which is the image along the rational forgetful map 
of the complement of $\Nstable$ in the nilpotent cone $h^{-1}(0)$.

If $\phi$ is nilpotent Higgs field on $E$, its kernel  
defines a subbundle $L_\phi$ of $E$ with $h^0(K L_\phi^2 \Lambda^{-1}) >0$.
Conversely, if $E$ has a subbundle $L$ with $h^0(K L^2 \Lambda^{-1}) >0$,
then there exist nilpotent Higgs fields on $E$ admitting $L$ as their kernel.
It follows that very stable bundles are maximally stable:
if $E$ is not maximally stable, i.e. $s(E) \leq g-2$,
then by degree reason its maximal subbundles are kernels of nilpotent Higgs fields.
In this sense the difficult part of characterizing the wobbly divisor on $\Nstable$ lies in
identifying maximally stable bundles that are wobbly.

%%%%%%%%%%%%%%%%%%%%%%%%%%%%%%%%%%%%%%%%%%%%%%%%%%%%%%%%%%%%%%%%%%%%%%
%%%%%%%%%%%%%%%%%%%%%%%%%%%%%%%%%%%%%%%%%%%%%%%%%%%%%%%%%%%%%%%%%%%%%%
%%%%%%%%%%%%%%%%%%%%%%%%%%%%%%%%%%%%%%%%%%%%%%%%%%%%%%%%%%%%%%%%%%%%%%
\subsection{Representing Higgs fields by abelian differentials}

Consider an element $\sfx = (L, \bfx) \in \M$ with local coordinates $(\bz, \bmx)$ as in \eqref{extension-data-1}
upon choosing a reference divisor $\mathbf{p} + \mathbf{\check{q}}$ and local coordinates.
Let $\phi$ be a Higgs field on $E_{\bq, \bmx}$, and suppose
\begin{equation}\label{phi-abelian-diff}
	\phi\mid_{X_0} = \begin{pmatrix} \phi_0& \phi_- 
		\\ \phi_+ & - \phi_0\end{pmatrix} 
\end{equation}
is its local form on $X_0$ in local frames adapted to the embedding $L_{\bq} \hookrightarrow E_{\bq, \bmx}$.
For a divisor $D$ on $X$, denote by $\Omega_{D}$ the space of meromorphic differentials
with divisor bounded below by $-D$.
Denote by $\bm{r}$ the divisor $2\bqcheck + \deg(\Lambda)\check{q}_0$,
which is of degree $2g + s_d$.

\begin{proposition}\label{prop-Higgs-abelian-diff}
	With the setup as above, $\phi_0$, $\phi_\pm$ are meromorphic differentials holomorphic on $X_0$ such that
	\begin{itemize}
		\item[(i)] $\phi_+$ is an element of $\Omega_{-2\bq + \bm{r}} \cong \Hzero$,
		\item[(ii)] $\phi_0$ is an element of $\Omega_{\bp}$ with $\underset{p_r}{\Res} \phi_0 = -x_r \phi_+(p_r)$ at each $p_r$,
		\item[(iii)] $\phi_- = (-\det(\phi) - \phi_0^2)/\phi_+$ 
        is an element of $\Omega_{2\bp + 2\bq - \bmr}$,
		with the singular parts at each $p_r$ fully determined in terms of $\bmx$, $\phi_0$ and $\phi_+$.
	\end{itemize}
\end{proposition} 
\begin{proof}
	These properties follow from checking the regularity of the local form of $\phi$ in the neighborhoods of points in $\bp+ \bq + \bqcheck$,
	which is the conjugation of \eqref{phi-abelian-diff} with the transition functions \eqref{extension-data-1} of $E_{\bq, \bmx}$.
\end{proof}

Note that $E_{\bq, \bmx}$ can be reconstructed
from $\divisor(\phi_+)$ and the residues at $\bp$ of $\phi_0$.
We will call such a matrix a \textit{Higgs differential} associated to the bundle $E_{\bq, \bmx}$.
In other words, a Higgs differential associated to the bundle $E_{\bq, \bmx}$ 
is a matrix of abelian differentials that has poles only at $\bp + \bq + \bqcheck$
which can be removed by 
conjugating it with the  transition functions of $E_{\bq, \bmx}$.
It is clear by construction that there is a 1-1 correspondence between
Higgs fields and Higgs differentials associated to $E_{\bq, \bmx}$.

\subsubsection{Components of Higgs differentials}
We now discuss the freedom associated to the components of the Higgs differentials on a given bundle $E_{\bq, \bmx}$. 
For simplicity, in the following discussion we do not distinguish between $E_{\bq, \bmx}$ and a bundle $E$ isomorphic to it.

%%%%%%%%%%%%%%%%%%%%%%%%%%%%%%%%%%%%%%%%
\paragraph{The lower-left component.}
The abelian differential $\phi_+$ corresponds to $c_{\bfx}(\phi)$ via the isomorphism 
$\Omega_{-2\bq + \bm{r}} \cong H^0(KL^{-2}\Lambda)$.
Let $\bmk = (k_1, ..., k_N)$ be defined by evaluating $\phi_+$ at $\bp$,

%%%%%%%%%%%%%%%%%%%%%%%%%%%%%%%%%%%%%%%%
\paragraph{The diagonal component.}
Let us now consider a fixed abelian differential $\phi_+ \in \Omega_{-2\bq + \bm{r}}$ 
with a corresponding section $s \in \Hzero$ in the image of $c_{\bfx}$.
We will  
denote by $\CH_{\bmx, \bmk}$ the space of Higgs differentials having $\phi_+$ as their lower-left component,
i.e. $\CH_{\bmx, \bmk}$ corresponds to $c_{\bfx}^{-1}(s)$.
Clearly $\CH_{\bmx, \bmk}$ is an affine space modeled over the space $\CH_{\bmx, \bm{0}}$ 
of upper-triangular Higgs differentials,
which corresponds to the space $H^0(E^\ast L K)$ of Higgs fields preserving $L$.
By assigning to any Higgs differential in $\CH_{\bmx, \bmk}$ its diagonal component, we define a map
\begin{align*}
	&\Pi_{\bmx, \bmk}: \CH_{\bmx, \bmk} \longrightarrow \Omega_{\bmx, \bmk},
	&\begin{pmatrix} \phi_0 & \phi_- \\ \phi_+ & -\phi_0 \end{pmatrix}
	\mapsto \phi_0,
\end{align*} 
where the target is the space of abelian differentials with fixed residues at points of $\bp$,
$$\Omega_{\bmx, \bmk} 
\coloneqq  \left\{ \omega_0 \in \Omega_{\bp} \mid  \underset{p_r}{\Res} \hspace{2pt} \omega_0 = -x_r k_r \text{ for } r = 1, ..., N \right\}, $$
which is an affine space modeled on $H^0(K)$. 
Note that condition that the sum of residues of $\phi_0$ vanishes is equivalent to 
the Serre duality constraint $\ima(c_\bfx) \subset \ker(\bfx)$ stated in \eqref{Serre-duality-constraint}. 

\paragraph{The top-right component.}
The space of strictly-upper triangular Higgs differentials 
corresponds to the space $H^0(KL^2 \Lambda^{-1}) \cong \Omega_{2\bq - \bmr}$ of nilpotent Higgs fields having $L$ as kernel.
For any $\phi_0 \in \Omega_{\bmx, \bmk}$ in the image of 
$\Pi_{\bmx, \bmk}$,
its preimage is an affine space modeled on $H^0(KL^2 \Lambda^{-1})$.
In other words, fixing $\phi_+$ and $\phi_0$,
the space of $\phi_-$ that would complete the Higgs differential
is an affine space modeled on $\Omega_{2\bq - \bmr}$.

%%%%%%%%%%%%%%%%%%%%%%%%%%%%%%%%%%%%%%%%

\medskip

A priori it is not clear if $\Pi_{\bmx, \bmk}$ is 
surjective.  The residues of $\phi_-$ are determined by $\phi_+$ and $\phi_0$, suggesting that
that condition that the sum of these residues vanishes could impose additional constraints on the
choice of $\phi_+$ and $\phi_0$. 
The following proposition shows that, for $s_d \leq g-1$, generically the map 
$\Pi_{\bmx, \bmk}$ is surjective.

\begin{proposition}\label{prop-surjective}
	Suppose $\deg(L^{-2}\Lambda) = s_d \leq g-1$.
	Then $\Pi_{\bmx, \bmk}$ is surjective if $h^0(L^{-2}\Lambda) = 0$.
	If in addition $h^{0}(L^{-1} E) =1$, 
 then the other direction also holds.
\end{proposition}
\begin{proof}
	The image of $\Pi_{\bmx, \bmk}$ is an affine space modeled 
	over the image of $\Pi_{\bmx, \bm{0}}: \CH_{\bmx, \bm{0}} \rightarrow 
	\Omega_{\bmx, \bm{0}} = H^0(K)$.
	The corresponding map $\Pi_{\bmx, \bm{0}}'$ for Higgs fields  fits in the l.e.s.
	\begin{equation*}
		0 \rightarrow H^0(L^2 \Lambda^{-1} K) \rightarrow H^0(E^\ast L K)
		\overset{\Pi_{\bmx, \bm{0}}'}{\rightarrow} H^0(K) \rightarrow H^1(L^2 \Lambda^{-1} K) \rightarrow ...
	\end{equation*}
	induced by the injection from the bundle $L^2 \Lambda^{-1} K$ of nilpotent Higgs fields
	having $L$ as the kernel into $E^\ast L K$.
	By exactness and Serre duality, 
	if $h^0(L^{-2}\Lambda) = h^1(L^2 \Lambda^{-1} K) = 0$,
	then $\Pi_{\bmx, \bm{0}}$ is surjective.
	In fact, one can compute 
        \begin{subequations}\label{RR}
        \begin{align}
            \begin{split}
            h^0(E^\ast L K ) &= 2g-2 - s_d + h^0(L^{-1} E), 
            \end{split} \\
            \begin{split}\label{RR-2}
            h^0(L^2 \Lambda^{-1} K) &= g-1 - s_d + h^0(L^{-2}\Lambda).
            \end{split}
        \end{align}
	  \end{subequations}
	using Riemann-Roch and Serre duality. 
	By exactness, the image of $\Pi_{\bmx, \bm{0}}$ has dimension
	\begin{equation*}
		g - 1 + h^0(L^{-1} E) - h^0(L^{-2}\Lambda).
	\end{equation*}
	Hence if in addition $h^0(L^{-1}E) = 1$,
	$\Pi_{\bmx, \bm{0}}$ is surjective 
	only if $h^0(L^{-2}\Lambda) = 0$.
\end{proof}

\begin{remark}\label{rm-extra-nilpotent}
For $s_d \leq g-1$, 
$h^0(L^{-2}\Lambda)$ determines  how much larger the dimension $h^0(K L^2 \Lambda^{-1})$ is compared to 
its minimal value $g-1 - s_d$ (cf. \eqref{RR-2}).
A non-vanishing value of  $h^0(L^{-2}\Lambda)$ implies 
existence of additional nilpotent Higgs fields having $L$ as kernel.
Another interpretation of $h^0(L^{-2}\Lambda)$ will be discussed in subsection 4.3.
\end{remark}

\begin{corollary}\label{cor-wobbly-not-surjective}
	Let $s(E) = s_d = g-1$.
	Then $\Pi_{\bmx, \bmk}$ fails to be surjective if and only if 
	$E$ is wobbly.
\end{corollary}
\begin{proof}
	For $s(E) = s_d = g-1$, $L$ is a maximal subbundle of $E$ and hence $h^0(L^{-1}E) = 1$.
	By Riemann-Roch and Serre duality, the space $H^0(L^2 \Lambda^{-1} K)$ of nilpotent Higgs fields admitting $L$ as the kernel has dimension $h^0(L^2 \Lambda^{-1} K) = h^0(L^{-2}\Lambda)$.
\end{proof}

%%%%%%%%%%%%%%%%%%%%%%%%%%%%%%%%%%%%%%%%

\subsubsection{Higgs differentials on {wobbly} bundles}\label{SsecHiggswobbly}
It follows from Proposition \ref{prop-surjective} and Corollary \ref{cor-wobbly-not-surjective}
that the constraints satisfied by the diagonal matrix element 
$\phi_0$ are stronger in the case of wobbly bundles than generically. 
We shall here make origin and consequences of this phenomenon more explicit. 

In order to understand the origin of the additional constraints, one may recall that
the singular parts of $\phi_-$ are determined by $\phi_+$ and $\phi_0$ according to 
Proposition \ref{prop-Higgs-abelian-diff}. Existence of solutions to the constraints formulated
in point (iii) of Proposition \ref{prop-Higgs-abelian-diff} can imply additional constraints
on $\phi_0$. It will turn out that this happens exactly if $\phi$ is a Higgs differential associated
to a Higgs pair $(E,\phi)$ with $E$ wobbly. 

In order to see this, note that elements $\phi_-$ 
of $\Omega_{2\bp + 2\bq - \bmr}$ can be represented as elements of $\Omega_{2\bp + 2\bq}$
vanishing at the reference divisor $\bmr$. 
An element $\chi$ of $\Omega_{2\bp + 2\bq}$ having the desired 
singular properties
can be represented in the form 
$$\chi = \chi_0 + \sum_{k=1}^{3g-1} l_k \chi_k ,\qquad \bm{l} = (l_1, ..., l_{3g-1}) \in \bC^{3g-1},$$
where
$\chi_0$ is an element of $\Omega_{2\bp + 2\bq}$ with singular parts 
determined by $\phi_+$ and $\phi_0$
according to Proposition \ref{prop-Higgs-abelian-diff},
and $(\chi_k)_{k=1}^{3g-1}$ be a basis of $\Omega_{2\bq}$.
If $\chi$ % = \chi_0 + \sum_{k=1}^{3g-1} l_k \chi_k $$
vanishes at $\bmr$ with double zero at each point in $\bqcheck$,
then it is an element of $\Omega_{2\bp + 2\bq - \bmr}$
that together with $\phi_+$ and $\phi_0$ defines a Higgs differential on $E_{\bq, \bmx}$.
These vanishing conditions are equivalent to an inhomogeneous linear system of equations
\begin{equation}\label{linear-system}
    A \bm{l} = \bm{l}_0.
\end{equation} 
Here $A$ is a $(2g + s_d) \times (3g-1)$ matrix depending on
the evaluations and first derivatives of $\chi_k$ at points in $\bmr$,
and $\bm{l}_0$ is obtained from $\chi_0$ in a similar way.
Note that $\bm{l}_0$ depends linearly on the variables $\bmx$, $\bmk$ and $\bm{\kappa}$ determining the residues of 
$\chi_0$ in terms of $\phi_0$.

The linear space spanned by the null-vectors of $A$
is isomorphic to $\Omega_{2\bq - \bmr}$. 
Let us furthermore note that $\Omega_{2\bq-\bmr}\simeq H^0(KL^2\Lambda^{-1})$.
We see that $A$ has maximal rank when  $h^0(KL^2\Lambda^{-1})=0$, corresponding to the 
case where $E$ has no nilpotent Higgs fields with kernel $L$.
Solutions to the equations (\ref{linear-system}) will then exist for arbitrary $\bm{l}_0$,
implying that the map $\Pi_{\bmx, \bmk}$ is surjective.

Otherwise we may observe that
null vectors of $A$ correspond to elements of $H^0(KL^2\Lambda^{-1})$ which can be used to define nilpotent Higgs fields.
There do not exist solutions to (\ref{linear-system}) for arbitrary $\bm{l}_0$,
implying that the map $\Pi_{\bmx, \bmk}$ is not surjective. However, there will exist 
solutions to (\ref{linear-system}) for choices of $\bm{l}_0$ taken from a proper subspace 
of $\bC^{3g-1}$ of co-dimension $h^0(\Omega_{2\bq - \bmr})$. There is a freedom in the 
choice of such solutions represented by addition of null-vectors of $A$, corresponding to 
nilpotent Higgs fields. 

Note that the determinant $\mathrm{det}(A)$ can be represented
as a function of the coordinates $\bm{\lambda}$ determining the choices of the line bundle $L$.
The considerations above imply that Higgs fields on very stable bundles will generically develop 
a singular behavior in the vicinity of values of $\bm{\lambda}$ corresponding to 
wobbly loci where $h^0(KL^2\Lambda^{-1})>0$,
caused by the vanishing of $\mathrm{det}(A)$. The loci where this singularity is absent 
can be characterised by linear equations for the coordinates $\bm{\kappa}$ determining $\phi_0$. 

%%%%%%%%%%%%%%%%%%%%%%%%%%%%%%%%%%%%%%%%%%%%%%%%%%%%%%%%%%%%%%%%%%%
%%%%%%%%%%%%%%%%%%%%%%%%%%%%%%%%%%%%%%%%%%%%%%%%%%%%%%%%%%%%%%%%%%%
%%%%%%%%%%%%%%%%%%%%%%%%%%%%%%%%%%%%%%%%%%%%%%%%%%%%%%%%%%%%%%%%%%%
\section{Poisson structures and Darboux coordinates}
We shift our attention now to the cotangent spaces 
$T^\ast \M$ and $T^\ast \N$.
As we will often use the notation $\sfx$ for elements of $\M$,
given an element $\xi \in T^\ast_{\sfx}\M$ with $\sfx = (L, \bfx)$,
we will use the notations $\sfx(\xi)$ and $\sfx_L(\xi)$
for $\sfx$ and $\bfx$ respectively.

Let us recall the projections from $\M$ and $\N$
that remembers only the isomorphism classes of stable rank-2 bundles and line bundles respectively via the diagram 
\begin{equation*}
\centering
\begin{tikzcd}
& &\M \arrow[dashed]{dl}[swap]{I} \arrow{dr}{J} 
& & &\N \arrow[dashed]{dl}[swap]{i} \arrow{dr}{j} & \\ %%%%%%%%
&\mN_\Lambda & 
&\mathrm{Pic}^d
&\Nstable & &\mathrm{Pic}^d
\end{tikzcd};
\end{equation*}
Recall also the projectivisation map $\pr$ 
from the complement of the zero section in the bundle $\M \rightarrow \pic^d$
to $\N$.
Let $(\bla, \bmx)$ be the local coordinates on some open set $\mathcal{U} \subset \M^{stable}$ 
upon choosing a reference divisor $\bp + \bqcheck$ and local coordinates.
Alternatively, one can use $(\bz, \bmx)$ as local coordinates on $\M$.
Let $\bkap = (\kappa_1, ..., \kappa_g)$, $\bzcheck = (\check{z}_1, ..., \check{z}_g)$ 
and $\bm{k} = (k_1, ..., k_N)$ be respectively the conjugate coordinates 
of $\bla$, $\bz$ and $\bmx$ on the fibers of $T^\ast \M$.
We are using here the same notation 
for the evaluation of the abelian differential $\phi_+$ at $\bp$ (cf. subsection 2.4)
and the canonical conjugate of $\bmx$, but proposition \ref{prop-coordinate-Higgs} that follows shortly justifies this abuse of notations.
The canonical symplectic form on $T^\ast\M$ takes the local form
\begin{equation}\label{Darboux-coordinates}
\wom = \sum_{i=1}^g d\lambda_i \wedge d\kappa_i 
+ \sum_{r=1}^N dx_r \wedge dk_r
= \sum_{i=1}^g dz_i \wedge d\check{z}_i 
+ \sum_{r=1}^N dx_r \wedge dk_r.
\end{equation}

\begin{remark}\label{rm-dense-nbd}
It follows from Proposition \ref{prop-coordinates-dense} that the Darboux coordinates 
in \eqref{Darboux-coordinates} are defined on a dense subset of $T^\ast \M$.
\end{remark}

The main result of the next two subsections relates these Darboux coordinates 
to the evaluation at certain points on $X$
of the abelian differentials representing the local form \eqref{extension-data-1} of $\phi$.
We will frequently identify elements $\sfx \in \mathcal{U}$
with the values of their coordinates $(\bz, \bm{x})$ or equivalently $(\bla, \bmx)$.
Recall the construction \eqref{extension-data-1} from these data of the bundles 
$E_{\bq, \bmx}$ which are representatives of the isomorphism classes $I(\sfx) \in \Nstable$.
The pull-back 
$$I^\ast_\sfx: H^0(\End_0(E_{\bq, \bmx}) \otimes K)  \cong T^\ast_{E_{\bq, \bmx}}\mN_{\Lambda}
\longrightarrow T^\ast_{\sfx} \M$$
maps Higgs fields on $E_{\bq, \bmx}$ to cotangent vectors on $\M$ at $\sfx$.
For such a Higgs field $\phi$, 
let us denote by $\check{z}_i(\phi)$, $\kappa_i(\phi)$
and $k_r(\phi)$ the corresponding coordinates of $I^\ast_\sfx(\phi)$.
We will see that $I^\ast_\sfx$ can be represented explicitly in terms of the 
the matrix elements $\phi_+$ and $\phi_-$ of 
$\phi=\big(\begin{smallmatrix} \phi_0 & \phi_- \\ \phi_+ & - \phi_0 \end{smallmatrix}\big)$ 
in local frames defined by a representation of 
$E$ as an extension $0\rightarrow L\rightarrow E\rightarrow L^{-1}\Lambda\rightarrow 0$.

\begin{proposition}\label{prop-coordinate-Higgs}
With the notations above,
suppose $\phi \in H^0(\End_0(E_{\bq, \bmx}) \otimes K)$ 
takes the local form 
$\big(\begin{smallmatrix} \phi_0 & \phi_- \\ \phi_+ & - \phi_0 \end{smallmatrix}\big)$
on $X_{0}$ as in \eqref{phi-abelian-diff}. 
Then
\[
\begin{aligned}
\check{z}_i (\phi) &= - 2 \phi_0( q_i), \quad i=1,\dots,g,\\
k_r(\phi) & = \phi_+(p_r),\qquad r=1,\dots,N,
\end{aligned}\qquad
\kappa_i (\phi) = - 2 \sum_{j=1}^g (dA^{-1})_{ij} \phi_0(q_j).
\]
\end{proposition}
\noindent One may note that the matrix elements $\phi_+$ and $\phi_-$
are uniquely determined by the values of the coordinates ${x}_r$, ${k}_r$, ${q}_i$, and 
${\kappa}_i$. The proof of Proposition \ref{prop-coordinate-Higgs}
will be given in Section 3.2.

\begin{remark}\label{rmk-forget-phi-}
Note that, upon choosing reference divisors $\bp + \bqcheck$ and local 
coordinates,
the Darboux coordinates of $I^\ast_\sfx(\phi)$ on $T^\ast \M$ 
determine, and are determined by, the lower triangular part of the Higgs differential, 
namely $\phi_0$ and $\phi_+$.
Pulling-back of a Higgs bundle $(E, \phi)$, 
regarded as a cotangent vector of $\Nstable$, under $I$
hence amounts to choosing a subbundle $L \hookrightarrow E$ 
and forgetting the upper-right component $\phi_-$ of the 
corresponding Higgs differential.
\end{remark}

%%%%%%%%%%%%%%%%%%%%%%%%%%%%%%%%%%%%%%%%%%%%%%%%%%%%%%%%%%%%%%%%%%%%%%%%
%%%%%%%%%%%%%%%%%%%%%%%%%%%%%%%%%%%%%%%%%%%%%%%%%%%%%%%%%%%%%%%%%%%%%%%%
\subsection{Tangent spaces}
\paragraph{Tangent vectors on $\pmb{\pic^d}$.}
Since $\bq$ is assumed to be a non-special divisor, 
the map $\bq \mapsto \bla(\bq)$ 
defines an isomorphism from a neighborhood of $\bq \in X^{[g]}$ to a local neighborhood of $L \simeq \mathcal{O}_X(\bq- \bqcheck) \in \pic^d$.
Let $\del_{z_i}$ be push-forward along $A$ of the local vector field defined by $z_i$,
and $\del_{\lambda_j}$ the local vector field defined by $\lambda_j$.
Then at $A(\bq)$ we have
\begin{align*}
\partial_{z_i} &= 
\sum_{j=1}^g (dA \mid_{\bq})_{ij} \partial_{\lambda_j} 
= \sum_{j=1}^g \omega_j(q_i)  \del_{\lambda_j}, \\
\del_{\lambda_i} &= \sum_{j=1}^g ( dA^{-1} \mid_{\bla} )_{ij} \del_{z_j}.
\end{align*}

%%%%%%%%%%%%%%%%%%%%%%%%
\paragraph{Tangent vectors on $\M$, $\N$ and $\mN_\Lambda$.}
Let $\mathcal{U} \subset \M$ be a equipped with coordinates $(\bz, \bmx)$, 
and suppose it contains no split extensions. 
We will in the following give explicit description of the local vector fields 
$X_{z_i}$ and $X_{x_r}$ on $\mathcal{U}$
corresponding to the variation of the coordinates $z_i(q_i)$ and $x_r$ respectively \cite{Hit99}. 
For a fixed $i \in \{1, ..., g\}$
and for each $\sfx = (L, \bmx) \in \mathcal{U}$ with coordinates $\bz = (z_i (q_i))_{i=1}^g$ and $\bmx$, 
%let $z_{\sfx, i} = z_i - z_i(q_i(\sfx))$.
consider the one-parameter family $\eta_{z_i}(t)$ of 
extensions that are defined by the data \eqref{extension-data-1}, 
except on $U_i \cap X'_{\mathbf{q}}$ where it takes the form 
\begin{equation*}\label{family-del-z}
\begin{pmatrix} z_{i} - z_i(q_i(\sfx)) - t & 0 \\ 
0 & \left( z_{i} - z_i(q_i(\sfx)) - t \right)^{-1} \end{pmatrix}.
\end{equation*}
Clearly 
%$\eta_{z_i}(0) = \sfx$ and 
$$J( \eta_{z_i}(t) ) = \mathcal{O} \left( \bq(\sfx) - q_{i}(\sfx) + q_{\sfx, i}(t) 
+ q_g(\sfx) - \bqcheck \right),$$
where $q_{\sfx, i}(t) \in U_i$ has coordinate $z_i(q_i(\sfx)) + t$.
This defines $\tX_{z_i}$ 
and its push-forward $X_{z_i}$ along $\pr: \mathcal{U} \rightarrow \N$.
Similarly, for $r \in \{1, ..., N\}$,
consider the one-parameter family $\eta_{x_r}(t)$ of extensions
that are defined by the data \eqref{extension-data-1}, 
except on $V_r \cap X'_{\mathbf{q}}$ where it takes the form
\begin{equation*}\label{family-del-y}
\begin{pmatrix} 1 & (x_r(\sfx) + t) w_r^{-1} \\ 0 & 1 \end{pmatrix}.
\end{equation*}
Clearly $\eta_{x_r}(0) = \sfx$ and $J(\eta_{x_r}(t)) = J(\sfx)$.
This defines $\tX_{x_r}$ and its pushforward $X_{x_r}$ along $\pr$.
Let us also define 
\begin{align*}
&\tX_{\lambda_i} = \sum_{j=1}^g \left( dA^{-1} \mid_{\bla} \right)_{ij} \tX_{z_j},
&i = 1, ..., g,
\end{align*}
the analogue of the change of basis $(\del_{z_i}) \rightarrow (\del_{\lambda_i})$ on $\pic^d$.
It is clear that
\begin{align*}
&J_\ast(\tX_{z_i}) = \del_{z_i},
& &J_\ast(\tX_{\lambda_i}) = \del_{\lambda_i},
& &J_\ast(\tX_{x_r}) = 0.
\end{align*}
If $\mathcal{U} \subset \M^{stable}$, denote by $Y_{z_i}$, $Y_{\lambda_i}$ and $Y_{x_r}$
respectively the push-forward of $\tX_{z_i}$, $\tX_{\lambda_i}$ and $\tX_{x_r}$
along $I: \mathcal{U} \rightarrow \mN_\Lambda$.

\begin{proposition}\label{prop-linear-rel-tangent-vectors}
We have 
\begin{equation*}
\sum_{r=1}^N x_r(\sfx) X_{x_r} = 0.
\end{equation*}
If in addition $\sfx$ realizes a stable bundle, then similarly 
\begin{equation*}
\sum_{r=1}^N x_r(\sfx) Y_{x_r} = 0.
\end{equation*}
\end{proposition}
\begin{proof}
The one-parameter family of bundles
with transition functions on $V_r \cap X'_{\mathbf{q}}$,
\begin{align}
&(E(t))_{V_r X'_{\mathbf{q}}} = \begin{pmatrix} 1 & w_r^{-1} x_r(\sfx) (t+1) 
\\ 0 & 1 \end{pmatrix},
&r = 1, ..., N,
\end{align}
defines the tangent vector $\sum_{r=1}^N x_r(\sfx) \tX_{x_r}$
on $\M$.
However, $E(t)$ is the scaling of the element $\sfx_L \in H^1(L^2\Lambda^{-1})$ corresponding to $\sfx$,
and hence defines the zero tangent vector on $\N$ and $\Nstable$.
\end{proof}

\begin{remark}\label{rem-projectivize}
Recall the projectivisation map $\pr$ that defines $\N$ 
from the open dense subset in $\M$ defined by non-split extensions classes.
Let $\sfx \in \M$ be defined by such a non-split extension class.
Since the pull-back 
$$ \pr^\ast_{\mathsf{x}}: T^\ast_{[\mathsf{x}]}\N \rightarrow T^\ast_\mathsf{x}\M $$
is injective, 
it follows from proposition \ref{prop-linear-rel-tangent-vectors}
that $T^\ast_{[\mathsf{x}]} \N$ is isomorphic to 
$$ \ima(\pr^\ast_\mathsf{x}) = \ker\left( \sum_{r=1}^N x_r(\mathsf{x}) \tX_{x_r} \right)
= \{ \xi \in T^\ast_\mathsf{x}\M \mid c(\xi) \in \ker(\mathsf{x}) \}. $$
\end{remark}

%%%%%%%%%%%%%%%%%%%%%%%%%%%%%%%%%%%%%%%%%%%%%%%%%%%%%%%%%%%%%%%%%%%
%%%%%%%%%%%%%%%%%%%%%%%%%%%%%%%%%%%%%%%%%%%%%%%%%%%%%%%%%%%%%%%%%%%
%%%%%%%%%%%%%%%%%%%%%%%%%%%%%%%%%%%%%%%%%%%%%%%%%%%%%%%%%%%%%%%%%%%
\subsection{Cotangent spaces}
\paragraph{Proof of proposition \ref{prop-coordinate-Higgs}.}
We would like to compute
\begin{align*}
&\check{z}_i(\phi) \coloneqq \langle I^\ast_{\mathsf{x}}(\phi), \tX_{z_i} \rangle
= \langle \phi , Y_{z_i} \rangle,
&k_r(\phi) \coloneqq \langle I^\ast_{\mathsf{x}}(\phi), \tX_{x_r} \rangle
= \langle \phi, Y_{x_r} \rangle.
\end{align*}
First note that a tangent vector $Y \in H^1(\End_0(E)) = T_{[E]}\Nstable$ 
defined by a one-parameter family $E(t)$ with transition functions $g_{\alpha \beta}(t)$
can be represented by the $\End_0(E)$-valued 1-cocyle
$g^{-1}_{\alpha\beta}(0) \dot{g}_{\alpha\beta}(0)$.
Hence $Y_{z_i}$ can be represented by the 1-cocyle that takes the form
\begin{equation}\label{cocycle-del-q}
	z_i^{-1} \begin{pmatrix} -1 & 0 \\ 0 & 1 \end{pmatrix}.
\end{equation}
on $U_i \cap X_0$ and vanishes elsewhere.
Similarly, $Y_{x_r}$ can be represented by 
the 1-cocyle that takes the form
\begin{equation}\label{cocycle-del-y}
	w_r^{-1} \begin{pmatrix} 0 & 1 \\ 0 & 0	\end{pmatrix}.
\end{equation}
on $V_r \cap X_0$ and vanishes elsewhere.
Using these representative 1-cocyles, it is straightforward now to compute the Serre duality pairing $\langle \phi, Y_{z_i} \rangle$, $\langle \phi, Y_{x_r} \rangle \in H^1(X, K)$
of these tangent vectors with $\phi$.
Indeed, suppose $\phi$ takes the local form 
$\big(\begin{smallmatrix} \phi_0 & \phi_+ \\ \phi_- & -\phi_0 \end{smallmatrix}\big)$ on $X_0$.
Then $\langle \phi, Y_{z_i} \rangle$ and $\langle \phi, Y_{x_r} \rangle \in H^1(X, K)$
respectively can be represented by the $K$-valued 1-cocyle that takes the form
$-2z^{-1}_i \phi_0$ and $\text{and } w_r^{-1}\phi_-$ on the respective domains. 
Recalling that the isomorphism $H^1(K) \overset{\Res}{\simeq} \bC$ can be made explicit using Mittag-Leffler distributions representing $K$-valued 1-cocyles, 
we can show $\check{z}_i(\phi) = -2\phi_0(q_i(\mathsf{x}))$ and $k_r(\phi) = \phi_-(p_r)$.
The evaluation 
$$\kappa_i(\phi) \coloneqq \langle I^\ast_{\mathsf{x}}(\phi), \tX_{\lambda_i} \rangle
= \langle \phi , Y_{\lambda_i} \rangle $$
comes from the definition of $\tX_{\lambda_i}$ as a linear combination of $\tX_{z_j}$. 
\hspace{-13pt} $\qed$

We now discuss some consequences of proposition \ref{prop-coordinate-Higgs}.

%%%%%%%%%%%%%%%%%%%%%%%%%%%%%%%%%%%
\paragraph{Change of coordinates.}
Recall that the local coordinates $(\bz, \bmx)$ we equip on $\M$
depends both on the choice of the reference divisor $\sum_{r=1}^N p_r + \sum_{i=1}^{g-d} \check{q}_i$ 
and the local coordinates around each $p_r$ and $\check{q}_i$.

\begin{corollary}\label{prop-coordinates-change}
    A change of coordinates $w_r \rightarrow w_r'$ on a local neighborhood of $p_r$ induces 
	a change of coordinates $x_r \rightarrow x_r' =  \frac{dw_r}{dw_r'} \Big\vert_{p_r} x_r$. 
\end{corollary}
\begin{proof}
	Note that $\langle \phi, X_{x_r'} \rangle = \phi_-(w_r'(p_r)) = \frac{dw_r}{dw_r'} \Big\vert_{p_r} \phi_-(w_r(p_r)) = \frac{dw_r}{dw_r'} \Big\vert_{p_r} \langle \phi, X_{x_r} \rangle$ holds at all $[E] \in \ima(I) \subset \mN$ and for all $\phi \in H^0(\End_0(E) \otimes K)$.
\end{proof}

%%%%%%%%%%%%%%%%%%%%%%%%%%%%%%%%%%%%
\paragraph{Wobbly bundles revisited.}
Let $\sfx \in \M^{stable}$,
and $E$ be a representative bundle of the isomorphism class $[E] = I(\sfx)$.
The following corollary 
follows immediately from Proposition \ref{prop-coordinate-Higgs}.
\begin{corollary}\label{cor-kernel-I}
$\ker(I^\ast_\sfx) 
\cong \{ (E, \phi)  \mid \phi \text{ nilpotent, } \ker(\phi) = L \}$.
\end{corollary}
By dimension counting, 
$I^\ast_\sfx$ always has a non-trivial kernel for $s_d < g-1$ 
but generically has no kernel for $s_d = g-1$.
In other words, for $s_d = g-1$, $I^\ast_\sfx$ is not an isomorphism 
only if $E$ is a wobbly bundle.
For the next corollary, recall that  
stable bundles with Segre invariant equal to $g-1$ are maximally stable.

\begin{corollary}\label{no-surjectivity}
Let $s(\Lambda,d) = g-1$.
In this case, the intersection of the loci of critical values of $i: \N \dashrightarrow \Nstable$
with the loci of maximally stable bundles
consists precisely of maximally stable wobbly bundles.
\end{corollary}
\begin{proof}
A point $[\sfx]\in \N$ is a critical point of $i$
if the push-forward
$i_{[\sfx],\ast}: T_{[\sfx]} \N \rightarrow T_{i([\sfx])}\Nstable$ is not injective,
or equivalently $i^\ast_{[\sfx]}$ is not surjective.
One now uses the above observation.
\end{proof}

\begin{remark}
Pal and Pauly proved a related result 
(cf. part (2) of Theorem 1.2 in \cite{Pal-Pauly}).
Namely, they showed that the Quot scheme $M(E)$ parametrising 
maximal subbundles of a maximally stable bundle $E$ with Segre invariant $g-1$
is degenerate, i.e. $M(E)$ is either non-reduced or of positive dimension,
if $E$ is wobbly with its maximal subbundles being kernels of nilpotent Higgs 
fields. 
Corollary \ref{no-surjectivity} suggests that the ``only if'' direction 
of their result should also be true.
Note that these wobbly bundles define one irreducible component 
of the wobbly locus (cf. Theorem 1.1 in \cite{Pal-Pauly}).
Corollary \ref{no-surjectivity} hence furthermore suggests 
an extension of part (ii) of Proposition \ref{prop-fiber-max-line}:
the number of maximal subbundles 
of wobbly bundles in other components of the wobbly locus should also be $2^g$.
\end{remark}

%%%%%%%%%%%%%%%%%%%%%%%%%%%%%%%%%%
\paragraph{Lifts of classical Hitchin Hamiltonians.}
Let us consider now the case $s_d = g-1$.
Recall that the classical Hitchin Hamiltonians are holomorphic functions $\mM_H(\Lambda) \rightarrow \bC^{3g-3}$
defined by taking the coordinates of $(E, \phi) \mapsto \det(\phi)$
w.r.t. a basis of $H^0(K^2)$.
Since $i^\ast_{[\sfx]}: T^\ast_E \hspace{2pt} \Nstable \rightarrow T^\ast_\sfx \N$ 
is an isomorphism for $E$ very stable, 
we can lift classical Hitchin Hamiltonians to an open dense set on $T^\ast \N$.
Consider an element $\sfx = (L, \bfx) \in \M$
where $E = i(\sfx)$ is a wobbly bundle with $h^0(KL^2 \Lambda^{-1}) = h^0(L^2 \Lambda^{-1}) > 0$.
It follows from the discussion regarding the linear system \eqref{linear-system}
and proposition \ref{prop-coordinate-Higgs}
that the extensions of these functions are in general singular at a generic covector in $T^\ast_{[\sfx]}\N$.
These functions are regular though at certain subspace of $T^\ast_{[\sfx]} \N$
that constrains the RHS of \eqref{linear-system} properly such that there exists solutions to the linear system.

%%%%%%%%%%%%%%%%%%%%%%%%%%%%%%%%%%%
\paragraph{How Darboux coordinates split cotangent spaces.}
Observe from subsection 3.1 that 
the coordinates $\bla$ (or equivalently $\bz$) and $\bmx$
can respectively serve as coordinates on $\pic^d$ and the fibers over it.
Hence these coordinates defines a splitting of the tangent space at $\sfx = (L, \bfx)$ into
\begin{equation*}\label{splitting-tangent}
    T_\sfx \M \simeq T_L \pic^d \oplus T_\sfx J^{-1}(L).
\end{equation*}
This splitting induces a splitting of the cotangent space
\begin{subequations}\label{splitting-cotangent}
\begin{equation}
T^\ast_{\sfx}\M \simeq 
T^\ast_{L}\pic^d(X) \oplus T^\ast_\sfx J^{-1}(L) 
\simeq H^0(K) \oplus \Hzero.
\end{equation}
Concretely, this splitting is done through the projection
\begin{equation}\label{splitting-projection}
\kappa: T^\ast_{\sfx}\M \longrightarrow H^0(K),
\hspace{25pt} \xi = (\bla, \bmx, \bkap, \bmk) \longmapsto 
\kappa(\xi) \coloneqq - \frac{1}{2}\sum_{i=1}^g \kappa_i \omega_i.
\end{equation}
\end{subequations}
Since the evaluation of $\kappa(\xi)$ at $q_i$ is $\check{z}_i$,
this projection indeed restricts to an inverse of the embedding
$J^\ast_{\sfx}: H^0(K) \hookrightarrow T^\ast_{\sfx}\M$ on its image.

Let us now highlight the relation between 
the splitting defined by $\kappa$ to the abelian differentials representing Higgs fields.
For $r = 2, ..., N$,
let $\omega_{p_+ - p_-}'$ be the unique abelian differential of the third kind
that vanishes at all points of $\bq$,
has simple poles at $p_\pm$ with respective residues $\pm 1$.
It can be expressed by an explicit formula
\begin{equation}\label{renormalized}
\omega_{p_+ - p_-}' =
     \omega_{p_+ - p_-} 
-  \sum_{i, j=1}^g  \omega_{p_+ - p_-}(q_j) \left( dA^{-1}\mid_{\bla} \right)_{ij} \omega_i,
\end{equation}
where $\omega_{p_+ - p_-}$ is the unique meromorphic differential 
having vanishing $A$-periods and simple poles at $p_\pm \in X$ with respective residues $\pm 1$. 
Then if $\xi = (\bla, \bmx, \bkap, \bmk)$ is the pull-back of a Higgs field $\phi$ on $E_{\bq, \bmx}$ 
constructed in \eqref{extension-data-1}, 
the diagonal component $\phi_0$ of the Higgs differential $\phi \mid_{X_0}$ can be written as
\begin{equation}\label{abelian-diff-explicit}
\phi_0 
= - \sum_{r=2}^N  k_r  x_r \omega_{p_r - p_1}' 
+ \kappa(\xi)
= - \sum_{r=2}^N  k_r  x_r \omega_{p_r - p_1}' 
- \frac{1}{2}\sum_{i=1}^g \kappa_i \omega_i.
\end{equation}

%%%%%%%%%%%%%%%%%%%%%%%%%%%%%%%%%%%%%%%%%%%%%%%%%%%%%%%%%%%%%%%%%%%
%%%%%%%%%%%%%%%%%%%%%%%%%%%%%%%%%%%%%%%%%%%%%%%%%%%%%%%%%%%%%%%%%%%
%%%%%%%%%%%%%%%%%%%%%%%%%%%%%%%%%%%%%%%%%%%%%%%%%%%%%%%%%%%%%%%%%%%
\subsection{Symplectic reduction}
There is a natural $\bC^\ast$-action of $\M$ 
that scales the corresponding extension classes.
By definition, the quotient
of the set of non-split extension classes defines $\N$.
On the other hand, one can check that the induced $\bC^\ast$-action on $T^\ast \M$ 
is symplectic.
It is known \cite{Hit87b} that
the cotangent of a manifold quotient 
is symplectomorphic
to the symplectic quotient of the original manifold.
For the purpose of this paper, we will consider the symplectic reduction of  a proper, open dense subset of $T^\ast\M$.
We will start by some preliminaries
that are needed to define this subset and to formulate the moment map of the $\bC^\ast$-action in an invariant manner

%%%%%%%%%%%%%%%%%%%%%%%%%%%%%%%
\paragraph{Preliminaries.}
At a point $\mathsf{x} = (L, \mathbf{x}) \in \M$, 
the pull-back of $J$ induces a s.e.s. of cotangent spaces
\begin{equation}\label{ses-M}
0 \longrightarrow T^\ast_{L}\pic^d(X) \overset{J^\ast_\mathsf{x}}{\longrightarrow} 
T^\ast_{\mathsf{x}}\M 
\longrightarrow T^\ast_{\mathsf{x}} J^{-1}(L)
\longrightarrow 0,
\end{equation}
Note that 
\footnote{Recall that if $v \in V$ is an element of a  vector space 
then $T_v V \simeq V$ canonically via the identification of $v' \in V$
with the one-parameter family $v + tv'$, $t\in \bC$.
We use this isomorphism in the second equality here.}
\begin{equation}\label{isom-1}
T^\ast_{\mathsf{x}} J^{-1}(L) 
\simeq \left( T_{\mathbf{x}} H^1(L^2\Lambda^{-1}) \right)^\ast 
\simeq \left( H^1(L^2\Lambda^{-1}) \right)^\ast
\simeq H^0(K L^{-2}\Lambda).
\end{equation}
We denote by
\begin{align}
    &k: T^\ast_{\mathsf{x}}\M \rightarrow T^\ast_{\mathsf{x}} J^{-1}(L)
    & &\text{and} & &k_L: T^\ast_{\mathsf{x}}\M  \rightarrow \Hzero
\end{align}
the quotient map in \eqref{ses-M} and its composition with \eqref{isom-1} respectively.
We have denoted these map as such since  
in terms of local Darboux coordinates,
$k_L(\xi)$ has coordinates $k_1$, ..., $k_N$
with respect to the dual of 
the coordinates 
in which $\mathbf{x} \in H^1(L^{2}\Lambda)^{-1}$ has coordinates $x_1$, ..., $x_N$.
Let $L \hookrightarrow E$ be an embedding that represents $\sfx$,
and recall the map $c_\bfx: T^\ast_{E} \Nstable \rightarrow H^0(K L^2 \Lambda)$ induced by this embedding.
The following proposition follows directly from proposition \ref{prop-coordinate-Higgs}.

\begin{proposition}\label{prop-analogue-c}
Let $\mathsf{x} \in \M^{stable}$. 
Suppose $\xi = I^\ast_\mathsf{x}(\phi)$
where $\phi$ is a Higgs field on a rank-2 bundle realized by $\mathsf{x}$.
Then 
$$k_L(\xi) = c_\bfx (\phi) .$$
\end{proposition}

The analogue of \eqref{ses-M} at $[\sfx] \in \N$ is induced by the pull-back of $j$, namely
\begin{equation}\label{ses-cotangent-N}
0 \longrightarrow T^\ast_{L}\pic^d(X) \overset{j^\ast_{[\mathsf{x}]}}{\longrightarrow} 
T^\ast_{[\mathsf{x}]}\mN_{\Lambda, d} 
\overset{k}{\longrightarrow} T^\ast_{[\mathsf{x}]}j^{-1}(L)
\longrightarrow 0.
\end{equation}
Note that 
$$T^\ast_{[\mathsf{x}]}j^{-1}(L) \simeq T^\ast_{[\bfx]} \bP_L  \overset{\bfx^\ast}{\simeq} \ker(\bm{x}) \subset \Hzero
$$
where the isomorphism $\bfx^\ast$ is defined using the choice of the representative $\bfx \in \Hone$ of the complex line $[\bfx]$.
We will also denote by $k$ and $k_L$
the quotient map in \eqref{ses-cotangent-N} 
and its composition that forms 
$T^\ast_{[\mathsf{x}]}\mN_{\Lambda, d}  \rightarrow \ker(\mathbf{x})$.
This abuse of notations is justified as the diagram
\begin{equation*}
\begin{tikzcd}
&T^\ast_{\mathsf{x}}\M 
\arrow{r}{k_L} 
&H^0(K L^{-2}\Lambda)  \\
&T^\ast_{[\mathsf{x}]}\mN_{\Lambda, d}  \arrow{r} 
\arrow{u}{\pr^\ast_\sfx}
&\ker(\mathbf{x}) \arrow[u, hook]
\end{tikzcd}
\end{equation*}
commutes.
In addition, for $\mathsf{x} \in \M^{stable}$ and $E$ a bundle realized by $\sfx$, 
the diagram 
\begin{equation*}
\begin{tikzcd}
& &T^\ast_\sfx\M  \\
&T^\ast_E \Nstable \arrow{r}{i^\ast_{[\sfx]}} \arrow{ur}{I^\ast_{\sfx}} 
&T^\ast_{[\sfx]}\N \arrow{u}{\pr^\ast_\sfx}
\end{tikzcd}
\end{equation*}
commutes.
In the sense of these commutative diagrams and proposition \ref{prop-analogue-c}, 
one can regard these maps $k_L$ as analogues of the map 
$c_\bfx$.
In particular, the analogues at the level of their respective kernels are among the space of $L$-invariant Higgs fields on $E$
and the copy of $H^0(K)$ in $T^\ast_\sfx\M$.
These two spaces in fact are isomorphic via $I^\ast_\sfx$
if and only if $h^0(L^{-2} \Lambda) = 0$.

%%%%%%%%%%%%%%%%%%%%%%%%%%%%%%%%
\paragraph{$\pmb{\bC^\ast}$-action and moment map.}
A nonzero complex number $\epsilon$ defines an automorphism on $\M$ via
scaling by $\epsilon$ the corresponding extension classes.
By construction this automorphism preserves the fibers of $\M \rightarrow \pic^d$.
The induced $\bC^\ast$-action on $T^\ast \M$ is defined by pulling-back.
Recalling that for an element $\xi \in T^\ast_{\sfx}\M$
with $\sfx = (L, \bfx)$,
we denote by $\sfx_L(\xi)$ the corresponding cohomology class in $\Hone$.
Then the $\bC^\ast$-action in particular satisfies
\begin{align*}
&\sfx_L(\epsilon.\xi) = \epsilon \hspace{1pt} \sfx_L(\xi) \in H^1(L^2\Lambda^{-1}), 
&k_L(\epsilon.\xi) = \epsilon^{-1} k_L(\xi) \in H^0(KL^{-2}\Lambda).
\end{align*}
In terms of local Darboux coordinates, 
\begin{equation*}
\xi = (\pmb{\lambda}, \bmx, \pmb{\kappa}, \bm{k} )
\longmapsto
\epsilon.\xi =
(\pmb{\lambda}, \epsilon \bmx, \pmb{\kappa}, \epsilon^{-1} \bm{k} ).
\end{equation*}
It is clear that the $\bC^\ast$-action preserves the canonical symplectic form.
One can check that the moment map of the $\bC^\ast$-action on $T^\ast \M$ is defined 
via the Serre duality pairing 
\begin{align*}
	&H: T^\ast\M \rightarrow \bC, 
	&\xi \mapsto \langle \sfx_L(\xi), k_L(\xi) \rangle. 
\end{align*}
In terms of local Darboux coordinates 
\begin{equation}\label{moment-map-explicit}
	H\left( (\pmb{\lambda}, \bmx, \pmb{\kappa}, \bm{k} ) \right) = \bmx.\mathbf{k}
	= \sum_{r=1}^N x_r k_r. 
\end{equation}
The level sets of $H$ are manifestly closed under the $\bC^\ast$-action.
Note that if $\xi \in T^\ast M$ is the pull-back of a Higgs bundle
then $H(\xi) = 0$.
This follows from constraint \eqref{Serre-duality-constraint}
and proposition \ref{prop-coordinate-Higgs}.

%%%%%%%%%%%%%%%%%%%%%%%%%%%%%%%%%%%%%
\paragraph{Symplectic reduction.}
Consider a fiber $T_\sfx^\ast \M$
with $\sfx = (L, \bfx)$ and $\sfx \neq 0$.
On such a fiber, the $\bC^\ast$-action is free and proper on $\{ \xi \in T_\sfx^\ast \mM_{\Lambda,d} \mid k(\xi) \neq 0 \}$.
Hence in particular the $\bC^\ast$-action is free and proper on
\begin{equation*}
	\{ \xi \in H^{-1}(0) \mid
	\sfx(\xi) = (L, \bfx), \bfx \neq 0, k_L(\xi) \neq 0 \}.
\end{equation*}
We will restrict our consideration further by imposing the additional condition that $k_L(\xi)$ has only simple zeroes.
Let
\begin{subequations}\label{domain-SoV}
\begin{equation}
T^\ast\M^s \coloneqq \{ \xi \in H^{-1}(0) \mid
\sfx(\xi) = (L, \bfx), \bfx \neq 0, k_L(\xi) \neq 0 
\text{ and has only simple zeroes} \}.
\end{equation}
(The superscript ``s" is meant to remind the Serre duality constraint and simple zeroes conditions that define $T^\ast\M^s$.)
By Marsden-Weinstein-Meyer theorem, the quotient 
\begin{equation}
    T^\ast\N^s \coloneqq T^\ast\M^s/\bC^\ast 
\end{equation}
\end{subequations}
is a smooth manifold and has a symplectic form $\omega$ 
such that its pull-back along $ T^\ast\M^s \rightarrow T^\ast\N^s$
is equal to the restriction of $\widetilde{\omega}$ to $T^\ast\M^s$.
One can check that (cf. appendix A) $T^\ast \N^s$
is symplectomorphic to the open dense set in $T^\ast \N$
which is the complement of the subset
$$ \{ \zeta \in T^\ast \N \mid k(\zeta) = 0 
\text{ or corresponds to a section having zeroes of order } > 1 \} $$
and hence is dense in $T^\ast \N$.

%%%%%%%%%%%%%%%%%%%%%%%%%%%%%%%%%%%%%%%%%%%%%%%%%%%%%%%%%%%%%%%%%%%%%%
%%%%%%%%%%%%%%%%%%%%%%%%%%%%%%%%%%%%%%%%%%%%%%%%%%%%%%%%%%%%%%%%%%%%%%
%%%%%%%%%%%%%%%%%%%%%%%%%%%%%%%%%%%%%%%%%%%%%%%%%%%%%%%%%%%%%%%%%%%%%%
\section{Baker-Akhiezer divisors}
In this section we introduce the notion of Baker-Akhiezer (BA) divisors associated to Higgs pairs $(L \hookrightarrow E,\phi)$. The BA divisors are divisors on non-degenerate spectral curves projecting to the divisors 
of zeros of the lower-left matrix elements of the Higgs fields in frames adapted to descriptions of $E$ as 
extensions $0\rightarrow L \rightarrow E \rightarrow L^{-1}\Lambda\rightarrow 0$. We are going to 
describe an inverse construction recovering Higgs pairs $(L \hookrightarrow E,\phi)$ from 
BA divisors and some supplementary discrete data.

\subsection{Definitions and basic properties of Baker-Akhiezer divisors}

Let $(E, \phi)$ be a stable Higgs bundle with an associated non-degenerate quadratic differential $q$, and $L$ a subbundle of $E$.
Recall from subsection 2.3 that
if $\phi$ takes the form 
$\big(\begin{smallmatrix} a_{\alpha} & b_{\alpha}\\
c_{\alpha} & -a_\alpha \end{smallmatrix}\big)$
in local frames adapted to $L$, 
the matrix elements $c_\alpha$ glue into a section $c \equiv c_{\bfx}(\phi)$ of $KL^{-2} \Lambda$,
where $\bfx \in H^1(L^2 \Lambda^{-1})$ is defined by the embedding $L \hookrightarrow E$.
Let $\bu = \sum_{i=1}^m u_i$ be the zero divisor of $c$.
At each $u_i$, equation (\ref{spectral-curve-eqn}) for the spectral curve $S_q \overset{\pi}{\rightarrow} X$ reduces to $v^2 - a(u_i)^2 = 0$.
If $u_i$ is not a branch point, then the two points in $\pi^{-1}(u_i)$ are unambiguously labeled by $v = \pm a(u_i)$;
in this case let $\tilde{u}_i$ be the point defined by $v = a(u_i)$.
If $u_i$ is a branch point then let $\tilde{u}_i$ be the corresponding ramification point.
We define 
\begin{align}\label{BA_divisor_explicit_def}
&\tbu \coloneqq \sum_{i=1}^m \tilde{u}_i,
&m = N +g-1 = 2g-2 + s_d.
\end{align}
Clearly $\tbu$ is dependent only on the data 
$\left( L \hookrightarrow E, \phi \right)$
up to the scaling of the embedding $L \hookrightarrow E$.
We say $\tbu$ is the \textit{Baker-Akhiezer (BA) divisor} of this data.
We will write $\tbu = \tbu\left(L \hookrightarrow E, \phi \right)$ 
when we want to emphasize this dependence, otherwise we will simplify the notation. 
Inspired by \cite{Hit87a} and \cite{HH21}, in definition \ref{def-BA-divisor-formal} we will characterize these divisors in an invariant way and including the case where the injection $L \rightarrow E$ has zeroes and hence does not define a subbundle.

\begin{remark}
	As $q = \det(\phi)$ has only simple zeroes, if a branch point of $S_q \overset{\pi}{\rightarrow} X$ is contained in $\divisor(c)$, it must have multiplicity $1$.
	The corresponding ramification point has multiplicity $1$ in $\tbu$.
	Since the pull-back to $S$ of a branch point on $X$, regarded as a divisor on $X$, takes multiplicity into account and so has multiplicity $2$, by construction $\tbu$ contains no part equal to the pull-back of a divisor on $X$.
	\label{rem-BA-no-pull-back}
\end{remark}

%%%%%%%%%%%%%%%%%%%%%%%%%%%%
\paragraph{Eigen-line bundles in terms of BA-divisors.}
The following proposition clarifies the link between the 
characterisation of the integrability in terms
of the eigen-line bundle $\mathcal{L}$
to the notion of BA-divisor used in this paper. 

\begin{proposition}\label{prop_D_versus_L}
Let $\left(E, \phi \right)$ be an $SL_2(\bC)$-Higgs bundle with associated non-degenerate spectral curve $S \overset{\pi}{\rightarrow} X$.
Let $L$ be a subbundle of $E$ and $\tbu$ the BA-divisor of the data $(L \hookrightarrow E, \phi)$.
Then
\begin{align}\label{D-versus-L}
    &\mathcal{L} \cong \pi^\ast\left( K^{-1} L \right) \otimes \mathcal{O}_{S} \left( \sigma(\tbu) \right),
    &\sigma^\ast(\mathcal{L}) \cong \pi^\ast\left( K^{-1} L \right) \otimes \mathcal{O}_{S} \left( \tbu \right).
\end{align} 
\end{proposition}

We note that a similar result appears in \cite{HH21}
after Proposition 5.17 in loc.cit.. 
We will here give a 
self-contained elementary proof. \\[-1ex]

\begin{proof}
We shall use a cover of $X$ formed by open sets $U_{\al}$, together with corresponding 
local trivialisations allowing us to 
represent the
transition functions  on
$U_\alpha\cap U_{\beta}$
and the Higgs fields on $U_{\alpha}$
as matrix-valued functions $g_{\alpha\beta}$ and $\vf_{\alpha}$,
respectively.
	In local frames of $\pi^\ast(E)$
 adapted to the subbundle $\pi^\ast(L)$, we shall use the notations
\[  \pi^{\ast}\vf_\al=\bigg(\begin{matrix} a_\al & b_\al\\
c_{\al} & -a_{\al}
\end{matrix}\bigg),\qquad
\pi^{\ast}g_{\al\be}^{}=
\bigg(\begin{matrix} 
 l_{\al\be}^{} & l_{\al\be}^{}\ep_{\al\be}^{} \\ 0 & l_{\al\be}^{-1}\la_{\al\be}^{} 
\end{matrix}\bigg),
 \]
 for the pull-backs of Higgs fields and
 transition functions.

 The key observation to be made at this point
 is that local sections of the form 
	$ \chi_\al^{}=\big(\begin{smallmatrix} v_{\al} + a_{\al} \\ c_{\al} \end{smallmatrix}\big) $
	are eigen-vectors of $\pi^\ast \vf_\al$ with eigen-value $v_{\al}$, and hence are local sections of the eigen-line bundle $\mathcal{L} \hookrightarrow \pi^\ast\left( E \right)$.
 The proposition will be proven by using the local sections $\chi_{\al}$ to describe $\mathcal{L}$ in terms of a set of transition functions associated to an explicit cover. 

 In the following we shall often simplify the notations by not indicating pull-backs under the projection $\pi$.
We may  note that the transformation law 
$
\vf_{\al}^{}=g_{\al\be}^{}\cdot\vf_{\be}^{}\cdot
g_{\al\be}^{-1}$
implies the relations 
\[
c_{\al}^{}=k_{\al\be}^{}l_{\al\be}^{-2}\la_{\al\be}^{}c_{\be}^{},
\qquad
a_{\al}=k_{\al\be}^{}(a_{\be}+\ep_{\al\be}c_\be),
\]
 where $k_{\alpha \beta}$ is the transition function of the canonical bundle $K$,
allowing us to calculate
\begin{equation}
g_{\al\be}^{}\cdot
\bigg(\begin{matrix}  v_\be+a_\be \\ c_\be \end{matrix}\bigg)=
l_{\al\be}\bigg(\begin{matrix}  v_\be+a_\be+\ep_{\al\be}c_\be \\
 l_{\al\be}^{-2}\la_{\al\be}^{}c_\be^{} \end{matrix}\bigg)=
 k_{\be\al}^{}l_{\be\al}^{-1}
 \bigg(\begin{matrix} v_\al+a_\al \\ c_\al \end{matrix}\bigg).
\end{equation} 
This shows that,
when transiting between neighborhoods where $\chi_\alpha$ and $\chi_\beta$
are nowhere-vanishing and can serve as local generators of $\mathcal{L}$,
the transition functions of $\mathcal{L}$ are those of $\pi^\ast(K^{-1} L)$.

For $U_\alpha$ containing $u_i$,
the local section $\chi_{\alpha}$ 
vanishes at $\sigma(\tilde{u}_i)$
with the order equal to the multiplicity of $u_i$ in $\bu$.
Indeed, if $u_i \in U_\alpha$ is not a branch point,
then it is clear that this order is equal to the order of vanishing of $c_\alpha$ at $u_i$. 
If $u_i \in U_\alpha$ is a branch point and has multiplicity one in $\bu$,
then $\sigma(\tilde{u}_i)$ is a simple zero of 
$v_\alpha + a_\alpha$.
This explains the correction $\mathcal{O}_S\left( \sigma(\tbu) \right)$ to $\pi^\ast \left( L K^{-1} \right)$ in (\ref{D-versus-L}).
\end{proof}

\begin{remark}
	The vector-valued function $\chi_{\al}=(\begin{smallmatrix}
		v_\alpha + a_\alpha \\ c_\alpha
	\end{smallmatrix}\big)$
	plays a role analogous to the Baker-Akhiezer functions in the integrable system literature \cite{Babelon}
	and since it vanishes precisely at $\tbu$.
 This has motivated us to propose the 
	 terminology Baker-Akhiezer divisor 
    for the divisor $\tbu$ defined in this 
    section.
\end{remark}

\begin{example}\label{ex-prym-cap-jac}
Let $q$ be a non-degenerate quadratic differential.
Recall that the intersection of the Hitchin fiber $h^{-1}(q)$ with the Hitchin section corresponding to the spin structure $K^{1/2}$ is 
the isomorphism class of 
$$(E_0, \phi_q) = \bigg( K^{1/2} \oplus K^{-1/2}, \bigg(\begin{matrix} 0 & -q \\ 1 & 0 \end{matrix}\bigg) \bigg).$$
The BA-divisors of the data defined by this Higgs bundle and taking $K^{1/2}$ and $K^{-1/2}$ as subbundles are respectively 
the trivial divisor and the ramification divisor on $S_q \overset{\pi_q}{\rightarrow} X$.
It follows from (\ref{D-versus-L}) that the eigen-line bundle is isomorphic to  $\pi_q^\ast( K^{-1/2})$ either way.  
\end{example}

\begin{remark}
The usage of these divisors is not entirely new.
Hitchin after theorem 8.1 in his original paper \cite{Hit87a} already characterized Higgs bundles with underlying unstable bundles in terms of these divisors,  
and the recent work of Hausel-Hitchin \cite{HH21} also made extensive use of them in particular in the analysis of \textit{very stable Higgs bundles}.
\end{remark}

%%%%%%%%%%%%%%%%%%%%%%%%%%%%%%%%
\paragraph{Formal definition of BA-divisors.}
We now give an invariant and slightly more general definition of BA-divisors. This characterization of these divisors has featured in \cite{Hit87a} \cite{HH21}. 

Suppose $\left(E, \phi \right)$ is a Higgs bundle with non-degenerate spectral curve $S \overset{\pi}{\rightarrow} X$.
The eigen-line bundle $\mathcal{L}$ of $\left(E,\phi\right)$ is a subbundle of $\pi^\ast\left( E \right)$ and hence defines an extension
\begin{align}
	0 \rightarrow \mathcal{L} \rightarrow \pi^\ast\left(E \right) \rightarrow \mathcal{L}^{-1} \pi^\ast\left( \Lambda \right) \rightarrow 0.
	\label{extension_eigen_line}
\end{align}
Let $L \rightarrow E$ be an injection which possibly has zeroes;
we will write ``$L \hookrightarrow E$'' if the injection has no zero, i.e. it is an embedding that makes $L$ into a subbundle of $E$.
Consider the composition 
\begin{align}
	\pi^\ast\left( L \right) \rightarrow \pi^\ast\left(E \right) \rightarrow  \mathcal{L}^{-1} \pi^\ast\left( \Lambda \right).
	\label{BA-div-def-composition}
\end{align}
The zero divisor of this composition consists of the pull-back of the zero divisor of $L \rightarrow E$ and the points where $\pi^\ast(L)$ coincide with $\mathcal{L}$ as subbundles of $\pi^\ast(E)$.
It is rather straightforward to show that 
in case $L$ is a subbundle of $E$, 
this zero divisor gives the BA-divisor defined as in \eqref{BA_divisor_explicit_def}.

\begin{definition}
	Given a Higgs bundle $(E,\phi)$ with non-degenerate spectral curve $S \overset{\pi}{\rightarrow} X$, a line bundle $L$ with an injection $L \rightarrow E$,
	\textit{the BA-divisor} associated to these data is the zero divisor of the composition $\pi^\ast\left( L \right) \rightarrow \pi^\ast\left(E \right) \rightarrow  \mathcal{L}^{-1} \pi^\ast\left( \Lambda \right)$.
	\label{def-BA-divisor-formal}
\end{definition}

The cases where $L \rightarrow E$ has zeroes is a straightforward generalization of \eqref{BA_divisor_explicit_def}.
Indeed, if $f: L \rightarrow E$ has $D$ as its zero divisor, then there exists a subbundle $L_D \hookrightarrow E$,
where $L_D \coloneqq L \otimes \mathcal{O}_X(D)$, 
such that its composition with the canonical injection of sheaves $L \rightarrow L_D$ is $f$.
The BA-divisors of $(L \rightarrow E,\phi)$ is equal to the pull-back of $D$ 
plus the BA-divisor of $( L_D \hookrightarrow E, \phi )$, 
with the latter containing no part equal to the pull-back of a divisor on $X$ (cf. remark \ref{rem-BA-no-pull-back}).
This observation combined with proposition \ref{prop_D_versus_L}
leads to the following proposition.

\begin{proposition}  \label{prop-basic-properties-BA-divisor}
	Let $\tbu$ be the BA-divisor of $\left( L \rightarrow E, \phi \right)$ on a non-degenerate spectral curve $S \overset{\pi}{\rightarrow} X$,
 and $\bu$ its projection to $X$.
	Then
	\begin{enumerate}
		\item $\tbu$ contains  $\pi^{\ast}( D)$ for some effective divisor $D$ on $X$ if and only if $L \rightarrow E$ vanishes at $D$, counted with multiplicity. In particular, $\tbu$ contains no part equal to the pull-back of a divisor on $X$ if and only if $L$ is a subbundle of $E$, 
		and in this case $\tbu$ is given by (\ref{BA_divisor_explicit_def}); %%%
		\item the eigen-line bundle $\mathcal{L}$ of $\left(E,\phi \right)$ is isomorphic to $\pi^\ast\left(L K^{-1} \right) \otimes \mathcal{O}_S(\sigma(\tbu))$; %%%
		\item $\tbu$ satisfies $\mathcal{O}_X(\bu) \cong KL^{-2}\Lambda$. %%%%%
	\end{enumerate}
\end{proposition}

%%%%%%%%%%%%%%%%%%%%%%%%%%%%%%
\paragraph{The map Separation of Variables.}
An effective divisor of degree $m$ on a spectral curve defines a point 
in the symmetric product $(T^\ast X)^{[m]}$.
On the other hand, recall from Proposition \ref{prop-coordinate-Higgs}
and Remark \ref{rmk-forget-phi-}
that for a Higgs bundle $(E,\phi)$ with $E$ stable 
and a point $\sfx \in \M$ that projects to $E$, 
the pull-back $I^\ast_{\sfx}(\phi) \in T^\ast\M$ 
and hence $i^\ast_{[\sfx]}(\phi) \in T^\ast \N$ captures the lower-triangular 
part of $\phi$ in local frames adapted to $(L \hookrightarrow E) \equiv \sfx$.
The following proposition states that the construction of BA-divisors,
which requires only this lower-triangular part of $\phi$,
defines rational maps from the zero level set $H^{-1}(0) \subset T^\ast \M$ 
and $T^\ast \N$ to $(T^\ast X)^{[m]}$ for $m = 2g-2 +\deg(\Lambda) - 2d$. 
It amounts to part (i) of the main Theorem \ref{intro-main-thm} of this paper.

\begin{proposition} (Part (i) of Theorem \ref{intro-main-thm})
There exist rational maps $\SoV'$ and $\SoV$,
\begin{center}
	\begin{tikzcd}
		&H^{-1}(0) \arrow[dashed]{r}{\SoV'} \arrow[dashed]{d}
		&(T^\ast X)^{[m]}  \\
		&T^\ast \N \arrow[dashed]{ur}[swap]{\SoV}
		&
	\end{tikzcd}
\end{center}
where $m = 2g-2 +\deg(\Lambda) - 2d$,
such that if $\tbu$ is the BA-divisor of some triple 
$(L \hookrightarrow E, \phi)$
then the evaluation of $\SoV$ on 
$i^\ast_{(L \hookrightarrow E)}(\phi) \in T^\ast_{(L \hookrightarrow E)} \N$
is the point in $(T^\ast X)^{[m]}$ defined by $\tbu$.
\end{proposition}
\begin{proof}
It suffices to show that if $(E, \phi_1)$ and $(E, \phi_2)$ are two distinct 
Higgs bundles with $i^\ast_{(L \hookrightarrow E)}(\phi_1) = i^\ast_{(L\hookrightarrow E)}(\phi_2)$, 
then the corresponding BA-divisors $\tbu_1$ and $\tbu_2$ 
on two distinct spectral curves define the same point in $(T^\ast X)^{[m]}$. 
It follows from Corollary \ref{cor-kernel-I} that $\phi_1 - \phi_2$ is 
a nilpotent Higgs field on $E$. It is clear from the explicit construction 
\eqref{BA_divisor_explicit_def} of BA-divisors, however, that adding a 
nilpotent Higgs field does not change the point defined in $(T^\ast X)^{[m]}$.
We then can define a rational map $\SoV': H^{-1}(0) \dashrightarrow (T^\ast X)^{[m]}$.
As BA-divisors are invariant upon the scaling of the embedding 
$L \hookrightarrow E$, $\SoV'$ descends to a rational map
$\SoV: T^\ast \N \dashrightarrow (T^\ast X)^{[m]}$.
\end{proof}

Inspired by the literature on integrable systems \cite{Babelon, Skl89},
we call the resulting map the Separation of Variables map
and have denoted it by $\SoV$. 

%%%%%%%%%%%%%%%%%%%%%%%%%%%%%%%%%%%%%%%%%%%%%%%%%%%%%%%%%%%%%%%%%%%%%%
%%%%%%%%%%%%%%%%%%%%%%%%%%%%%%%%%%%%%%%%%%%%%%%%%%%%%%%%%%%%%%%%%%%%%%
%%%%%%%%%%%%%%%%%%%%%%%%%%%%%%%%%%%%%%%%%%%%%%%%%%%%%%%%%%%%%%%%%%%%%%
\subsection{Inverse construction}
The construction of BA-divisors can be inverted in the following sense.
For given data $(q, \tbu)$ where $q$ is a non-degenerate quadratic differential
and $\tbu$ an effective divisor on the corresponding spectral curve $S_q$,
one can find a stable Higgs bundle $(E, \phi)$ 
and an injection $L \rightarrow E$ that together induce $\tbu$ as a BA-divisor.

\begin{proposition}	\label{prop-inverse-SOV-fixed-determinant}
Let $\tbu$ be an effective divisor 
on the non-degenerate spectral curve $S_q \overset{\pi}{\rightarrow} X$ corresponding to a quadratic differential $q$,
and $L$ a line bundle satisfying $K L^{-2} \Lambda \cong \mathcal{O}_X(\bu)$
where $\bu = \pi(\tbu)$.
Then there exist a unique up to isomorphism stable Higgs bundle $(E,\phi)$ with $\det(E) \cong \Lambda$
and a unique up to scaling injection $L \rightarrow E$, 
such that $\tbu$ is the BA-divisor of $\left(L \rightarrow E, \phi \right)$. 
Furthermore, $L \rightarrow E$ vanishes at a divisor $D$ on $X$
if and only if $\pi^\ast(D) < \tbu$.
\end{proposition}
\begin{proof}
The proof follows directly 
from the formal construction of BA-divisor
and is based on the discussion following the proof of theorem 8.1 in Hitchin's original work \cite{Hit87a}.
From the spectral correspondence established by Hitchin
and proposition \ref{D-versus-L},
the Higgs bundle $(E, \phi)$ 
we seek is the direct image of $\pi^\ast(L) \otimes \mathcal{O}_{S_q}(\sigma(\tbu))$. 
In particular, the underlying bundle $E$ is isomorphic to
\begin{equation*}
    L \otimes \pi_\ast (\mathcal{O}_{S_q}(\sigma(\tbu)) ).
\end{equation*}
Hence there exists an injection $L \rightarrow E$
obtained by taking the direct image of the 
canonical section of $\mathcal{O}_{S_q}(\sigma(\tbu))$.
By construction, the BA-divisor of the data $\left(L \rightarrow E, \phi \right)$ is $\tbu$, 
and the injection $L \rightarrow E$ vanishes (counted with multiplicity)
only at a divisor $D$
that satisfies $\pi^\ast(D) < \tbu$.
\end{proof}

Let us define an isomorphism class $[ L \rightarrow E, \phi ]$
of the input data of Baker-Akhiezer divisors
by saying that two representative data 
are isomorphic if there are isomorphisms of the underlying bundles and line bundles
that commute with the injections and Higgs fields
\footnote{Since scalings are isomorphisms of line bundles,
scaling the injections from line bundles
to rank-2 bundles 
will define the same isomorphism class $[ L \rightarrow E, \phi ]$.}.
Clearly Baker-Akhiezer divisors defined by isomorphic data
coincide.
The following theorem summarizes the resulting relations between Higgs pairs and BA-divisors.

\begin{theorem}\label{thm-BA-1-1}
Let $q$ be a non-degenerate quadratic differential and 
$S_q \overset{\pi}{\rightarrow} X$ its corresponding spectral curve.
Then the construction of BA-divisors defines a bijection 
\begin{align*}
		\Bigg\{ [ L \rightarrow E, \phi ] \Bigg| 
			\begin{array}{l}
			 \det(E) = \Lambda,\\ \det(\phi) = q \\		 
		 	\end{array} \Bigg\}
		\longleftrightarrow 
		\Bigg\{ ( [L], \tbu ) \Bigg| \begin{array}{l}
				\tbu \text{ effective on } S_q,  \\
				K L^{-2}\Lambda \cong \mathcal{O}_X(\pi(\tbu) )
		\end{array} \Bigg\}.
\end{align*}
In the cases where $L$ is a sub-bundle of $E$ we obtain the following bijection 
\begin{align*}
		\Bigg\{ [ L \hookrightarrow E, \phi ] \Bigg|
			\begin{array}{l}
			 L \text{ a subbundle of } E, \\
			 \det(E) = \Lambda,\\ \det(\phi) = q \\		 
		 	\end{array} \Bigg\}
		\longleftrightarrow 
		\Bigg\{ ( [L], \tbu ) \Bigg| \begin{array}{l}
				\tbu \text{ effective on } S_q,\text{ contains} \\
				\text{no pull-back of divisors on } X, \\
				K L^{-2}\Lambda \cong \mathcal{O}_X(\pi(\tbu) )
   \end{array} \Bigg\}.
\end{align*}
The map induced by forgetting the subbundle, i.e. $[L \hookrightarrow E, \phi] \mapsto D$, is a $2^{2g}:1$ map. 
\end{theorem}
%\begin{proof}
	The property of the map to be bijective follows from the inverse construction of Baker-Akhiezer divisors (cf. Proposition \ref{prop-inverse-SOV-fixed-determinant}).
	The $2^{2g}$ covering property follows from the fact that
	twisting an input data by a square-root of $\mathcal{O}_X$ 
	defines another,
	which exhausts all possible input data of a Baker-Akhiezer divisor
	since its projection to $X$ 
	determines the line bundle $L$ up to such a twist. 
%\end{proof}

\begin{remark}\label{rem-anti-sym-D}
Even though a BA-divisor $D$ is defined depending on not only the Higgs bundle $(E, \phi)$ 
but also the choice of sub-line bundle $L \hookrightarrow E$,
its anti-symmetrisation $D - \sigma(D)$, 
where $\sigma$ is the involution of the spectral curve,
has the following invariant meaning.
Consider the identification $\prym(S_q) \simeq h^{-1}(q)$ defined by identifying 
the origin of $\prym(S_q)$ with
the intersection of $h^{-1}(q)$ with a Hitchin section.
Suppose that under this identification
$P \in \prym(S_q)$ corresponds to $(E, \phi) \in h^{-1}(q)$.
Then $D - \sigma(D)$
represent precisely the point $2P \in \prym(S_q)$,
and in particular does not depend on the choice of $L \hookrightarrow E$.
\end{remark}

\subsection{Explicit description of the inverse construction}

We will now discuss how to describe this inverse construction explicitly 
in terms of Higgs differentials (cf. Subsection 2.4) 
in the generic case of $s_d \leq g-1$.
Assume $\tbu$ contains no pull-back of an effective divisor on $X$ and has no point with multiplicity $> 1$.
Let us write points in $\tbu = \sum_{i=1}^m \tu_i$ 
as $\tu_i = (u_i, v_i)$.\\[1ex]
\noindent
{\bf Step 1: Construction of $L$ and the reference divisors.}
First, note that $\tbu$ determines $L$ up to tensoring with $2^{2g}$ square-roots of $\mathcal{O}_X$
via the constraint $K L^{-2} \Lambda \cong \mathcal{O}_X(\bu)$.
Let us choose and fix such a line bundle $L$.
Next, fix a reference divisor 
$\check{\bm{q}}=\sum_{j=1}^{g-d}\check{q}_j$ 
such that there is a unique divisor $\bq = \sum_{i=1}^g q_i$ 
satisfying $L \cong \mathcal{O}_X(\bq - \bqcheck)$. We shall furthermore 
choose a reference divisor $\bp = \sum_{r=1}^N p_i$ that together with $\bqcheck$ satisfies conditions \eqref{p-q-condition}.

\noindent
In the following steps, we discuss how one can construct a Higgs differential 
$ \big(\begin{smallmatrix} \phi_0 & \phi_- \\ \phi_+ & -\phi_0  \end{smallmatrix}\big) $
whose poles and zeroes are compatible with $L$ (cf. Proposition \ref{prop-Higgs-abelian-diff})
and is such that
\begin{align}\label{inverse-BA-linear}
    &\phi_0(u_i) = v_i,
    &i = 1, ..., m.
\end{align}

\noindent
{\bf Step 2: Construction of $\phi_+$.}
Knowing the poles and zeroes of this abelian differential (cf. Proposition \ref{prop-Higgs-abelian-diff}),
we can determine it up to a scaling factor and express it by an explicit formula (see Appendix B)
\begin{equation}\label{phi+-prime-form}
\phi_+(x) = u_0 \frac{\prod_{i=1}^g E(x, q_i(\bm{u}) )^2 \prod_{k=1}^{N + g-1} E(x, u_k)}{(E(x,\check{q}_0))^{\mathrm{deg(\Lambda)}}\prod_{j=1}^{g-d} E(x,\check{q}_k)^2} (\si(x))^{2}.
\end{equation}
Here $u_0 \in \bC^\ast$ is a scaling factor,
and $E(p,q)$ is the prime form on $\tilde{X} \times \tilde{X}$,
where $\tilde{X}$ is a fundamental domain of $X$ obtained by cutting along a basis of canonical cycles. The definition 
of $\si(x)$ can be found in Appendix B.
We may then define $k_r \coloneqq \phi_+(p_r)$ for $r = 1, ..., N$.\\[1ex]
\noindent
{\bf Step 3: Construction of $\phi_0$.} 
We will use \eqref{abelian-diff-explicit} as an ansatz for $\phi_0$,
which is now a function of $\bmx = (x_1, ..., x_N)$ and $\bkap = (\kappa_1, ..., \kappa_g)$.
Together with the Serre duality constraint \eqref{Serre-duality-constraint},
$\sum_{r=1}^N k_r x_r = 0$,
conditions \eqref{inverse-BA-linear}
now translate to a non-homogeneous linear system on $\bmx$ and $\bkap$,
\begin{align}\label{inverse-BA-linear-2}
v_j 
= - \sum_{r=2}^N  k_r  \bigg(  \omega_{p_r - p_1}(u_j) 
-  \sum_{i, j=1}^g  \omega_{p_r - p_1}(q_j) \big( dA^{-1}\mid_{\bla} \big)_{ij} \omega_i(u_j) \bigg) x_r  
- \frac{1}{2} \sum_{i=1}^g   \omega_i(u_j) \kappa_i, 
\end{align}
for $j = 1, ..., N + g-1$. Unique solutions to this system of equations will exist for 
generic choices of the divisors $\tbu$ and $\bm{p}$. As a change of $\bm{p}$ 
will induce a redefinition of the coordinates $\bm{x}$, one expects that
for all $\tbu$ there exist solutions which are well-defined in generic systems of coordinates $\bm{x}$
defined by the divisors $\bm{p}$.
Inserting the solution 
$(\bm{x},\bm{\kappa})$
into (\ref{abelian-diff-explicit})
determines a unique Abelian differential 
$\phi_0$. 
We may then recall that the data $(\bq,\bm{x})$ determines 
a bundle $E_{\bq, \bmx}$ (cf. \eqref{extension-data-1}).\\[1ex]
\noindent
{\bf Step 4: Construction of $\phi_-$.}
We define 
\begin{equation}\label{phiminus-q}
\phi_- = (-q - \phi_0^2)/\phi_+.
\end{equation}

The assumption that $\tilde{u}_i\in S_q$ implies 
that $q(u_i) + \phi_0^2(u_i)=0$, ensuring that $\phi_-$ is holomorphic at $u_i$.

\begin{remark}\label{rmk-step4}
As noted above, one may represent the initial data $(q, \tbu)$ of
this construction equivalently in terms of a quadratic differential $q$,
together with a point 
$\bm{p} = [p_1, ..., p_m] \in (T^\ast X)^{[m]}$ such that
the spectral curve $S_q$ passes through $p_1$, ..., $p_m$ and the induced
effective divisor is $\tbu$. One may observe that the input data 
needed to carry out Steps 1-3 above only involve the point $\bm{p}\in (T^\ast X)^{[m]}$,
and that the quadratic differential $q$ only enters in the construction of $\phi_-$ in Step 4. 
A spectral curve $S_q$ passing through $p_1$, ..., $p_m$ will exist for generic $\bm{p}$.
It will be unique if $s_d = g-1$. We will discuss in Subsection \ref{specloc} below what 
happens if there does not exist a spectral curve $S_q$ passing through $p_1$, ..., $p_m$.
\end{remark}

Let us now consider the following subset of $(T^\ast X)^{[m]}$,
\begin{equation*}
    (T^\ast X)^{[m]}_s 
    = \left\{ [p_1, ..., p_m] \in (T^\ast X)^{[m]} \hspace{2pt} \bigg|
    \begin{array}{l}
    p_i \neq p_j \text{ for } i \neq j, 
    \exists \text{ non-degenerate} \\ \text{spectral curve passing through } p_1, \dots, p_m 
    \end{array}
    \right\}.
\end{equation*}

\begin{proposition}
The restriction of $\SoV: T^\ast_s\N \rightarrow (T^\ast X)^{[m]}$ 
to a sufficiently small neighborhood of a 
generic point in $T^\ast_s\N$ is invertible.
\end{proposition}
\begin{proof}
Let $\bm{p} = [p_1, ..., p_m] $ be a point in $(T^\ast X)^{[m]}_s$, 
and choose a quadratic differential $q$ such that $S_q$ passes 
through $p_1, \dots, p_m$.
We then can construct from $(q, \bm{p})$ 
the data $(L \hookrightarrow E, \phi_q)$ 
such that its BA divisor would be the divisor on $S_q$ defined by $\bm{p}$. 
By Remarks \ref{rmk-forget-phi-} and \ref{rmk-step4}, if $q'$ 
is another quadratic differential with $S_{q'}$ 
passing through $p_1, \dots, p_m$,
the pull-back of the induced data $(E, \phi_{q'})$ and $(E, \phi_{q})$ 
to $T^\ast \N$ would coincide. 
This defines a (non-surjective) map 
$(T^\ast X)^{[m]}_s \rightarrow T^\ast \N$ 
which is a partial inverse of $\SoV$. 
There are in total $2^{2g}$ such maps, labelled by points in the fiber 
of a point in $(T^\ast X)^{[m]}_s$ along $\SoV$.
The images of these maps are distinct, 
and their union gives the preimage of $(T^\ast X)^{[m]}_s$ via $\SoV$ 
and is dense in $T^\ast \N$.
\end{proof}

%%%%%%%%%%%%%%%%%%%%%%%%%%%%%%%%%%%%%%%%%%%%%%%%%%%%

\subsection{Special loci}\label{specloc}

In the description in terms of BA-divisors there exist two types of special loci. We shall here describe these loci 
both in terms of conditions on the BA-divisors, and in terms of the corresponding properties of the Higgs pairs. 
This will in particular yield an alternative description 
of parts of the wobbly loci.

\paragraph{BA-divisors and stratification.}
Let $\tbu = \sum_{i=1}^m \tu_i$ be an effective divisor on a non-degenerate spectral curve $S \overset{\pi}{\rightarrow} X$
that contains no pull-back of effective divisors on $X$.
Let us now deform $\tbu$ by sending, say, $\tu_1$ and $\tu_2$ 
to form the pull-back $\pi^{-1}(p)$ of a point $p \in X$.
We can choose a corresponding family of data $(L \hookrightarrow E, \phi)$
that admit these divisors as BA-divisors. 
By Theorem \ref{thm-BA-1-1}, we can associate to the limit of this family 
the divisor $\tbu' = \sum_{i=3}^m \tu_i$ of degree $m-2$, and 
a sub-bundle $L'$ having degree generically increased by one compared to $L$.
In such cases we have thereby defined a family of Higgs bundles
that limits to a lower stratum.
A stronger statement holds for $m < 2g-2$ 
as the underlying bundles are unstable and admit unique maximal destabilising subbundles:
in this case there is no doubt that 
we have defined a family of Higgs bundles limiting to a lower stratum.

In the forthcoming papers \cite{D24, DFT}, 
we will discuss how apparent singularities of projective connections
are the natural analogues of the projection $\bu$ to $X$ of BA-divisors,
and how this analogy extends to the degeneration phenomenon above.

%%%%%%%%%%%%%%%%%%%%%%%%%%%%%%%%%%
\paragraph{Points in $(T^\ast X)^{[m]}$ not admitting a spectral curve.}
As mentioned above, there can exist points $\bm{p}=(p_1,\dots,p_{2g-2+s_d})\in (T^\ast X)^{[m]}$ such that there does not 
exist a spectral curve passing through all $p_r$.
We shall here elaborate on the nature of these special loci.
To begin, let us note that (\ref{phiminus-q}) establishes a one-to-one correspondence between
(i) the off-diagonal components $\phi_-$ of the Higgs differentials associated to $(q,\tilde{\bm{u}})$ by the construction above, 
and (ii) quadratic differentials $q$ satisfying 
\begin{equation}\label{qveqns}
q(u_r)+v_r^2=0,\quad v_r=-\phi_0(u_r),\qquad\text{for}\qquad r=1,\dots,2g-2+s_d.
\end{equation}
Equations \rf{qveqns} can be regarded as a system of linear equations for
$q\in H^0(K^2)$, with inhomogeneous term $v_r^2$.
It follows from equation (\ref{phiminus-q}) that Abelian differentials $\phi_-$ satisfying 
the conditions stated in Proposition \ref{prop-Higgs-abelian-diff} will exist if and only if there exist solutions to (\ref{qveqns}). 

This observation is of particular interest in the case $s_d = g-1$, where
equations (\ref{qveqns}) will generically determine $q$ uniquely, leading to a unique determination of $\phi_-$ via (\ref{phiminus-q}).
However, as we had discussed in Section \ref{SsecHiggswobbly}, the Higgs differential associated wobbly bundles admit more freedom if the diagonal 
elements $\phi_0$ are contained in certain subspaces. Restriction to this subspace
ensures existence of $\phi_-$, and therefore of a quadratic differential solving \rf{qveqns}.
The additional freedom for the choice of the off-diagonal matrix element $\phi_-$ existing in this case
is directly related to the existence of  nilpotent Higgs fields. 
We conclude that BA-divisors associated to 
maximally stable wobbly 
bundles are characterised by the existence of a non-trivial space of solutions $q$ to the system of equations (\ref{qveqns}).

These observations can be generalised to the case $s_d < g-1$ as follows.
For an effective divisor $D$ on $X$ let
$$ Q_{-D} = \{0\} \cup \{ q \in H^0(K^2 ) \mid \divisor(q) \geq D  \} \subset H^0(K^2 ).$$
Note that for a generic divisor $D$,
$$ \dim Q_{-D} =
\begin{cases}
3g-3 - \deg(D) & \text{ for } \deg(D) < 3g-3, \\
0 & \text{ for } \deg(D) \geq 3g-3.
\end{cases} $$
We will say that an effective divisor $D$ on $X$ is $Q$-generic 
if the space $Q_D$ of quadratic differentials whose zero divisors
are bounded below by $D$ 
has the minimal dimension
$$ \dim_m Q_{D} =
\begin{cases}
3g-3 - \deg(D) & \text{ for } \deg(D) < 3g-3, \\
0 & \text{ for } \deg(D) \geq 3g-3;
\end{cases} $$
otherwise, if $\dim Q_D > \dim_m Q_D$ we say that $D$ is $Q$-special.

Multiplying sections of $H^0(KL^{2}\Lambda^{-1})$ and $H^0(KL^{-2}\Lambda)$
yields quadratic differentials. Elements of $Q_{\bm{u}}$ are obtained, in particular,  by multiplying
elements of $H^0(KL^{2}\Lambda^{-1})$ defining 
nilpotent Higgs fields on wobbly bundles with sections of 
$H^0(KL^{-2}\Lambda)$ defined by Abelian differentials $\phi_+$ 
with zero divisor $\bm{u}$.
We may recall from Remark \ref{rm-extra-nilpotent} that, 
for $0 < s_d \leq g-1$, 
$ h^0(KL^{2}\Lambda^{-1}) = g-1 - s_d + h^0(L^{-2}\Lambda)$.  
Hence for $2g-2 < \deg(\bu) \leq 3g-3$, 
we find that $\bu$ is $Q$-generic
if and only if $h^0(L^{-2}\Lambda) = 0$.
We summarise the conclusions in the following proposition.

\begin{proposition}\label{prop-wobbly-Q-generic}
Let $\bu$ be an effective divisor of degree $\leq 3g-3$ on $X$,
$\tbu$ a divisor 
on a non-degenerate spectral curve 
$S \overset{\pi}{\rightarrow} X$
that projects to $\bu$,
and let $(L \hookrightarrow E, \phi)$ be a Higgs bundle 
having $\tbu$ as its BA-divisor. 
Suppose $\bu$ is $Q$-special. 
Then $E$ is a wobbly bundle, 
on which the space $H^0(KL^{2}\Lambda^{-1})$ 
of nilpotent Higgs fields with kernel $L$ 
has dimension larger than the minimal value $g-1 - s_d$. 
\end{proposition}

\begin{remark}
Varying the data (spectral curves, effective divisors on them)
such that the projection are $Q$-special divisors
and taking closure of the loci of bundles $E$ in Proposition \ref{prop-wobbly-Q-generic},
we will obtain one component of the wobbly loci. 
In fact, Proposition \ref{prop-wobbly-Q-generic} 
still holds in the range $\deg(\bu) \leq 4g-4$, 
and by expanding the range of $\deg(\bu)$ 
we would obtain all components of the wobbly loci  
described by Pal-Pauly \cite{Pal-Pauly} (cf. \cite{D-wobbly}).

One interesting question would be: can we still produce the same wobbly loci 
if we fix a generic spectral curve? 
The answer would be positive as a corollary of a theorem announced by Donagi-Pantev  (cf. the first theorem in Section 6 in \cite{DP09}).
\end{remark}

%%%%%%%%%%%%%%%%%%%%%%%%%%%%%%%%%%%%%%%%%%%%%%%%%%%%%%%%%%%%%%%%%%%%%%
%%%%%%%%%%%%%%%%%%%%%%%%%%%%%%%%%%%%%%%%%%%%%%%%%%%%%%%%%%%%%%%%%%%%%%
%%%%%%%%%%%%%%%%%%%%%%%%%%%%%%%%%%%%%%%%%%%%%%%%%%%%%%%%%%%%%%%%%%%%%%
\section{Separation of variables as local symplectomorphism}
Recall from the end of section 4.1 that the construction of Baker-Akhiezer divisors 
allows us to define a rational map 
\begin{equation*}
    \begin{tikzcd}
    &\SoV: T^\ast\N \arrow[dashed]{r} &(T^\ast X)^{[m]}
\end{tikzcd}
\end{equation*}
The main result of this section is the proof of the main theorem of this paper.

\begin{theorem}\label{sect5-main-thm} (Theorem \ref{intro-main-thm})
$\SoV$ is generically étale,
i.e. its derivative is an isomorphism, with generic $2^{2g}:1$ fibers.
Its restriction to the neighborhood of a generic point in 
$T^\ast\CN_{\Lambda,d}$ is a symplectomorphism.
\end{theorem}

Let us summarise the proof strategy of the local symplectomorphism
statement in Theorem \ref{sect5-main-thm}.
We will work with Poisson manifolds 
of which the domain and target spaces of $\SoV$ are Poisson reductions.
Consider the Poisson manifold
$$(T^\ast X)^{[m]} \times T^\ast \bC^\ast, $$
where the Poisson structure 
splits into the canonical one on $(T^\ast X)^{[m]}$ and
\begin{equation}\label{Poisson-Cstar}
    \{u_0, v_0\} \coloneqq u_0, 
\qquad \qquad (u_0, v_0)\in \bC^\ast\times \bC \simeq T^\ast\bC^\ast.    
\end{equation}
The $\bC^\ast$-action on the target space defined by scaling only $u_0$ has 
moment map
\begin{equation}\label{moment-map-v0}
    H_0: (T^\ast X)^{[m]} \times T^\ast \bC^\ast \longmapsto \bC, 
\qquad \quad
\left( (\bm{u}, \bm{v}), (u_0, v_0) \right) \longmapsto v_0    
\end{equation}
and $(T^\ast X)^{[m]}$ as its Poisson reduction,
$$ (T^\ast X)^{[m]} \simeq H_0^{-1}(0) /\bC^\ast.$$
We will construct a rational Poisson map (cf. \eqref{tsov})
\begin{equation*}
    \begin{tikzcd}
    &\tSoV: T^\ast\M \arrow[dashed]{r} &(T^\ast X)^{[m]} \times T^\ast \bC^\ast
\end{tikzcd}
\end{equation*}
upon choosing reference divisors on $X$
and Darboux coordinates on an open dense subset of $T^\ast\M$.
The map $\SoV$ is the induced map between the Poisson reductions
of the domain and target spaces of $\tSoV$.
It follows from the results of Marsden-Ratiu \cite{MR86} that
$\SoV$ is Poisson if $\tSoV$ is Poisson. 
The (generic) local symplectic property of $\SoV$ then follows,
as the Poisson structures on the domain and target spaces of $\SoV$ 
come from their canonical symplectic structures.

%%%%%%%%%%%%%%%%%%%%%%%%%%%%%%%%%%%%%%%%%%%%%%%%
\subsection{Preparations}
\paragraph{BA-divisors in Darboux coordinates.}
Upon fixing reference divisors on $X$, 
let $\mathcal{U} \subset T^\ast \M$ be induced open set 
on which we can use the Darboux coordinates $(\bm{\la},\bm{\kappa},\bm{x},\bm{k})$.
Recall also from Remark \ref{rm-dense-nbd} that $\mathcal{U}$ 
is dense in $T^\ast \M$. 
In the following, we shall recall the key equations \rf{phi+-prime-form} 
and \rf{inverse-BA-linear-2} defining the construction of BA-divisors in terms of Darboux coordinates $(\bm{\la},\bm{\kappa},\bm{x},\bm{k})$ 
and $(\bm{u}, \bm{v})$ on $(T^\ast X)^{[m]}$, 
and furthermore write them in a form 
that will be convenient later. 

To this aim let us note that the existence of 
an Abelian differential $\phi_+$ having divisor $\bu + 2\bq - 2\bqcheck$ implies that the Abel map
$A(\bu + 2\bq - 2\bqcheck)$ satisfies
\begin{equation}\label{Abelconst}
	\int_{x_0}^{\bu + 2\bq - 2\bqcheck} \omega_i
	= 2\lambda_i + \int_{x_0}^{\bu - 2\bqcheck} \omega_i=C_i,
\end{equation}
with $\bm{C}=(C_1,\dots,C_g)$ being the image of the Abel map of the canonical divisor. It follows that
\begin{equation}\label{identity-Abel}
	\frac{\del q_j(\bu)}{\del u_n} 
	= \sum_{i=1}^g \frac{\del q_j}{\del \lambda_i} \bigg\vert_{\bm{\lambda}(\bu)} \frac{\del \lambda_i}{\del u_n} \bigg\vert_{u_n} 
	= -\frac{1}{2} \sum_{i=1}^g \left( dA^{-1} \mid_{\bla} \right)_{ij} \omega_i(u_n),
\end{equation}
where by $\frac{\del q_j}{\del u_n}$ we mean the partial derivative
of the local coordinate $z_j$ of $q_j$ 
w.r.t. some local coordinate of $u_n$.
We have used the relation $\frac{\del \lambda_i}{\del u_n} = - \omega_i(u_n)/2$ 
following from \rf{Abelconst}. By using \rf{identity-Abel} one may rewrite 
\rf{inverse-BA-linear-2} in the form
\begin{align}\label{v_j-eqn}
	v_n 
	&=  - \sum_{r=2}^N  k_r  x_r \bigg(  \omega_{p_r - p_1}(u_n) + 2 \sum_{j=1}^g \omega_{p_r - p_1}(q_j) \frac{\del q_j}{\del u_n} \bigg) 
	- \frac{1}{2} \sum_{i=1}^g  \kappa_i \omega_i(u_n).  
\end{align}
We may furthermore observe that 
the normalized abelian differentials of the third kind with vanishing $A$-cycles can be written as 
$\omega_{p_+ - p_-}(x) = d_x \log E(p_+, x) - d_x \log E(p_-, x)$. 
Taking into account $\sum_rx_rk_r=0$, 
i.e. restricting to the intersection of $\mathcal{U}$ 
with the level set $H^{-1}(0)$ of the moment map $H = \bm{x}.\bm{k}$, 
one sees that 
\begin{equation}\label{single-valued}    
   \sum_{r=2}^N    \omega_{p_r - p_1}(u)k_r x_r= \sum_{r=1}^N d_u  \log E(p_r, u)k_r x_r.
\end{equation}
The main equations defining the 
construction of BA-divisors in terms of Darboux coordinates on 
$\mathcal{U}\cap H^{-1}(0) \subset T^\ast \M$ and $(T^\ast X)^{[m]}$
can therefore be represented in the following form
\begin{align}\label{k_r-eqn}
&k_r= u_0 \frac{\prod_{i=1}^g E(p_r, q_i(\bm{u}) )^2 \prod_{k=1}^{N + g-1} E(p_r, u_k)}{(E(p_r,\check{q}_0))^{\mathrm{deg(\Lambda)}}\prod_{j=1}^{g-d} E(p_r,\check{q}_k)^2} (\si(p_r))^{2},
\qquad 2\lambda_i + \int_{x_0}^{\bu - 2\bqcheck} \omega_i=C_i,\\
\label{v_n-eqn-prime}
	&v_n 
	=  - \sum_{r=1}^N  k_r  x_r \bigg(  \frac{\partial\log E(p_r, u_n)}{\partial u_n} 
 + 2 \sum_{j=1}^g \frac{\partial \log E(p_r, q_j)}{\partial q_j} \frac{\del q_j}{\del u_n} \bigg) 
	- \frac{1}{2} \sum_{i=1}^g  \kappa_i \omega_i(u_n), 
\end{align}
for $r=1,\dots,N$, $i=1,\dots,g$, and $n=1,\dots,2g-2+s_d$, which are to be 
supplemented by the equation $\sum_{r=1}^Nk_rx_r=0$, with $N=g-1+s_d$.

\paragraph{Construction of the map $\tSoV$.}
We may now note that although
the restriction to $\mathcal{U} \cap H^{-1}(0)$ has the effect of 
equating \eqref{v_j-eqn} with $\eqref{v_n-eqn-prime}$,
the expressions \eqref{k_r-eqn} and \eqref{v_n-eqn-prime}
define complex-valued functions $\left( u_n, v_n \right)_{n=1}^m$
on all of the open dense subset $\mathcal{U} \subset T^\ast \M$. 
Together with the assignment
$$ \mathcal{U} \longrightarrow  \bC^\ast\times \bC = T^\ast \bC^\ast,
\qquad \qquad 
(\bm{\la},\bm{\kappa}, \bm{x},\bm{k}) \longmapsto 
\left(u_0, H \right) = \left(u_0, -\bm{x}.\bm{k} \right)$$
where $u_0$ is the multiplicative factor in \eqref{k_r-eqn},
we can define a rational function 
\begin{equation}\label{tsov}
    \begin{tikzcd}
    &\tSoV: T^\ast\M \arrow[dashed]{r} &(T^\ast X)^{[m]} \times T^\ast \bC^\ast.
\end{tikzcd}
\end{equation}
By construction, $\tSoV$ restricts to a map 
between the level sets $H^{-1}(0)$ and $H_0^{-1}(0)$
of the moment maps (cf. \eqref{moment-map-v0}),
and in fact between the $\bC^\ast$-orbits of in these level sets.
The map $\SoV$ 
of our interest is then obtained by passing to the space of $\bC^\ast$-orbits
of these level sets.

\begin{equation}
\begin{tikzcd}
    &T^\ast\M \arrow[dashed]{r}{\tSoV} &(T^\ast X)^{[m]} \times T^\ast \bC^\ast \\
    &H^{-1}(0) \arrow[dashed]{r} \arrow[dashed]{d} \arrow[hook]{u}
    &(T^\ast X)^{[m]} \times \bC^\ast \arrow{d} \arrow[hook]{u} \\
    &T^\ast\N \arrow[dashed]{r}[swap]{\SoV} &(T^\ast X)^{[m]}
\end{tikzcd}    
\end{equation}

Note that a priori different choices of 
reference divisor on $X$ and the associated 
open dense subset $\mathcal{U} \subset T^\ast \M$ 
might induce rational maps 
$T^\ast\M \dashrightarrow (T^\ast X)^{[m]} \times T^\ast \bC^\ast$.
These maps however agree on the level set $H^{-1}(0)$ 
as they shall define BA-divisors, which are 
coordinate-independent constructions. 

\begin{remark}
Consider $\sfx = (L, \bfx) \in \M$
where $\bfx$ realizes an unstable bundle,
i.e. it defines a point in a sufficiently low secant variety of $\bP_L$ (cf. subsection 2.1).
Then although a cotangent vector $\xi\in T^\ast_\sfx \M$ 
is not the pull-back of a cotangent vector on $\Nstable$,
it nevertheless captures the lower-triangular parts of Higgs fields on this unstable bundles 
and defines BA-divisors
as long as $k(\xi) \neq 0$.
In particular, as long as $\xi$ is contained in the domain of 
$\tSoV$,
these BA-divisors define the same point in $(T^\ast X)^{[m]}_s$.
\end{remark}

%%%%%%%%%%%%%%%%%%%%%%%%%%%%%%%%%%%%%%%%%%%%%%%%%%%%%%%%%%%%%%%%%%%%%%
%%%%%%%%%%%%%%%%%%%%%%%%%%%%%%%%%%%%%%%%%%%%%%%%%%%%%%%%%%%%%%%%%%%%%%
%%%%%%%%%%%%%%%%%%%%%%%%%%%%%%%%%%%%%%%%%%%%%%%%%%%%%%%%%%%%%%%%%%%%%%
\subsection{Comparison of Poisson structures}
From now on, we will use the notations $u_0, \dots, u_m$ and $v_0, \dots, v_m$ 
for the coordinate expression of the map $\tSoV$ in \eqref{tsov}.
The following lemma will be the key to the proof of Theorem \ref{sect5-main-thm}.
\begin{lemma}\label{lemma-crucial}
Let $F$ be a complex-valued function on an open set in $T^\ast\M$ 
equipped with local Darboux coordinates $(\bla, \by, \bkap, \bmk)$ satisfying $\{k_r,F\}=0$, and $\{\la_\ell,F\}=0$.
Then 
\begin{equation}\label{eqn-lem-cruc} \frac{\del }{\del u_n} F(\bla(\bu), \bmk(\bu))
= - \left\{ v_n, F \right\}, \quad n=1,\dots,m,
\quad  u_0\frac{\del }{\del u_0} F(\bla(\bu), \bmk(\bu))=-\left\{ v_0, F \right\}.
\end{equation}
\end{lemma}
\noindent

%%%%%%%%%%%%%%%%%%%%%%%%%%%%%%
\paragraph{Proof of Lemma \ref{lemma-crucial}.}
We want to show that $\left\{ v_n, F \right\}$, with $v_n$ given in \rf{v_n-eqn-prime},
is equal to 
\begin{equation}\label{partialu_nF} -\frac{\del F}{\del u_n} 
=  - \sum_{r=1}^N \frac{\del k_r}{\del u_n} \frac{\del F}{\del k_r} - \sum_{i=1}^g \frac{\del \lambda_i}{\del u_n} \frac{\del F}{\del \lambda_i} 
= - \sum_{r=1}^N \frac{\del k_r}{\del u_n} \frac{\del F}{\del k_r} + \frac{1}{2} \sum_{i=1}^g \omega_i(u_n) \frac{\del F}{\del \lambda_i}.
\end{equation}
It follows from \rf{k_r-eqn} that 
\begin{equation*}
	\frac{1}{k_r}	\frac{\del k_r}{\del u_n}
	= \frac{\del \log E(p_r, u_n)}{\del u_n}  + 2 \sum_{i=1}^g \frac{\del \log E(p_r, q_i)}{\del q_i} \frac{\del q_i(u_n)}{\del u_n}.
\end{equation*}
Inserting this into equation \rf{partialu_nF}  yields
\begin{equation}\label{d_un-final}
	\frac{\del F}{\del u_n}
	= \sum_{r=1}^N  \bigg( \frac{\del \log E(p_r, u_n)}{\del u_n} 
	+ 2 \sum_{i=1}^g \frac{\del \log E(p_r, q_i)}{\del q_i} \frac{\del q_i}{\del u_n} 	\bigg) k_r \frac{\del F}{\del k_r}
	- \frac{1}{2} \sum_{i=1}^g \omega_i(u_n) \frac{\del F}{\del \lambda_i}.
\end{equation}
We may next insert the relations
\begin{align*}
	&
	 \frac{\del F}{\del k_r}=\big\{\,x_r\,, \,F\, \big\},
	&
	 - \frac{\del F}{\del \lambda_i}=\big\{\,\kappa_i\,,\, F\,\big\}.
\end{align*}
Comparison with \rf{v_n-eqn-prime}
yields the first relation in \rf{eqn-lem-cruc}. The second one follows easily 
by noting $u_0\frac{\partial k_r}{\partial u_0}=k_r$,
leading to $u_0\frac{\partial F}{\partial u_0}=\sum_r k_r\frac{\partial F}{\partial k_r}=\sum_r \{x_rk_r,F\}$.
\qed

%%%%%%%%%%%%%%%%%%%%%%%%%%%%%
\paragraph{Proof of Theorem \ref{sect5-main-thm}.}
To show generically étale property of $\SoV$,
it suffices to show that its image is dense, i.e. $\SoV$ is dominant. 
But one can observe this fact explicitly by noting that 
given a reduced point $P \in (T^\ast X)^{[m]}$ 
through which a smooth spectral curve passes through, 
we can find triples $(E, L, \phi)$ whose associated spectral curves 
are smooth and whose BA-divisors define such point in $(T^\ast X)^{[m]}$. 
The generic $2^{2g}:1$ fiber property follows from twisting with 
square-roots of the trival line bundle $\mathcal{O}_X$.

To show local symplectic property of $\SoV$, it suffices to show 
Poisson property of $\tSoV$ \cite{MR86}.
In other words, we need to show that the functions
$u_n$ and $v_m$ satisfy
\[
\begin{aligned}
&\{u_n, u_m \} = 0,\quad n,m=0,\dots,N,\\
&\{v_n, v_m \} = 0,\quad n,m=0,\dots,N.
\end{aligned},\qquad
\{u_m, v_n \} = \delta_{n,m},\quad 
\begin{aligned} 
&m=1,\dots,N,\\
&n=1,\dots,N,
\end{aligned}\quad
\{u_0,v_0\}=u_0.
\]
with respect to the canonical Poisson structure of $T^\ast\M$. 

The relations $\{u_n, u_m \} = 0$ follow from observation that $u_n = u_n(\bla, \bmk)$ 
while $\lambda_1, ..., \lambda_g$, $k_1, ..., k_N$ are Poisson commuting.
The relations for $\{u_m, v_n \}$ follow  by applying  Lemma \ref{lemma-crucial} to 
$F = u_m$. 

In order to prove $\{v_n, v_m \} = 0$, let us 
introduce the grading on the algebra of 
polynomial functions in variables $k_r$ and $\kappa_i$
assigning degree one to the generators $k_r$, $r=1,\dots,N$ and 
$\kappa_i$, $i=1,\dots,g$. As the Higgs field 
$\phi$ is homogeneous of degree one
with respect to this grading, it follows that  
the Poisson bracket $\{v_n,v_m\}$  must have the same 
property.
We may observe, on the other hand,
\begin{align*}
\big\{\{v_n,v_m\},f(\bm{u})\big\}&=\big\{v_n,\{v_m,f(\bm{u})\}\big\}-
\big\{v_m,\{v_n,f(\bm{u})\}\big\}\\
&=\big\{v_n,f_{u_m}(\bm{u})\big\}-
\big\{v_m,f_{u_n}(\bm{u})\big\}, \qquad \quad f_{u_m}(\bm{u}):=\frac{\pa}{\pa u_m}f(\bm{u})\\
&=\frac{\pa}{\pa u_n}f_{u_m}(\bm{u})-
\frac{\pa}{\pa u_n}f_{u_m}(\bm{u})=0
\end{align*}
As $\{v_n,v_m\}$ is homogeneous of degree one, the relation
$\big\{\{v_n,v_m\},f(\bm{u})\big\}=0$ is
enough to conclude that $\{v_n,v_m\}=0$.
\qed

\begin{remark}
It follows from the generically étale property of $\SoV$ that 
its image is dense, i.e. $\SoV$ is dominant. 
In fact, one can observe this fact explicitly by noting that 
given a reduced point $P \in (T^\ast X)^{[m]}$ 
through which a smooth spectral curve passes through, 
we can find triples $(E, L, \phi)$ whose associated spectral curves 
are smooth and whose BA-divisors define such point in $(T^\ast X)^{[m]}$. 
\end{remark}

%%%%%%%%%%%%%%%%%%%%%%%%%%%%%%%%%%%%
\paragraph{Poisson structure on bundles of quadratic differentials.}
The following discussion is particularly relevant for $0 < s_d < g-1$.
In this range, 
on one hand, a generic point in $\M$ realises a stable but not maximally stable bundle,
and on the other hand, a generic point in $(T^\ast X)^{[m]}$
does not fix a spectral curve since $m < 3g-3$.

We can supplement the data of spectral curves as follows.
Given a point $\bm{P} \in (T^\ast X)^{[m]}_s$,
denote by $Q^{\bm{P}}$ the set of spectral curves that pass through all $m$ points of $\bm{P}$.
Note that $Q^{\bm{P}}$ is an affine space modeled over the space
$Q_{-\pi(\bm{P})}$ of quadratic differentials
vanishing at the projection $\pi(\bm{P})$  to $X^{[m]}$ of $\bm{P}$. 
Consider the set 
$$ T^\ast\N^Q \coloneqq \{ (\xi, q) \mid \xi \in T^\ast\N, 
q \in Q^{\SoV(\xi)} \}. $$
Note that $T^\ast\N^Q$ is \textit{almost} a fiber bundle
over the domain of $\SoV$ in $T^\ast\N$, with the fiber over $\xi$ 
identified with $Q^{\SoV(\xi)}$.
This is strictly speaking not a fiber bundle since 
$$ \dim Q^{\SoV(\xi)} = h^0(KL^{2}\Lambda^{-1}) $$
depends on $h^0(L^{-2} \Lambda)$ and hence varies.
Imposing the condition $h^0(L^{-2} \Lambda) = 0$ (the generic case), 
we would get a fiber bundle over an open dense subset of $T^\ast\N$.
There is a natural Poisson structure on $T^\ast\N^Q$ 
in which functions that depend only on coordinates of the fibers 
$Q^{\SoV(\xi)}$ play a central role.

To supplement the data of spectral curves for the target space, let 
$$ (T^\ast X)^{[m]}_Q
\coloneqq \{ (\bm{P}, q) \mid \bm{P} \in (T^\ast X)^{[m]}_s, 
q \in Q^{\bm{P}} \},
$$
which is an affine bundle on $(T^\ast X)^{[m]}$ 
with fiber $Q^{\bm{P}}$ over $\bm{P}$.
Restricting to the open dense subset of $(T^\ast X)^{[m]}_s$
where the dimension of $Q^{\bm{P}}$ is constant, we can similarly 
define a Poisson structure on $(T^\ast X)^{[m]}_Q$.
The local rational symplectomorphism $\SoV$ then extends to a rational Poisson map
\begin{equation}
\begin{tikzcd}
    &T^\ast\N^Q \arrow[dashed]{r}{\SoV_Q} \arrow[dashed]{d} &(T^\ast X)^{[m]}_Q \arrow[dashed]{d}\\
    &T^\ast\N^s \arrow[dashed]{r}[swap]{\SoV} &(T^\ast X)^{[m]}
\end{tikzcd}
\end{equation}
which we have denoted by $\SoV_Q$.

%%%%%%%%%%%%%%%%%%%%%%%%%%%%%%%%%%%%%%%%%%%%%%%%%%%%%%%%%%%%
%%%%%%%%%%%%%%%%%%%%%%%%%%%%%%%%%%%%%%%%%%%%%%%%%%%%%%%%%%%%
\subsection{Local symplectomorphism to torus fibration}

To round off the picture, we shall here briefly discuss the relation to the torus fibration picture, and 
the extension to Higgs pairs with unstable 
underlying bundles. 

Let $S_q$ be a smooth spectral curve associated to $q \in H^0(X, K_X^2)$.
Let $\left( \tilde{\omega}_1, \dots, \tilde{\omega}_{4g-3} \right)$
be a basis of $H^0(S_q, K_{S_q})$, 
where $\tilde{\omega}_i$ is anti-invariant (invariant)
with respect to the involution $\sigma$,
i.e. $\sigma^\ast \tilde{\omega}_i = - \tilde{\omega}_i$
(respectively, $\sigma^\ast \tilde{\omega}_i = \tilde{\omega}_i$)
for $i = 1, \dots, 3g-3$ (respectively, $i = 3g-2, \dots, 4g-3)$.\footnote{
One can construct the $\sigma$-invariant holomorphic differentials by 
pulling back those from $X$ and the $\sigma$-anti-invariant ones e.g. by
$\tilde{\omega}_i \coloneqq \frac{\partial \lambda}{\partial H_i}$, 
where $\lambda$ is the canonical 1-form on $T^\ast X$ and $H_i$
is the Hitchin Hamiltonian defined by a choice of basis of $H^0(X, K_X^2)$.}
Consider the Abel map defined w.r.t. these holomorphic differentials,
$$A: \pic^0(S_q) \longrightarrow \jac(S_q), \qquad \qquad
\mathcal{O}_{S_q} \left(\sum_i n_i p_i \right) 
\longmapsto \left( \sum_i n_i\int^{p_i} \tilde{\omega}_j \right)_{j=1}^{4g-3}, 
$$
where, we recall, $\jac(S_q)$ is the complex torus $\simeq \bC^{4g-3}/\Gamma_{S_q}$ with $\Gamma_{S_q}$ being the image of $H_1(S_q, \mathbb{Z})$
via the integration defining the Abel map.
Due to the splitting of the basis of $H^0(S_q, K_{S_q})$ into the 
$\sigma$-anti-invariant and -invariant parts,
the restriction of the Abel map to the Prym variety
\begin{align*}
\prym(S_q) &= \{ \mathcal{L} \in \pic^0(S_q) 
\mid \mathcal{L} \otimes \sigma^\ast(\mathcal{L}) \simeq \mathcal{O}_{S_q} \} \\
&= \{ \mathcal{O}_{S_q}(\tbu - \sigma(\tbu)) \mid \tbu 
\text{ is an effective divisor on } S_q \}.
\end{align*}
has images 
$$ \left( \int^{\tbu}_{\sigma(\tbu)} \tilde{\omega}_1, \dots,
\int^{\tbu}_{\sigma(\tbu)} \tilde{\omega}_{3g-3}, 
0, \dots, 0\right); $$
here we have chosen the path from $\sigma(\tbu)$ to $\tbu$ such that
its image in $X$ represents $0 \in H_1(X, \mathbb{Z})$
(cf. Section 3.4 of \cite{Witten}).
It follows that one can identify the Prym variety with the complex torus 
$\bC^{3g-3}/\Gamma_0$, where $\Gamma_0$ is the image of the kernel of the natural map 
$H_1(S_q, \mathbb{Z}) \rightarrow H_1(X, \mathbb{Z})$,
i.e. $\Gamma_0$ consists of images of 
cycles in $H_1(S_q, \mathbb{Z})$ 
that are anti-invariant under involution.
The Abel-Prym map
\begin{subequations}\label{Abel-Prym}
\begin{equation}
\mathrm{AP}: S_q^{[m]} \longrightarrow \prym(S_q), 
\qquad \qquad \tbu \longmapsto \mathcal{O}_{S_q}(\tbu - \sigma(\tbu)),
\end{equation} 
then, upon post-composing with the identification
$\prym(S_q) \simeq \bC^{3g-3}/\Gamma_0$, can be written in coordinates
\begin{equation}
\mathrm{AP}(\tbu) = 
\left( \int^{\tbu}_{\sigma(\tbu)} \tilde{\omega}_1, \dots,
\int^{\tbu}_{\sigma(\tbu)} \tilde{\omega}_{3g-3} \right).    
\end{equation}
\end{subequations}

Let $B_{\mathrm{reg}}$ be the regular locus of the Hitchin base $H^0(X, K_X^2)$ 
consisting of quadratic differentials with simple zeroes, 
and $\mathcal{T} \rightarrow B_{\mathrm{reg}}$ the corresponding torus fibration,
i.e. the fiber over $q \in B_{\mathrm{reg}}$ is $\prym(S_q)$.
The Abel-Prym map for a smooth spectral curve $S_q$ 
then extends to a rational map 
\begin{equation*}
\begin{tikzcd}
    &(T^\ast X)^{[3g-3]} \arrow[dashed]{rr}{\mathrm{AP}} \arrow[dashed]{rd} &
    &\mathcal{T} \arrow{ld} \\
    & &B_{\mathrm{reg}} &
\end{tikzcd}.
\end{equation*}
For each $S_q$, let $\{\al_1,\dots,\al_{3g-3};\be_1,\dots,\be_{3g-3}\}$ 
be a basis for $\Gamma_0$. 
The restriction to $S_q$ of the Liouville form $\la = y \hspace{1pt} dx$ on $T^\ast X$
defines a canonical holomorphic differential on $S_q$,
out of which we can define the periods $\bm{a}=\left(a_i \right)_{i=1}^{3g-3} 
= \big( \int_{\al_i}\la \mid_{S_q} \big)_{i=1}^{3g-3}$
which can serve as coordinates on $B_{\mathrm{reg}} \subset H^0(X, K_X^2)$.
One can define the $\sigma$-anti-invariant differentials by
$\tilde{\omega}_i = \frac{\partial}{\partial a_i} \lambda$. 
Let us denote by $\bm{\theta}(\tbu) = (\theta_1(\tbu), \dots, \theta_{3g-3}(\tbu))$ 
the evaluation of the Abel-Prym map \ref{Abel-Prym} w.r.t. these differentials.
Then there exists on $\mathcal{T}$ a natural symplectic structure
\begin{equation}
\Omega_{\mathcal{T}} = \frac{1}{2}\sum_{k=1}^{3g-3}da_k\wedge d\theta_k.    
\end{equation}
The factor $\frac{1}{2}$ comes in since for a BA-divisor $\tbu$, 
its anti-symmetrisation $\tbu - \sigma(\tbu)$ represents twice the point in 
the Prym variety that the Higgs bundles corresponds to 
(cf. Remark \ref{rem-anti-sym-D}).

Let $\mathcal{T}' \subset (T^\ast X)^{[3g-3]}$ be the loci at which the projection 
$(T^\ast X)^{[3g-3]} \dashrightarrow B_{\mathrm{reg}}$
is an actual morphism, 
i.e. it is the complement of the union of the loci defined 
by $Q$-special divisors on $X$ 
and the loci that define non-smooth spectral curves.

\begin{lemma}\label{Abel-symp} 
\begin{enumerate}[label=(\roman*)]
    \item The generic fiber of $\mathrm{AP}$ is $2^{g}:1$.
    \item The restriction to $\mathcal{T}'$ of $\mathrm{AP}$
is a local symplectomorphism.

\end{enumerate}
\end{lemma}

\begin{proof}
%\begin{enumerate}[label=(\roman*)]
(i) This follows from part (ii) of Proposition \ref{prop-fiber-max-line}.\\[1ex]
(ii) Let us consider 
\[\mathcal{W}({\bm{u}},\bm{a}):=\sum_{k=1}^{3g-3}\int^{\tilde{u}_k}_{\sigma(\tilde{u}_k)} \la
\]
where the integration path on $S_q$ is chosen such that 
its image to $X$ is homologically trivial.
This function clearly satisfies 
\[
\frac{\pa}{\pa u_k}\mathcal{W}({\bm{u}},\bm{a})= 2v_k,
\]
and furthermore 
\[
\frac{\pa}{\pa a_s}\mathcal{W}({\bm{u}},\bm{a})=
\sum_{k=1}^{3g-3}\int^{\tilde{u}_k}_{\sigma(\tilde{u}_k)} \omega_s=\theta_s,
\]
since we have, by construction, $\frac{\pa}{\pa a_s}\la=\omega_s$.
Using these two relations we may compute
\[ 
d\mathcal{W}({\bm{u}},\bm{a})=\sum_{k=1}^{3g-3}\theta_k da_k +
2\sum_{k=1}^{3g-3}v_k du_k,
\]
and $d^2=0$ implies the relation
\[
\sum_{k=1}^{3g-3}dv_k\wedge du_k =
-\frac{1}{2} d\sum_{k=1}^{3g-3}\theta_k da_k =
\frac{1}{2}\sum_{k=1}^{3g-3}da_k\wedge d\theta_k
= \Omega_{\mathcal{T}},
\] 
which proves our claim.
%\end{enumerate}
\end{proof}

In summary,
the results of this section can be directly applied to describe the symplectic structure of the image in 
$(T^\ast X)^{[3g-3]}$ of the Higgs pairs 
$(E,\phi)$ with $E$ maximally stable.  Having established the 
relation between the symplectic form on $T^\ast\mathcal{N}_\Lambda$ to the one on 
$(T^\ast X)^{[3g-3]}$ in Theorem \ref{sect5-main-thm} 
we may now use 
Lemma \ref{Abel-symp} to conclude that 
the symplectic form on $T^\ast\mathcal{N}_\Lambda$
coincides with the pull-back of 
$\Omega_{\mathcal{T}}$ to $T^\ast\mathcal{N}_\Lambda$.
Related results have been reported in \cite{Hur, GNR}.

So far we had often restricted attention to 
the parts of Hitchin's moduli spaces associated to
Higgs pairs $(E,\phi)$ with $E$ stable;
in particular, Lemma \ref{Abel-symp} works for maximally stable $E$.
We are now ready to lift this restriction, 
i.e. we consider $(T^\ast X)^{[m]}$ with $m < 3g-3$ ($E$ not maximally stable)
and, in particular, even $m \leq 2g-2$ ($E$ not stable).
In this case, 
we need to supplement the choice of spectral curves.
Consider the space $(T^\ast X)^{[m]}_Q$ constructed 
at the end of Section 5.2 and the diagram
\begin{equation}
\begin{tikzcd}
& &(T^\ast X)^{[m]}_Q \arrow[dashed]{dl} \arrow[dashed]{dr} \arrow[dashed]{rr}{\mathrm{AP}} 
& &\mathcal{T} \arrow{dl} \\
&(T^\ast X)^{[m]} & 
&B_{\mathrm{reg}} &
\end{tikzcd},
\end{equation}
where $\mathrm{AP}$ is the extension of a variant of the Abel-Prym map 
\eqref{Abel-Prym} for effective divisors of degree $m < 3g-3$.
The proof of Lemma \ref{Abel-symp} then carries out straightforwardly to show 
that the pull-back along $\mathrm{AP}$ of $\Omega_{\mathcal{T}}$
coincides with the pull-back of the canonical symplectic form 
on $(T^\ast X)^{[m]}$.

We further remark that 
the proof of Theorem \ref{sect5-main-thm} can then be adapted to
the strata in Hitchin moduli space associated to Higgs pairs 
$(E,\phi)$ with $E$ unstable. 
The moduli spaces $\M$ and $\N$ are respectively vector bundle 
and projective fiber bundle over appropriate Brill-Noether loci of 
the Picard components of $X$.
The Darboux coordinates $(\bm{\la},\bm{x},\bm{\kappa},\bm{k})$ of $T^\ast \M$
still allow us to define rational local symplectomorphism 
$\SoV: T^\ast \N \dashrightarrow (T^\ast X)^{[m]}$ for $m \leq 2g-2$.
Consequently, the symplectic structure of $T^\ast \N$
still coincides with the pull-back of $\Omega_{\mathcal{T}}$
via 
$$\mathrm{AP} \circ \SoV: T^\ast \N \dashrightarrow \mathcal{T}.$$
As the torus fibration with symplectic structure $\Omega_{\mathcal{T}}$  
represents a partial compactification of the cotangent bundles of the moduli spaces of stable bundles 
defined by adding 
strata associated to unstable bundles,
we thereby arrive at a
uniform and concrete treatment of all strata appearing in Hitchin's moduli spaces in terms of the Separation of Variables representation.

%%%%%%%%%%%%%%%%%%%%%%%%%%%%%%%%%%%%%%%%%%%%%%%%%%%%%%%%%%%%%%%%%%%%%%
%%%%%%%%%%%%%%%%%%%%%%%%%%%%%%%%%%%%%%%%%%%%%%%%%%%%%%%%%%%%%%%%%%%%%%
%%%%%%%%%%%%%%%%%%%%%%%%%%%%%%%%%%%%%%%%%%%%%%%%%%%%%%%%%%%%%%%%%%%%%%
%\clearpage
\vspace{20pt}
\begin{appendices}
\section{$\mathbf{T^\ast \N}$ as essentially a symplectic reduction.}
For a point $\sfx \in \mM_{\Lambda,d}$ defined by non-split extension classes,
it follows from remark \ref{rem-projectivize}
and \eqref{moment-map-explicit} that
\begin{equation*}
	\ima(\pr_{\sfx}^\ast) = H^{-1}(0) \cap T^\ast_{\sfx} \mM_{\Lambda,d}.
\end{equation*}
By inverting $\pr^\ast_\sfx$ at each non-split extension $\sfx$,
one can define a holomorphic map $T^\ast\M^s \rightarrow T^\ast \N$.
As this map is $\bC^\ast$-equivariant,
it descends to a holomorphic map
$$F: T^\ast\N^s \rightarrow T^\ast \N.$$ 
The image of $F$ is the complement of the subset
$$ \{ \zeta \in T^\ast \N \mid k(\zeta) = 0 
\text{ or corresponds to a section having zeroes of order } > 1 \} $$
where $k(\zeta)$ is defined in \eqref{ses-cotangent-N}.
The image of $F$ is open dense in $T^\ast \N$.

\begin{proposition}\label{prop-symplectic-reduction}
$F$ is a symplectomorphism onto its image.
\end{proposition}
\begin{proof}
Clearly $F$ is a holomorphic diffeomorphism onto its image, so it suffices to show that the pull-back 
of the canonical symplectic form on $T^\ast \N$ 
along $F': T^\ast\M^s \rightarrow T^\ast \N$ 
coincides with $\wom \mid_{T^\ast\M^s}$.
We show this by first equipping local coordinates on $T^\ast\M^s$ and $T^\ast\N$.
Consider an open set $\mathcal{U} \subset T^\ast \M$ with local Darboux coordinates $(\bm{\lambda}, \by; \bkap, \mathbf{k})$
such that $x_1 \neq 0$.
Then $\mathcal{U}_1 = \mathcal{U} \cap T^\ast\M^s$
defines an open neighborhood on $T^\ast\M^s$.
We then can use
\begin{align*}
\mathcal{U}_1 \ni \xi \mapsto 
&(\bm{\lambda}(\xi), \by(\xi); \bkap(\xi), \mathbf{k'}(\xi) ),
&\mathbf{k'} = (k_2, ..., k_N)
\end{align*}
as coordinates on $\mathcal{U}_1$.
The inclusion $\mathcal{U}_1 \hookrightarrow T^\ast \M$
is defined by supplementing these coordinates with 
\begin{align*}
k_1 ((\bm{\lambda}, \by; \bkap, \mathbf{k'}))  = -x_1^{-1} \sum_{r=2}^N x_r k_r.
\end{align*}
In these coordinates,
\begin{equation*}
	\wom\mid_{T^\ast\M^s} = (-x_1^{-1} dx_1) \wedge \sum_{r=2}^N d (x_r k_r) 
	+ \sum_{r=2}^N dx_r \wedge dk_r 
	+ \sum_{i=1}^g d\lambda_i \wedge d\kappa_i.
\end{equation*} 
Consider the image of $\mathcal{U}$ 
via the composition $T^\ast\M' \overset{\Pi}{\rightarrow} \M' \overset{\pr}{\rightarrow} \N$.
Let us use $(\bm{\lambda}, \bm{y})$ 
as coordinates on this image,
where if $\xi$ has coordinates $\by' = (x_2, ..., x_N)$
then its image has coordinates $\bm{y} = (y_2, \dots, y_N) = x_1^{-1}\by'$.
Let $\check{\bm{y}} = (\check{y}_2, ..., \check{y}_N)$
be the canonical conjugate coordinates on the fibers of $T^\ast \N$.
In these coordinates,
\begin{align*}
F': T^\ast\M^s &\rightarrow T^\ast \N \\
(\bm{\lambda}, \by, \bkap, \mathbf{k})' &\mapsto
\left( \lambda, x_1^{-1} \by', \bkap, x_1 \mathbf{k}' \right).
\end{align*}
One can now check that indeed
the pulling-back of the canonical symplectic form $\sum_{r=2}^N dy_r \wedge d\check{y}_r + \sum_{i=1}^g d\lambda_i \wedge d\kappa_i$ on $T^\ast\N$
coincides with the restriction of $\wom$ to ${T^\ast\M^s}$.
\end{proof}

\section{Constructing differentials using the prime form}

The goal of this appendix is to describe the construction of meromorphic differentials on $X$ with the help of  the prime form,
following \cite{VV,HS}. We shall freely 
use standard background on the 
theory of Riemann surfaces which can be found in references 
like \cite{Fay} or lecture notes
\cite{Bo}, for example.

The basic building block is a function called prime form, defined as
\begin{equation}
E(z,w)=\frac{\theta(A(z)-A(w)+\Delta)}{\sqrt{\omega_{\Delta}(z)}\sqrt{\omega_{\Delta}(z)}},
\end{equation}
where $A\equiv A_{z_0}$ is the Abel map with $k$-th component 
$A(z)_k=\int_{z_0}^z\omega_k$, 
$\Delta$ is an odd theta-characteristic, $\theta$
is the Riemann theta function, and $\omega_{\Delta}$ is 
the holomorphic differential 
\[ \omega_{\Delta}=\sum_{\ell=1}^g\pa_i\theta_{0}(\Delta)\omega_i.
\]
We mainly need the properties that
$E(z,w)$ transforms as a $(-\frac{1}{2},-\frac{1}{2})$-form under changes of local
coordinates, that $E(z,w)$ vanishes {\it only} for $z=w$, and that
\begin{equation}\label{E-period}
(\mu_{\be_k}^zE)(z,w)=e^{-\frac{1}{2}B_{kk}-A(z)_k+A(w)_k-\Delta_k}E(z,w),
\end{equation}
where $\mu_{\be_k}^z E$ is the function 
obtained by analytic continuation of 
$E$ in the variable $z$
along the cycle $\be_k$ which is an
element of a canonical homology basis,
and
$B_{kl}$ are the elements of the period matrix.

We shall also use the multi-valued 
$g/2$-differential $\si$ defined 
up to a constant by the formula
\begin{equation}
\frac{\si(z)}{\si(w)}=\frac{\theta(A(z)-\sum_{\ell=1}^gA(\rho_\ell)+K)}{\theta(A(w)-\sum_{\ell=1}^gA(\rho_\ell)+K)}
\prod_{\ell=1}^g\frac{E(w,\rho_{\ell})}
{E(z,\rho_{\ell})},
\end{equation}
where $\sum_{\ell=1}^g\rho_
\ell$ is a generic effective divisor of degree
$g$, and $K$ is the vector of Riemann constants, 
\[K=-A(D_0-(g-1)P_0),
\]
with $P_0$ being a base point, and $D_0$
being the divisor of the holomorphic spin bundle with zero theta characteristic.
Note that $\si$ has neither poles nor 
zeros on $X$, and that it 
satisfies
\begin{equation}\label{si-period}
(\mu_{\be_k}^z\si)(z)=
e^{\frac{1}{2}(g-1)B_{kk}+A(D_0-(g-1)z)_k}\si(z).
\end{equation}
In the main text we consider combinations 
of the form 
\begin{equation}
c(z)=\frac{\prod_{k=1}^{M+2g-2}E(z,u_k)}{\prod_{r=1}^{M}E(z,v_r)}
\prod_{\ell=1}^{g}\frac{E(z,q_\ell)}{E(z,q_\ell')}\,\big(\si(z)\big)^2,
    \end{equation}
with $\bm{u}=\sum_{k=1}^{M+2g-2}u_k$,
$\bm{v}=\sum_{r=1}^{M}v_r$,
$\bm{q}=\sum_{\ell=1}^{g}q_\ell$
and $\bm{q}'=\sum_{\ell=1}^{g}q_{\ell}'$
satisfying 
\begin{equation}\label{Abel-constr}
A_{z_0}(\bm{u}-\bm{v})+A_{z_0}(\bm{q}-\bm{q}')
-2A(D_0-(g-1)z_0)=0.
\end{equation}
Note that the left side of \rf{Abel-constr} is independent of $z_0$.
It  follows from \rf{E-period}, 
\rf{si-period} and \rf{Abel-constr}
that $c(z)$ is a single-valued
one-form having poles only 
at $\bm{v}$ and $\bm{q}'$,
and vanishing at $\bm{u}$ and $\bm{q}$.
\end{appendices}

%%%%%%%%%%%%%%%%%%%%%%%%%%%
%\clearpage
%\titleformat{\section}{\large\bfseries\boldmath}{}{0pt}{\quad\\ \large}
%\title{my title}

%\bibliographystyle{JHEP_TD}
%\bibliography{AGT-Lang}

\end{document}